\documentclass[12pt,reqno]{amsart}

\usepackage{amsthm,amsmath,amsfonts,amssymb}
\usepackage{mathrsfs}
\usepackage{bm}
\usepackage[numbers]{natbib}
\usepackage{here,comment}
\usepackage{ascmac}
\usepackage{fancybox}
\usepackage{graphicx}
\usepackage{xcolor}
\usepackage{tabularx}
\usepackage{booktabs}

\usepackage{amsaddr}
\usepackage{fancyhdr}
\usepackage{fullpage}
\usepackage{newtxtext,newtxmath}

\newtheorem{theorem}{Theorem}

\newtheorem{lemma}{Lemma}
\newtheorem{proposition}{Proposition}
\newtheorem{assumption}{Assumption}
\theoremstyle{definition}
\newtheorem{definition}{Definition}
\newtheorem{remark}{Remark}
\newtheorem{example}{Example}

\newcommand{\R}{\mathbb{R}}
\newcommand{\N}{\mathbb{N}}

\newcommand{\mI}{\mathcal{I}}

\newcommand{\mB}{\mathcal{B}}

\newcommand{\mD}{\mathcal{D}}

\newcommand{\mS}{\mathcal{S}}

\newcommand{\mP}{\mathcal{P}}

\newcommand{\mN}{\mathcal{N}}
\newcommand{\mX}{\mathcal{X}}

\newcommand{\Ep}{E}

\newcommand{\KL}{\mathrm{KL}}
\newcommand{\KV}{\mathrm{KV}}

\renewcommand{\hat}{\widehat}
\renewcommand{\tilde}{\widetilde}
\newcommand{\pr}{\mathbb{P}}
\newcommand{\argmin}{\operatornamewithlimits{argmin}}

\newcommand{\mone}{\textbf{1}}

\usepackage{mathtools}
\mathtoolsset{showonlyrefs=true}

\title{Bayesian Analysis for Over-parameterized Linear Model \\ via Effective Spectra}

\author{Tomoya Wakayama$^\dagger$ \and Masaaki Imaizumi$^{\dagger,\ddagger}$}
\address{$^\dagger$RIKEN Center for Advanced Intelligence Project, $^\ddagger$The University of Tokyo}

\date{\today, \textit{Contact}: \textit{tomoya.wakayama@riken.jp, imaizumi@g.ecc.u-tokyo.ac.jp}}

\allowdisplaybreaks

\begin{document}
\maketitle

\begin{abstract}

In high-dimensional Bayesian statistics, various methods have been developed, including prior distributions that induce parameter sparsity to handle many parameters. Yet these approaches often overlook the rich spectral structure of the covariate matrix, which can be crucial when true signals are not sparse. To address this gap, we introduce a data-adaptive Gaussian prior whose covariance is aligned with the leading eigenvectors of the sample covariance. This prior design targets the data’s intrinsic complexity rather than its ambient dimension by concentrating the parameter search along principal data directions. We establish contraction rates of the corresponding posterior distribution, which reveal how the mass in the spectrum affects the prediction error bounds. Furthermore, we derive a truncated Gaussian approximation to the posterior (i.e., a Bernstein–von Mises–type result), which allows for uncertainty quantification with a reduced computational burden. Our findings demonstrate that Bayesian methods leveraging spectral information of the data are effective for estimation in non-sparse, high-dimensional settings.
\end{abstract}

\section{Introduction}

\subsection{Overview}

We consider an over-parameterized linear regression problem. Suppose we observe $n$ independent and identically distributed (i.i.d.) pairs $(y_1, \bm{x}_1),\ldots,(y_n, \bm{x}_n) \in \R \times \R^p$ generated from the linear regression model:
\begin{equation}\label{eqn-model}
y_i = \bm{x}_i^{\top} \bm{\theta}^* + \varepsilon_i, \quad (i=1,\ldots,n)
\end{equation}
where $\bm{\theta}^* \in \R^p$ is the true regression parameter vector, $\varepsilon_i\sim N(0, (\sigma^*)^2)$ are i.i.d. noise terms independent of the covariates $\bm{x}_i$, and $(\sigma^*)^2 > 0$ represents the true noise variance. Let $P^*$ denote the true data-generating distribution of $(\bm{x}_i,y_i)$ under model \eqref{eqn-model} with true parameters $(\bm{\theta}^*, (\sigma^*)^2)$. Here, the dimension $p$ is allowed to be much larger than the sample size $n$, that is, $p \gg n$. Crucially, we allow the true parameter $\bm{\theta}^*$ to be non-sparse, meaning many or all of its components may be non-zero and relevant to the response $y_i$. In this study, we investigate a Bayesian method for model \eqref{eqn-model} by constructing prior distributions for the unknown parameters and then examining the associated posterior distribution under certain assumptions on the covariates.

In over-parameterized Bayesian statistics, numerous prior distributions that leverage sparsity have been developed to effectively handle a large number $p$ of parameters. Specifically, starting with the double exponential prior for the Bayesian Lasso \citep{park2008bayesian}, many prior distributions for the linear regression problem \citep[e.g.,][]{carvalho2010horseshoe, rovckova2014emvs} have been proposed to extract information from an excess number of covariates. These methods excel at deriving insights from large datasets and determining which parameter components should be zero or non-zero. Although these approaches have demonstrated interpretability and theoretical underpinnings \citep{castillo2015bayesian,ning2020bayesian}, they require the stringent assumption that most regression parameters are zero or near zero.

In contrast, within a frequentist framework, several studies have focused on non-sparse high-dimensional models, commonly known as over-parameterized models, motivated by the success of large-scale models such as deep learning \citep{lecun2015deep}.
These studies often analyse estimators through the lens of the spectral properties of the data, that is, the eigenvalues and eigenvectors of the covariance matrix of the covariates. For instance, \cite{dobriban2018high} and \cite{hastie2022surprises} studied the ridge estimator and minimum norm interpolator of a linear regression model and derived precise error characterizations in the high-dimensional limit as $p \to \infty$. \cite{bartlett2020benign} and \cite{tsigler2020benign} further investigated the estimators for linear regression, uncovering the ``benign overfitting'' phenomenon: the prediction risk can converge to zero even when $p\gg n$ and the model perfectly fits noisy training data, with the risk governed by spectral properties summarized in an ``effective rank'' rather than the ambient dimension $p$. These results spurred a significant amount of subsequent research on non-sparse high-dimensional statistics.
However, traditional high-dimensional Bayesian methods, primarily designed for sparsity, are not directly equipped to handle such non-sparse, over-parameterized scenarios or leverage spectral information in the same manner. Therefore, developing Bayesian methods tailored to the non-sparse scenarios, especially by introducing suitable prior distributions, is crucial.

In this work, we propose and analyse a Bayesian approach specifically designed for potentially non-sparse, over-parameterized linear regression. We employ a $p$-dimensional centred Gaussian prior for $\bm{\theta}$, whose covariance structure is determined by a low-rank approximation of the empirical covariance matrix of the covariates. Specifically, the prior mass concentrates on a data-dependent low-dimensional subspace (low-rankness). The rank is assigned a prior distribution and inferred from the data within a hierarchical Bayesian framework. We then derive the contraction rates of the posterior distribution towards the true parameter, which characterizes the rate by an effective rank of the covariance matrix.
Furthermore, we establish a Bernstein–von Mises type theorem for this setting, providing an explicit truncated Gaussian approximation to the posterior distribution of $\bm{\theta}$~\citep{ghosal1999asymptotic,bontemps2011bernstein}.
For practical use, we develop a method to select hyperparameters of the prior distribution.

We highlight three key advantages of our Bayesian method: (i) It adapts to the intrinsic structure characterized by the spectrum rather than the nominal data dimension, leading to greater efficiency in non-sparse high-dimensional settings. Specifically, the upper bound on prediction error depends primarily on an effective rank, independent of the ambient dimension, as shown by our theoretical results. (ii) It relaxes certain data requirements imposed in existing studies. In particular, our framework removes the sub-Gaussian covariate assumption required by previous over-parameterized regression theories. (iii) It provides realistic and effective uncertainty quantification in high-dimensional, non-sparse setups. The Bernstein–von Mises approximation can be obtained directly without onerous Bayesian computations, which is especially advantageous in high-dimensional contexts.

We discuss the connection of the proposed method to prior works. The prior distribution's anisotropy has been empirically demonstrated to improve prediction performance in complicated over-parameterized models~\citep{zhao2019adaptive,louizos2017multiplicative}.
Furthermore, the maximum a posteriori (MAP) estimator derived from our anisotropic prior distribution approximates the Bayes optimal two-layer linear neural network estimator, aside from truncation, in some cases, especially when the covariates are aligned with the true parameters~\citep{suzuki2024optimal}.
Low-rankness is important for stabilizing high-dimensional covariance matrices, ensuring numerical robustness. Moreover, as \cite{castillo2024bayesian} states ``one needs to take structural assumptions \dots when proposing a prior distribution,'' our approach explicitly incorporates the effective dimension as a structural assumption when constructing a prior, in line with benign overfitting~\citep{bartlett2020benign, tsigler2020benign}. Also, we permit the rank to increase with the sample size because selecting an excessive number of principal components can enhance predictive performance~\citep{Xu2019pca}, and we determine it adaptively via Bayesian updating.

\subsection{Related Works} \label{sec:related_work}

The literature on Bayesian methods for linear regression models is abundant. For high-dimensional problems, prior distributions considering the sparsity structure have been developed, such as the Laplace \citep{park2008bayesian}, horseshoe \citep{carvalho2010horseshoe}, spike-and-slab \citep{rovckova2014emvs,nie2023bayesian}, normal-gamma \citep{brown2010inference}, double-Pareto \citep{armagan2013posterior}, Dirichlet--Laplace (DL) \citep{bhattacharya2015dirichlet}, horseshoe+ \citep{bhadra2017horseshoe+}, and triple gamma priors \citep{cadonna2020triple}.

Regarding the asymptotic behaviour of the posterior distribution, \cite{ghosal1997normal,ghosal1999asymptotic,ghosal2000asymptotic} examined an approximation of the posterior distribution using a Gaussian distribution in various situations, including high-dimensional settings (not over-parameterized settings).
\cite{ghosal2000convergence} provided a general framework and established widely used sufficient conditions for proving posterior contraction, significantly advancing Bayesian asymptotic theory.
\cite{castillo2012needles,belitser2020needles} proved posterior concentration and variable selection properties for certain point-mass priors.
\cite{castillo2015bayesian,martin2017empirical,belitser2020empirical,belitser2020needles,wu2023statistical} achieved posterior concentration and variable selection in high-dimensional linear models.
\cite{armagan2013generalized} demonstrated posterior consistency with several shrinkage priors.
\cite{van2014horseshoe} showed posterior concentration with the horseshoe prior. \cite{bhattacharya2015dirichlet} showed a similar result using the DL prior.
\cite{bai2018high,song2022nearly} considered a general class of continuous shrinkage priors.
\cite{ning2020bayesian} considered variable selection with an unknown covariance matrix and established posterior consistency results.
\cite{zhang2014confidence,yang2019posterior} studied the efficiency and Gaussian approximation of one coordinate of the high-dimensional parameters.
\cite{banerjee2021bayesian} presented a detailed review of these references.

There has been a rapidly growing body of theoretical work on the frequentist analysis of over-parameterized models, specifically concerning benign overfitting.
A theoretical study closely related to ours is \cite{bartlett2020benign}, which showed that the excess risk converges to zero in an over-parameterized linear model.
This phenomenon has been studied in other situations such as regularized regression~\citep{tsigler2020benign,koehler2021uniform,wang2021tight,li2021minimum,chatterji2022foolish}, logistic regression~\citep{chatterji2021finite,muthukumar2021classification}, and kernel regression~\citep{liang2020just,liang2020multiple}, as well as regressions for dependent data \citep{nakakita2022benign,tsuda2023benign}.

\subsection{Notation}
For $q\in \mathbb{N}$ and any vector $\bm{v}$, $\|\bm{v}\|_q$ denotes the $\ell_q$-norm of $\bm{v}$. 
For any vector $\bm{v}$ and symmetric matrix of the same dimension $A$, we define $\|\bm{v}\|_{A} = \sqrt{\bm{v}^{\top}A\bm{v}}$. For any matrix $A$, $\|A\|_{\mathrm{op}}$ denotes its operator norm induced by the Euclidean norm, defined as $\|A\|_{\mathrm{op}} = \sup_{\|\bm{v}\|_2 = 1} \|A\bm{v}\|_2$. For sequences $\{a_n\}_{n \in \N}$ and $\{b_n\}_{n \in \N}$, $a_n \gtrsim b_n$ (or $b_n \lesssim a_n$) means the existence of a constant $c > 0$ such that $a_n \geq c b_n$ for all $n \geq \overline{n}$ with some finite $\overline{n} \in \N$.
$a_{n}=o(b_{n})$ indicates that $|a_{n}/b_{n}|\to 0$ as $n \to \infty$, $a_{n}=\omega(b_{n})$ means $|a_{n}/b_{n}|\to\infty$ as $n \to \infty$, and $a_{n}=O(b_{n})$ represents that $|a_{n}| \lesssim |b_{n}|$. For any square matrix $A$, let $A^{-1}$ be the inverse of $A$ (if it exists), $A^{\dagger}$ be the Moore--Penrose pseudoinverse, and $\mathrm{tr}(A)$ be the trace of $A$. For a pseudo-metric space $(S,d)$ and a positive value $\delta$, $\mN(\delta, S,d)$ denotes the covering number, that is, the minimum number of closed balls of radius $\delta$ in terms of $d$ that cover $S$. $\mB(X)$ denotes the Borel $\sigma$-field on a topological space $X$.
For any logical statement $E$, $\mone\{E\}$ denotes the indicator function, which equals $1$ if $E$ is true and $0$ otherwise. For probability measures $\Pi_1$ and $\Pi_2$ on the same measurable space $(\mX,\mB)$, the total variation distance is given by $\|\Pi_1-\Pi_2\|_{\mathrm{TV}}:= \sup_{B\in\mB}| \Pi_1(B)-\Pi_2(B)|$, and the Kullback--Leibler (KL) divergence and variation between them are respectively defined as $\KL(\Pi_1, \Pi_2)=\int_{\mX} \log(\mathrm{d}\Pi_1/\mathrm{d}\Pi_2) \mathrm{d}\Pi_1$ and $\KV(\Pi_1, \Pi_2)=\int_{\mX} \{\log(\mathrm{d}\Pi_1/\mathrm{d}\Pi_2)-\KL(\Pi_1, \Pi_2)\}^2\mathrm{d}\Pi_1$ if $\Pi_1$ is absolutely continuous with respect to $\Pi_2$; otherwise, both are $\infty$. 

\section{Setup and Proposed Method} \label{sec:pri}

\subsection{Setup}
We consider a linear regression model~\eqref{eqn-model} that generates independent and identically distributed samples $\mD = \{(y_i, \bm{x}_i)\}_{i=1}^n$, and we perform predictions based on the samples in a Bayesian manner. Suppose that $\bm{x}_i$ is a $p$-dimensional random vector with a non-atomic distribution and that both $\bm{x}_i$ and $y_i$ have zero mean. Let $\Sigma := \Ep[\bm{x}_i\bm{x}_i^{\top}]$ be the expected covariance matrix of the covariate $\bm{x}_i$. We adopt a setting where the dimension $p$ increases and the covariance matrix $\Sigma$ changes as the sample size $n$ increases.

We measure the distance of a parameter vector $\bm{\theta}$ from the true parameter $\bm{\theta}^*$ based on the norm $\|\cdot\|_\Sigma$ weighted by the covariance matrix $\Sigma$. Assessing this norm has several advantages: it is consistent with the predictive risk, that is, we have 
\begin{equation}
    \|\bm{\theta} - \bm{\theta}^*\|_\Sigma^2 = \Ep\bigl[\bigl\{\bm{x}^{\top}({\bm{\theta}}-\bm{\theta}^*)\bigl\} ^2\bigl], \label{def:sigma_norm}
\end{equation}
where $\bm{x}$ in the expectation follows the marginal distribution of $\bm{x}_i$ in the regression model \eqref{eqn-model}. 

We consider the spectral decomposition of $\Sigma$ as follows:
\begin{equation}
    \Sigma = \sum_{j=1}^{p}\lambda_j\bm{v}_j\bm{v}_j^{\top}, 
\end{equation}
where $\{\lambda_j\}_{j=1}^p$ are the eigenvalues sorted in descending order as $\lambda_1 \ge \lambda_2 \ge \ldots \ge \lambda_p \geq 0$, and $\{\bm{v}_j\}_{j=1}^p$ are the corresponding $p$-dimensional eigenvectors. $\{\lambda_j\}_{j=1}^p$ and $\{\bm{v}_j\}_{j=1}^p$ depend on the dimension $p$, and thus implicitly on the sample size $n$ under our setting.

\subsection{Preparation: Empirical Estimation of Covariance Matrix} \label{sec:empirical_estimation_Sigma}

We prepare an empirical estimate of the covariance matrix $\Sigma$ for the prior design (developed in Section~\ref{sec: prior_design}).
First, we split the dataset $\mD$ into $\mD_1$ and $\mD_2$. 
The split ratio is arbitrary as long as it does not depend on $n$.
Without loss of generality, we assume that $n$ is even and define $\mD_1=\{y_i, \bm{x}_i\}_{i=1}^{n/2}$ and $\mD_2 = \{y_i, \bm{x}_i\}_{i=n/2 + 1}^{n}$. $\mD_1$ is used to construct the prior distribution, and $\mD_2$ is used to form the likelihood hereafter. This sample splitting is important for simplifying the theoretical analysis.

We estimate the eigenvalues and eigenvectors using $\mD_1$. We consider the empirical analog $\hat{\Sigma} := (2/n) \sum_{i=1}^{n/2}(\bm{x}_i - \bar{\bm{x}})(\bm{x}_i - \bar{\bm{x}})^{\top}$ of $\Sigma$ using $\mD_1$, where $\bar{\bm{x}} := (2/n)\sum_{i=1}^{n/2} \bm{x}_i$ is the sample mean of the covariates in $\mD_1$. We also define its spectral decomposition: $\hat{\Sigma} = \sum_{j=1}^{p}\hat{\lambda}_j\hat{\bm{v}}_j\hat{\bm{v}}_j^{\top}$, where $\{\hat{\lambda}_j\}_{j=1}^p$ and $\{\hat{\bm{v}}_j\}_{j=1}^p$ denote the eigenvalues $\hat{\lambda}_1 \geq \hat{\lambda}_2 \geq \cdots \hat{\lambda}_p \geq 0$ and eigenvectors of $\hat{\Sigma}$. Since $p > n/2$, we have $\hat{\lambda}_j=0$ for $j=n/2+1,\ldots,p$.
Under certain conditions, the empirical eigenvalue $\hat{\lambda}_j$ and eigenvector $\hat{\bm{v}}_j$ converge in probability to $\lambda_j$ and $\bm{v}_j$, respectively \citep[e.g.,][]{loukas2017close}.

We further define low-rank approximations of $\Sigma$ and $\hat{\Sigma}$ with a truncation level $k \in \{1,2,\ldots,p\}$, namely, we define the $k$-rank approximations as
\begin{equation}
\Sigma_{1:k}:=\sum_{j=1}^{k}\lambda_j\bm{v}_j\bm{v}_j^{\top} \quad \mbox{and}\quad  \hat{\Sigma}_{1:k} := \sum_{j=1}^{k}\hat{\lambda}_j\hat{\bm{v}}_j\hat{\bm{v}}_j^{\top}.
\end{equation}
We denote $\Sigma_{k+1:p} := \Sigma - \Sigma_{1:k}$ and $\hat{\Sigma}_{k+1:p} := \hat{\Sigma} - \hat{\Sigma}_{1:k}$.
Many studies have studied the low-rank nature of matrices in the Bayesian context \citep[][]{alquier2013bayesian,trippe2019lr}.  

\subsection{Prior Design}\label{sec: prior_design}
We construct a prior distribution based on spectra (eigenvalues) of the covariance matrix of the covariates. To this end, we render the prior distribution sensitive solely to the $k$-dimensional principal eigenspace of the empirical covariance matrix of $\{\bm{x}_i\}_{i\in\mD_1}$, which is referred to as \textit{effective spectra}. Intuitively, the prior distribution captures large eigenvalues, which typically represent the primary information of size $O(n)$, from the dataset. 
We also determine the level $k$ via posterior inference using hierarchical priors. 

Our method transfers the principle of benign overfitting \citep{bartlett2020benign,tsigler2020benign}, which achieves consistent estimation in high-dimensional settings using spectral information, to the Bayesian context. The benign overfitting theory shows that a (non-Bayesian) minimum norm interpolator extracts an $n$-dimensional subspace of the $p$-dimensional parameter space to achieve consistency. By contrast, our prior distribution seeks a $k(<n/2)$-dimensional subspace, which could be more critical.

\textbf{Prior for Regression Parameter with Effective Spectra:} 
We develop a prior distribution $\Pi_{\vartheta}(\cdot)$ on the $k$-dimensional subspace $S_{\mD_1,k} := \mathrm{span}\{\hat{\bm{v}}_1, \hat{\bm{v}}_2,\ldots, \hat{\bm{v}}_k \} \subset \mathbb{R}^p$ using a truncated and centred Gaussian distribution. In this subspace, every $\bm{\theta}\in S_{\mD_1,k}$ is uniquely represented as  
\begin{equation*}
    \bm{\theta} = \sum_{i=1}^k \beta_i\hat{\bm{v}}_i,
\end{equation*}
with a coefficient vector $\bm{\beta}_{1:k}=(\beta_1,\dots,\beta_k)^{\top}\in \R^k$.
Then, we define a density function of the prior distribution on the coefficient vector $\bm{\beta}_{1:k}$ as
\begin{equation*}
    \pi_{\mathfrak{b}}(\bm{\beta}_{1:k}\mid \mD_1,k)
    =\frac{\exp\left\{-\bm{\beta}_{1:k}^\top \tilde{\Sigma}_{1:k}^{-1}\bm{\beta}_{1:k}\right\}\mone\left\{  \|\bm{\beta}_{1:k}\|_{\tilde{\Sigma}_{1:k}}\le R \right\}}{\displaystyle \int_{\mathbb{R}^k}\exp\left\{-\bm{\beta}_{1:k}'^\top \tilde{\Sigma}_{1:k}^{-1} \bm{\beta}_{1:k}'\right\}\mone\left\{\|\bm{\beta}_{1:k}'\|_{\tilde{\Sigma}_{1:k}} \le R \right\}\mathrm{d}\bm{\beta}_{1:k}'},
\end{equation*}
where $\tilde{\Sigma}_{1:k}=\mathrm{diag}(\hat{\lambda}_1,\dots,\hat{\lambda}_k)$ and $R > 0$ denotes the pre-specified radius of the support, which is a sufficiently large constant. We set the matrix $\tilde{\Sigma}_{1:k}$ as the covariance matrix of this prior distribution, and it has the following roles: (i) it approximately utilizes the spectrum of $\Sigma$ because $\tilde{\Sigma}_{1:k}$ comprises the estimated eigenvalues, which can be regarded as an empirical Bayesian method similar to the $g$-prior \citep{zellner1986assessing}, and (ii) this prior distribution is supported on the $k$-dimensional principal eigenspace of $\hat{\Sigma}$, which acts as dimensionality extraction. Informally, for the prior distribution of $\bm{\theta}$, we consider an anisotropic Gaussian distribution that favours the principal eigen-directions of $\bm{x}$, $N(\bm{0},\hat{\Sigma}_{1:k})$, with rank and support constraints. Remark that while our method is sparse in the frequency domain, the true parameters themselves are not assumed to be sparse, even in this sense. 

Next, we explicitly specify the prior distribution for $\bm{\theta}$. Let $T_{\mD_1,k}:(\R^k,\mB(\R^k))\to (S_{\mD_1,k}, \mB(S_{\mD_1,k}))$ be a bijective linear map as $T_{\mD_1,k}(\bm{\beta}_{1:k}) = \sum_{i=1}^k \beta_i\hat{\bm{v}}_i$ and $\Pi_{\mathfrak{b}}$ be a probability measure defined as $\Pi_{\mathfrak{b}}(B\mid \mD_1,k)=\int_B \pi_{\mathfrak{b}}(\bm{\beta}_{1:k}\mid \mD_1,k)\mathrm{d}\bm{\beta}_{1:k}$ for any Borel set $B\in \mB(\R^k)$. The prior distribution for $\bm{\theta}$ is defined as a pushforward measure $\Pi_{\vartheta}(B\mid \mD_1,k) = \Pi_{\mathfrak{b}}( T_{\mD_1,k}^{-1}(B\cap S_{\mD_1,k}) \mid \mD_1,k)$ for any Borel set $B \in \mB(\R^p)$; namely,
\begin{equation}\label{prior:theta}
        \Pi_{\vartheta}(B\mid \mD_1,k) = \frac{\int_{T_{\mD_1,k}^{-1}(B\cap S_{\mD_1,k})} \exp\left\{-\bm{\beta}_{1:k}^\top \tilde{\Sigma}_{1:k}^{-1}\bm{\beta}_{1:k}\right\} \mone\left\{\|\bm{\beta}_{1:k}\|_{\tilde{\Sigma}_{1:k}} \le R \right\}\mathrm{d}\bm{\beta}_{1:k} }{\int_{ \R^k} \exp\left\{-\bm{\beta}_{1:k}'^\top \tilde{\Sigma}_{1:k}^{-1}\bm{\beta}_{1:k}'\right\} \mone\left\{\|\bm{\beta}_{1:k}'\|_{\tilde{\Sigma}_{1:k}} \le R \right\} \mathrm{d}\bm{\beta}_{1:k}' }.
\end{equation}
Since $\Pi_{\vartheta}(\cdot\mid \mD_1,k)$ is the pushforward measure of $\Pi_{\mathfrak{b}}(\cdot\mid \mD_1,k)$ under a bijective measurable map, it preserves non-negativity, countable additivity, and normalization, and is well-defined as a probability measure. 
If we set $R=\infty$, the prior distribution is a degenerate multivariate Gaussian distribution in $\R^p$ whose support is confined to the $k$-dimensional subspace.

\begin{figure}[t]
  \begin{center}
  \includegraphics[width=0.8\textwidth]{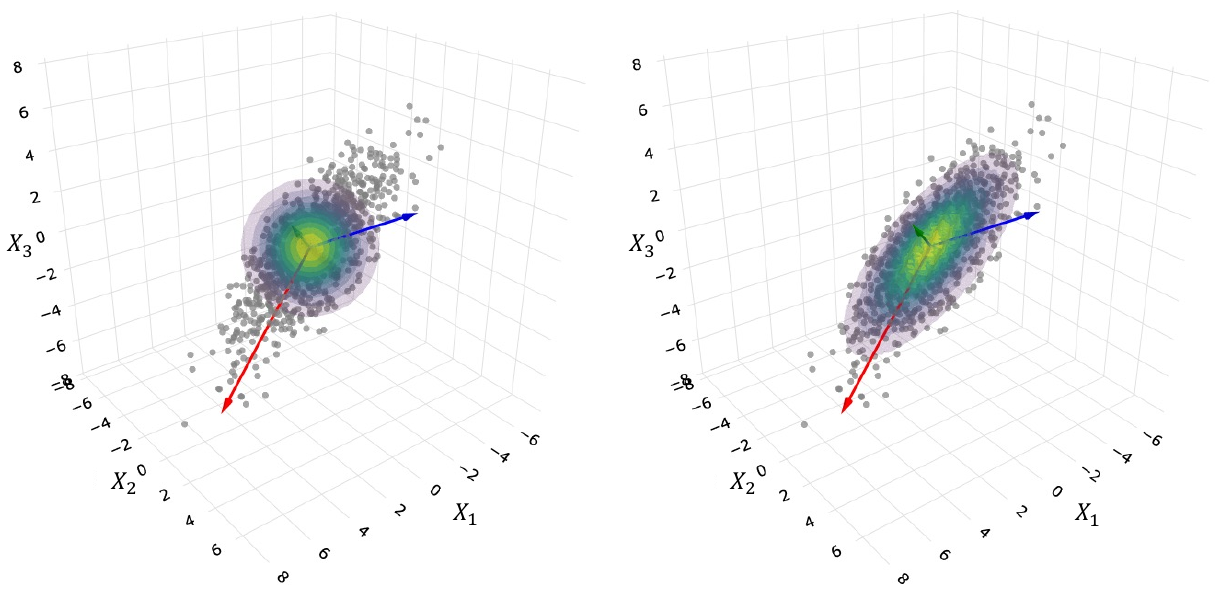}
  \end{center}
  \caption{Left: 3D isotropic normal distribution. Right: Proposed distribution. Red, blue, and green arrows represent the first to third principal components of the black data points in 3D space. The proposed distribution assigns weights along the principal component directions proportional to the eigenvalues. \label{fig:prior}}
\end{figure}

\textbf{Prior for Truncation Level:}
Second, we develop a prior distribution for the truncation level $k$ of $\hat{\Sigma}_{1:k}$, which regulates the information in the prior for $\bm{\theta}$. Given a pair $(L_{\kappa}, U_{\kappa})$ such that $1\leq L_{\kappa} \leq U_{\kappa} < n/2$ and a measurable function $f:\N\to\R$, we define the prior distribution for $k$ as a categorical distribution over the discrete set $\{L_{\kappa},\ldots U_{\kappa}\}$, where the probability function is
\begin{equation} \label{def:prior_k}
\pi_{\kappa}(k) = \frac{e^{f(k)}}{\sum_{k=L_{\kappa}}^{U_{\kappa}} e^{f(k)}},\quad ( k = L_{\kappa},\ldots U_{\kappa}).
\end{equation}
Restricting $k$ to $\{L_{\kappa},\ldots U_{\kappa}\}$ is essential because neither an overabundance nor a lack of information will lead to successful predictions. We allow $L_{\kappa}$ and $U_{\kappa}$ to depend on $n$, with $L_{\kappa} \to \infty$ as $n \to \infty$.

\textbf{Prior for Noise Variance:}
We specify the prior distribution for the variance in the noise term $\varepsilon_i$; that is, we consider the inverse-Gaussian prior $\Pi_{\varsigma^2}(\cdot)$ with the following density function: 
\begin{align}
    \pi_{\varsigma^2}(\sigma^2) = \sqrt{\frac{\xi}{2 \pi \sigma^6}}\exp \left( \frac{- \xi (\sigma^2 - \eta)^2 }{2 \eta^2 \sigma^2} \right), \label{def:prior_sigma}
\end{align}
where 
$\eta>0$ and $\xi>0$ denote the mean and shape parameters, respectively, which could be assigned arbitrarily.
An important property of the inverse-Gaussian distribution is the rapidly decreasing tails on both sides, whereas the inverse gamma distribution, which is a conjugate prior distribution for the Gaussian likelihood, has a light lower tail and a heavy upper tail.
From a Bayesian perspective, this prior reflects the belief that the noise variance parameter does not have extremely large or small values. This prior distribution for the variance parameter of the Gaussian distribution is also employed in \cite{szabo2013empirical, ning2020bayesian}.

\textbf{Joint Prior Distribution for Parameters:}
We introduce a joint prior distribution $\Pi_{\vartheta,\varsigma^2}(\cdot)$ of $(\bm{\theta}, \sigma^2)$, which is useful in various analyses. It is defined as
\begin{equation}
    \Pi_{\vartheta,\varsigma^2}\left(  B \times B' \mid \mD_1 \right) \\
    = \sum_{k=L_{\kappa}}^{U_{\kappa}}\int_{T_{\mD_1,k}^{-1}(B\cap S_{\mD_1,k})} \int_{B'} \pi_{\mathfrak{b}}(\bm{\beta}_{1:k}\mid \mD_1,k) \pi_{\kappa}(k)\pi_{\varsigma^2}(\sigma^2)  \mathrm{d} \sigma^2 \mathrm{d} \bm{\beta}_{1:k}, 
\end{equation}
for any Borel sets $B \in \mB(\R^p)$ and $B' \in \mB(\R_+)$.

\subsection{Posterior} \label{sec:posterior}

We present a posterior distribution corresponding to our prior design above.
Under the regression model~\eqref{eqn-model}, the likelihood is $\prod_{(y_i, \bm{x}_i) \in \mD_2}\phi(y_i;\bm{x}_i^{\top}\bm{\theta},\sigma^2)$, where $\phi(a;b,c)$ denotes the probability density of the Gaussian distribution evaluated at $a$ with mean $b$ and variance $c$.

We consider the posterior distribution associated with the prior distribution and the likelihood function. Its density function over the parameter space $\bigcup_{k=L_{\kappa}}^{U_{\kappa}}(\R^k \times \R_+ \times \{k\})$ is given by Bayes' theorem as
\begin{align}
    &\pi(\bm{\beta}_{1:k},\sigma^2,k\mid \mD) \\
    &= \frac{\pi_{\mathfrak{b}}(\bm{\beta}_{1:k}\mid \mD_1,k) \pi_{\kappa}(k)\pi_{\varsigma^2}(\sigma^2) {\prod_{(y_i, \bm{x}_i) \in \mD_2}\phi(y_i;\bm{x}_i^{\top}T_{\mD_1,k}(\bm{\beta}_{1:k}),\sigma^2 )}}{ \sum_{k'=L_{\kappa}}^{U_{\kappa}}\int_{\R^{k'}} \int_{\R_+} \pi_{\mathfrak{b}}(\bm{\beta}_{1:k'}' \mid \mD_1,k') \pi_{\kappa}(k')\pi_{\varsigma^2}(\sigma'^2) {\prod_{(y_i, \bm{x}_i) \in \mD_2}\phi(y_i;\bm{x}_i^{\top}T_{\mD_1,k'}(\bm{\beta}_{1:k'}'),\sigma'^2 )}  \mathrm{d} \sigma'^2 \mathrm{d} \bm{\beta}_{1:k'}'  } , \label{eq:density_posterior}
\end{align}
The first three factors in the numerator express the joint prior density of the parameters, and the product term is equivalent to the likelihood. We denote the posterior distribution of $\bm{\theta}$ as 
\begin{equation}
    \Pi \left( B \mid \mD \right) = \sum_{k=L_{\kappa}}^{U_{\kappa}}\int_{T_{\mD_1,k}^{-1}(B\cap S_{\mD_1,k})} \int_{\R_+} \pi(\bm{\beta}_{1:k},\sigma^2,k \mid \mD)  \mathrm{d} \sigma^2 \mathrm{d} \bm{\beta}_{1:k} \label{eq:definition_posterior}
\end{equation}
for any Borel set $B \in \mB(\R^p)$. The posterior distribution of $\sigma^2$ is defined in the same manner. The method for selecting hyperparameters $R, L_\kappa$, and $U_\kappa$ will be explained in Section~\ref{sec:hyperparameter}.

\begin{remark}[Cross-fitting]
To mitigate information loss due to data splitting, we can utilize cross-fitting \citep{chernozhukov2018double}. 
For point estimation, a simple approach is to swap $\mD_1$ and $\mD_2$, construct another posterior estimator, and average it with the original posterior estimator, which maintains consistency if each estimator does so individually. If one is interested in the posterior distribution, we can also achieve an integrated posterior via the Wasserstein barycenter \citep{Srivastava2018scalable} or Bayesian stacking \citep{10.1214/17-BA1091}. 
\end{remark}

\section{Theoretical Guarantee}\label{sec:res}

In this section, we examine the theoretical properties of the posterior inference based on the proposed prior distributions. To this end, we introduce some notions to describe the assumptions and the theoretical results.
Recall that we study the risk with respect to the norm $\|\cdot\|_\Sigma$ as in \eqref{def:sigma_norm}, based on the predictive risk and that $\Sigma$, $\{\lambda_i\}_i$ and the related quantities vary as $n$ increases.
Remark that our theoretical results do not require the sparsity of the true parameter.

\subsection{Baseline Case: sub-Gauss and trace-class assumptions}\label{subsec:baseline}
We begin with a statement of the theorem by considering the case in which $\bm{x}$ is sub-Gaussian and the covariance matrix is trace class; namely, the sum of its eigenvalues is finite. 
These settings are often employed in the field of high-dimensional statistics and over-parameterized regression theory \citep{koltchinskii2017concentration,bartlett2020benign,tsigler2020benign}.

Formally, we make the following two assumptions regarding the distribution of the covariates.
\begin{assumption} 
\label{ass:DGP}
The covariate vector $\bm{x}\in\mathbb{R}^p$ is generated from a centred sub-Gaussian distribution; that is, there exists a constant $b\in\mathbb{R}$ such that $\Ep\left[ \exp(\bm{c}^{\top} \bm{x}) \right] \le \exp(\|\bm{c}\|_2^2 b^2/2)$ holds for all $\bm{c}\in \mathbb{R}^p.$
\end{assumption}

\begin{assumption}
\label{ass:trace}
The covariance matrix $\Sigma$ is a trace-class operator; that is, we have $\limsup_{p\ge 1}\sum_{k=1}^{p} \lambda_k < \infty $.
\end{assumption}
These conditions are fairly standard and are often stipulated for the use of concentration inequalities \citep[e.g.,][]{li2021towards,cheng2022dimension}. We suppose they hold for every $n \in \N$, even if the distributions of $\bm{x}$ and $\Sigma$ depend on $n$. We relax these assumptions in Section~\ref{sec:relax}.

\begin{definition}[Effective Ranks (Def. 3 of~\cite{bartlett2020benign})]
For a covariance matrix $\Sigma \in \R^{p \times p}$, we define two types of effective ranks of $\Sigma$ as follows:
    \begin{equation*}
        r_k(\Sigma) = \frac{\sum_{i>k} \lambda_i }{\lambda_{k+1}},\quad \quad R_k(\Sigma) = \frac{(\sum_{i>k} \lambda_i)^2 }{\sum_{i>k} \lambda_i^2} . 
    \end{equation*}
\end{definition}
Both indicators measure the complexity of the matrix.
$r_k(\Sigma)$ represents the magnitude of the eigenvalues after the $k$-th one relative to the $(k+1)$-th eigenvalue, while $R_k(\Sigma)$ reflects the uniformity of the eigenvalue decay after the $k$-th one (a smaller value indicates less uniformity or faster decay).

\subsection{Posterior Contraction in Baseline Case} \label{sec:posterior_contraction}

We present a contraction theorem for the developed posterior distribution.
In preparation, we define an effective variance bound $V_n$ and an effective bias bound $B_n$ that capture the contributions of variance and bias components, respectively, to the contraction rate. These bounds depend on the spectra of $\Sigma$ via the eigenvalues $\{\lambda_i\}_i$ as well as the effective ranks $r_k(\Sigma)$ and $ R_k(\Sigma)$.
Also, the hyperparameters $(L_{\kappa}, U_{\kappa})$ and $\pi_{\kappa}$ affect the terms. Specifically, with $\mI := \{ i=1,\ldots,U_{\kappa} : n  ({\lambda_i \sum_{i>L_{\kappa}} \lambda_{i}^2})^{1/2} > ( \sum_{i>L_{\kappa}} \lambda_{i})^{3/2} \}$, we define the bounds for $\varepsilon > 0$ as follows:
\begin{align}
    V_n(\varepsilon) &:= \sum_{i\in\mI} \frac1n \log \left( \frac{n\sqrt{\lambda_i} \sqrt{\sum_{i>L_{\kappa}} \lambda_{i}^2}}{ ( \sum_{i>L_{\kappa}} \lambda_{i})^{3/2}  } \right) + \frac{U_{\kappa}+\log( n/2\varepsilon)}{n} +\frac{n \sum_{i>L_{\kappa}} \lambda_{i}^2 }{( \sum_{i>L_{\kappa}} \lambda_{i})^2},\label{eq:V_n} \\
    B_n &:= \frac{1}{n}\log\frac{1}{\pi_{\kappa}(L_{\kappa})} + \frac{\| \bm{\theta}^* \|_2^2}{n}\sum_{i=1}^{L_{\kappa}} \frac{1}{{\lambda}_i} + \frac{L_{\kappa}\log n}{n}. \label{eq:B_n}
\end{align}
They are upper bounds on the bias and variance of the estimation in the over-parameterization setting; \cite{tsigler2020benign} derives similar upper bounds on the bias and variance of the ridgeless estimator, and our definition has a correspondence with them. A detailed comparison will be given later.

Using the bounds, we obtain the following theorem on posterior contraction:
\begin{theorem}\label{thm:poscon}
    Consider the regression model~\eqref{eqn-model} and the posterior distributions of $\bm{\theta}$ and $\sigma^2$ with $R\le\infty$. Suppose that Assumptions~\ref{ass:DGP}~and~\ref{ass:trace} hold and $R$ satisfies $\| \bm{\theta}^* \|_{\Sigma} < R/2$ and $\| \bm{\theta}^* \|_2 < \infty$. For any sequence $\{\varepsilon_n\}_{n \in \N}$ satisfying that for any $\delta_1, \delta_2 > 0$ and $M > 0$, there exists $\overline{n}\in\N$ such that for any $n\ge \overline{n}$, $\varepsilon_n<\delta_1$, $n\varepsilon_n^2> M$, $\sqrt{(r_0(\Sigma) + 1) / n\lambda_{L_{\kappa}+1}^2} < \delta_2$, $n/\lambda_{L_{\kappa}} > r_{L_{\kappa}}(\Sigma) R_{L_{\kappa}}(\Sigma)$, ${\varepsilon}_n^2 > 4\| \bm{\theta}^* \|_2^2 \sum_{i>L_{\kappa}} {\lambda}_i$, and 
    \begin{align}
        \max\{ V_n(\varepsilon_n), B_n\} \le \varepsilon_n^2, \label{inequalities:eps_bias_variance}
    \end{align}
    we have the following as $n \to \infty$, for some constant $C>0$:
\begin{align}
    \Pi\Big(\{\bm{\theta}:\|\bm{\theta}-\bm{\theta}^*\|_{\Sigma}> C\varepsilon_n \}\mid \,\mD\Big) \overset{P^*}{\longrightarrow} 0. \label{eq:main1}
\end{align}
Additionally, for the posterior distribution of $\sigma^2$, we have, for some constant $C>0$:
\begin{align}    
\Pi\Big(\{\sigma^2:|\sigma^2-(\sigma^{*})^2|> C \varepsilon_n\} \mid \mD\Big) \overset{P^*}{\longrightarrow} 0. \label{eq:main2}
\end{align}
\end{theorem}

Regarding the interpretation of additional conditions in the statement, $n/\lambda_{L_{\kappa}} > r_{L_{\kappa}}(\Sigma) R_{L_{\kappa}}(\Sigma)$ ensures the eigenvalue decay is not too slow. ${\varepsilon}_n^2 > 4\| \bm{\theta}^* \|_2^2 \sum_{i>L_{\kappa}} {\lambda}_i$ implies that the contraction rate is hindered by the contribution of the tail eigenspace. Remark that under Assumption~\ref{ass:trace} and assuming $L_{\kappa} \to \infty$ as $n\to \infty$, $4\| \bm{\theta}^* \|_2^2 \sum_{i>L_{\kappa}} {\lambda}_i$ can converge to zero.

Theorem~\ref{thm:poscon} describes the contraction of the posterior distribution at the rate $\varepsilon_n$. The contraction rate $\varepsilon_n$ is mainly designed to bound the effective bias and variance bounds as in \eqref{inequalities:eps_bias_variance}. The convergences are driven by the whole dataset $\mD$ for calculating the posterior. The convergence \eqref{eq:main1} shows that the posterior distribution of the regression parameter concentrates in the region where the predictive risk vanishes.
This result is a Bayesian counterpart of vanilla benign overfitting \citep[e.g.,][]{bartlett2020benign,tsigler2020benign}. The convergence~\eqref{eq:main2} represents the posterior contraction of the error variance. 
Section~\ref{sec:proof-outline} discusses the technical points, and the appendix provides the full proof. Also, we provide examples of $\Sigma$ and our prior that satisfy the theorem's assumptions as well as the effective variance and bias bounds in the appendix, and present several important examples related to the generalized version of the theorem in Section~\ref{sec:proof-outline}.

\subsection{Advanced Case: Beyond sub-Gauss and trace-class assumptions}
\label{sec:relax}
We relax the requirements of sub-Gaussianity on the covariate $\bm{x}_i$ and the trace-class property on the covariance $\Sigma$. 
In particular, we introduce a new assumption and demonstrate posterior contraction based on it, instead of relying on Assumptions~\ref{ass:DGP}~and~\ref{ass:trace}. This extension is practically important, although most theories of over-parameterization assume (sub-)Gaussianity \citep{bartlett2020benign, koehler2021uniform}.

In this section, we primarily focus on the case where $\Sigma$ is not trace-class. More precisely, we consider $\Sigma$ depending on $n$, which satisfies, 
as $n \to \infty$,
\begin{align}
    \mathrm{tr}(\Sigma) \to \infty. \label{def:non-trace_class}
\end{align}
Although the results in this section can be easily extended to $\Sigma$ in the trace-class setting, we consider this non-trace-class case to keep the presentation simple.

First, we introduce assumptions regarding the estimation error of $\Sigma$, which simultaneously characterize both the distribution of the covariate $\bm{x}_i$ and the covariance structure $\Sigma$.
Recall that $\hat{\Sigma}$ denotes the empirical estimate of $\Sigma$, as described in Section \ref{sec:empirical_estimation_Sigma}.
\begin{assumption}\label{ass:emp}
Suppose that $\|\Sigma\|_{\mathrm{op}}$ is bounded. Additionally, there exist sequences of non-negative reals $\{\rho_n\}_{ n \in \N}$ and $\{\ell_n\}_{n \in \N}$ such that $\rho_n =o(1)$ and $\ell_n =o(1)$ hold as $n \to \infty$, and they satisfy the following for every sufficiently large $n \in \N$: 
      \begin{align}
          \pr \left( \frac{\| \hat{\Sigma} - \Sigma \|_{\mathrm{op}}}{\|\Sigma\|_{\mathrm{op}}}  \leq  \rho_n \right) \geq 1 - \ell_n.
      \end{align}
\end{assumption}
\noindent
Assumption~\ref{ass:emp} requires that the empirical covariance matrix $\hat{\Sigma}$ used in the prior distribution accurately reflects the true covariance matrix $\Sigma$. $\rho_n$ denotes the upper bound of the estimation error of $\Sigma$, typically on the order of $O(n^{-1/2})$ when $\Sigma$ has a bounded trace~\citep{koltchinskii2017concentration,vershynin2018high}.

The following examples of $\Sigma$ satisfy Assumption \ref{ass:emp}. The first example is for sub-Gaussian random vectors.

\begin{example}(Sub-Gaussian random vectors)\label{example:subg}
Suppose $\bm{x}_i, i= 1,\ldots,n$ are sub-Gaussian random vectors with covariance matrix $\Sigma$. 
From Theorem 9 in \cite{koltchinskii2017concentration}, there exists $t_n >0$ such that, with probability at least $1-\exp(-t_n)$, it holds that for some constant $c>0$,
    ${\| \hat{\Sigma} - \Sigma \|_{\mathrm{op}}}/  {\|\Sigma\|_{\mathrm{op}}} \le  c \sqrt{({r_0(\Sigma)+t_n})/{n}}$.
In other words, Assumption~\ref{ass:emp} holds with $\rho_n = O(\sqrt{(r_0(\Sigma) + t_n) / n})$ and  $\ell_n = \exp(-t_n)$.
\end{example}

Next, we consider the log-concave random vectors, which include heavy-tailed distributions, such as the Laplace distribution \citep{bagnoli2006log}.

\begin{example}(Log-concave random vectors)\label{example:logc}
 Suppose $\bm{x}_i, i=1,\ldots,n$ are log-concave random vectors with the covariance $\Sigma$ satisfying \eqref{def:non-trace_class}. In addition, suppose that $\Sigma$ satisfies $r_0(\Sigma) = o(n)$. From Theorem 3 in \cite{zhivotovskiy2021dimension}, there are constants $c_1,c_2>0$ such that, with probability at least $1- c_1\exp\{- (n ~ r_0(\Sigma))^{1/4}\}- 2\exp(-r_0(\Sigma)) $, it holds that
    ${\| \hat{\Sigma} - \Sigma \|_{\mathrm{op}}}/  {\|\Sigma\|_{\mathrm{op}}} \le  c_2 \sqrt{{r_0(\Sigma)} / {n}}$.
In other words, Assumption \ref{ass:emp} holds with $\rho_n = O(\sqrt{r_0(\Sigma) / n})$ and  $\ell_n = O(\exp\{- (n r_0(\Sigma))^{1/4}\} + \exp(-r_0(\Sigma)))$.
\end{example}

As these examples show, Assumption~\ref{ass:emp} allows a broad class of covariate distributions. Specifically, Example~\ref{example:logc} can handle covariates that follow a log-concave distribution, and there are additional ways to analyse their tail probability (e.g., \cite{nakakita2022benign}). 

\subsection{Posterior Contraction in Advanced Case}

If Assumption~\ref{ass:emp} is satisfied, the posterior contraction holds, even without sub-Gaussianity and finite traces. Hence, the following theorem holds under several conditions with general $\rho_n$ satisfying Assumption~\ref{ass:emp}. The result also incorporates the effective variance/bias bounds $V_n$ and $B_n$ defined in \eqref{eq:V_n} and \eqref{eq:B_n}.

\begin{theorem}\label{thm:posadv}
    Consider the regression model~\eqref{eqn-model} and the posterior distribution of $\bm{\theta}$ and $\sigma^2$ with $R\le\infty$. Suppose that Assumption~\ref{ass:emp} with $\rho_n = o(\lambda_{L_{\kappa}+1})$ holds and $R$ satisfies $\|\bm{\theta}^*\|_{\Sigma} < R/2$ and $\| \bm{\theta}^* \|_2 < \infty$. For any sequence $\{\varepsilon_n\}_{n \in \N}$ satisfying that for any $\delta>0$ and $M>0$, there exists $\overline{n}\in\N$ such that for any $n\ge \overline{n}$, $\varepsilon_n<\delta$ and $n\varepsilon_n^2>M$, $n/\lambda_{L_{\kappa}} > r_{L_{\kappa}}(\Sigma) R_{L_{\kappa}}(\Sigma)$, ${\varepsilon}_n^2 > 4 \| \bm{\theta}^* \|_2^2 \sum_{i=L_{\kappa}+1}^p {\lambda}_i$, and $\max\{ V_n(\varepsilon_n), B_n\} \le \varepsilon_n^2$, then, we have the following as $n \to \infty$, for some constant $C>0$:
\begin{align}
\Pi\left(\{\bm{\theta}:\|\bm{\theta}-\bm{\theta}^*\|_{\Sigma}> C\varepsilon_n \}\mid \,\mD\right) \overset{P^*}{\longrightarrow} 0.
\end{align}
Additionally, for the posterior distribution of $\sigma^2$, we have, for some constant $C>0$:
\begin{align}    
\Pi\left(\{\sigma^2:|\sigma^2-(\sigma^{*})^2|> C \varepsilon_n\} \,\mid \,\mD\right) \overset{P^*}{\longrightarrow} 0.
\end{align}
\end{theorem}

This result demonstrates the robustness of our Bayesian approach, achieving posterior contraction under weak distributional assumptions on the covariates, unified through the covariance estimation accuracy condition.

\section{Approximation for Posterior Distribution} \label{sec:approx_posterior}

We examine an approximation theorem for the posterior distribution using a truncated Gaussian distribution. This analysis is useful when the posterior distribution does not have a closed-form expression. Specifically, we develop a variant of the Bernstein--von Mises theorem to approximate posterior distributions commonly studied in Bayesian analyses~\citep{le2012asymptotic,ghosal2017fundamentals}.

\subsection{Setup and Approximator}

In this section, we approximate the posterior distribution $\Pi( \cdot \mid \mD, \sigma^2)$ of $\bm{\theta}$ for any fixed $\sigma^2$, which is defined as
\begin{align}
    &\Pi(B \mid \mD, \sigma^2) \\
    & = \frac{
        \sum_{k=L_{\kappa}}^{U_{\kappa}} \int_{T_{\mD_1,k}^{-1}(B\cap S_{\mD_1,k})}
         \pi_{\mathfrak{b}}(\bm{\beta}_{1:k}\mid \mD_1,k) \pi_{\kappa}(k) \prod_{(y_i, \bm{x}_i) \in \mD_2}\phi(y_i;\bm{x}_i^{\top}T_{\mD_1,k}(\bm{\beta}_{1:k}),\sigma^2 )
        \mathrm{d} \bm{\beta}_{1:k}
    }{
        \sum_{k'=L_{\kappa}}^{U_{\kappa}} \int_{\R^{k'}}
         \pi_{\mathfrak{b}}(\bm{\beta}_{1:k'}' \mid \mD_1,k') \pi_{\kappa}(k') \prod_{(y_i, \bm{x}_i) \in \mD_2}\phi(y_i;\bm{x}_i^{\top}T_{\mD_1,k'}(\bm{\beta}_{1:k'}'),\sigma^2 ) 
        \mathrm{d} \bm{\beta}_{1:k'}'
    } \label{eq:definition_posterior_fixsigma}
\end{align}
for any Borel set $B \in \mB(\R^p)$. While we analyse the approximation for fixed $\sigma^2$, in practice, if $(\sigma^*)^2$ is unknown, we can estimate $(\sigma^*)^2$ separately and substitute this value \citep{tong2005estimating,wang2008effect,shen2020optimal}. 

As an approximator, we define a truncated Gaussian distribution conditional on $\mD$ and $\sigma^2$.
In preparation, we define $X_2 = (\bm{x}_{\frac{n}{2} + 1},\bm{x}_{\frac{n}{2} + 2},\ldots,\bm{x}_{n})^{\top} \in \R^{\frac{n}{2}\times p}$, $\bm{y}_2 = (y_{\frac{n}{2} + 1},y_{\frac{n}{2} + 2},\ldots,y_{n})^{\top} \in \R^{\frac{n}{2}}$, $\hat{V}_{1:L_{\kappa}} = (\hat{\bm{v}}_1,\hat{\bm{v}}_2,\ldots,\hat{\bm{v}}_{L_{\kappa}}) \in \R^{
p\times L_{\kappa}}$ and the minimum norm interpolator:
\begin{equation}
    \bar{\bm{\theta}}_2 := (X_2^{\top}X_2)^{\dagger}X_2^{\top}\bm{y}_2 
    \equiv \argmin_{\bm{\theta}\in \R^p} \left\{\|\bm{\theta}\|_2 ; \sum_{(y_i, \bm{x}_i) \in \mD_2} (y_i - \bm{x}_i^\top \bm{\theta})^2 = 0 \right\},
\end{equation}
which is the effective non-Bayesian estimator in the over-parameterized linear regression \citep{bartlett2020benign,tsigler2020benign}. Further, we introduce a projected version of the interpolator $\bar{\bm{\theta}}_2$ and a corresponding covariance matrix onto the (prior) principal subspace spanned by $\hat{V}_{1:L_{\kappa}}$. Given $\sigma^2$, we define the projected empirical covariance matrix $ \bar{\Sigma}_{L_\kappa} := \hat{V}_{1:L_{\kappa}}^{\top}X_2^{\top}X_2\hat{V}_{1:L_{\kappa}}$ and its regularized version $\Lambda_{\sigma^2} := \tilde{\Sigma}_{1:L_{\kappa}}^{-1} +  ({2\sigma^2 })^{-1} \bar{\Sigma}_{L_\kappa} \in \R^{L_{\kappa} \times L_{\kappa}}$.
Also, we define the projected version of  $\bar{\bm{\theta}}_2$ as 
$\bm{m}_{\sigma^2} := \{ I - (({2\sigma^2 })^{-1}\bar{\Sigma}_{L_\kappa})^{-1} \tilde{\Sigma}_{1:L_{\kappa}}^{-1}\} \bar{\Sigma}_{L_\kappa}^{-1}\hat{V}_{1:L_{\kappa}}^{\top}X_2^{\top}X_2\bar{\bm{\theta}}_2$. This vector is scaled by a shrinkage factor, which vanishes as the sample size increases. Then, we define the $p$-dimensional approximator distribution $\Pi^{\infty}(\cdot\mid \mD,\sigma^2)$ as, for any Borel set $B \in \mB(\R^p)$,
\begin{align}
    &\Pi^{\infty}_{\vartheta}(B\mid \mD,\sigma^2)\\
    &:= \frac{\int_{T_{\mD_1,L_\kappa}^{-1}(B\cap S_{\mD_1,L_\kappa})} \exp\left\{-(\bm{\beta}_{1:L_{\kappa}} - \bm{m}_{\sigma^2})^\top \Lambda_{\sigma^2} (\bm{\beta}_{1:L_{\kappa}} - \bm{m}_{\sigma^2})\right\} \mone\left\{    \|\bm{\beta}_{1:L_{\kappa}}\|_{\tilde{\Sigma}_{1:L_{\kappa}}}\le R 
  \right\}d\bm{\beta}_{1:L_{\kappa}}}{\int_{\R^{L_{\kappa}}} \exp\left\{-(\bm{\beta}_{1:L_{\kappa}}' - \bm{m}_{\sigma^2})^\top \Lambda_{\sigma^2} (\bm{\beta}_{1:L_{\kappa}}' - \bm{m}_{\sigma^2})\right\}\mone\left\{   \|\bm{\beta}_{1:L_{\kappa}}'\|_{\tilde{\Sigma}_{1:L_{\kappa}}}\le R 
 \right\}d\bm{\beta}_{1:L_{\kappa}}'}.\label{def:approximator_density}
\end{align}

\subsection{Distribution Approximation Theorem}
We develop the distribution approximation theorem for the posterior distribution \eqref{eq:definition_posterior_fixsigma}, as a variant of the Bernstein--von Mises theorem.
Its proof is provided in the appendix.
\begin{assumption}\label{ass:emp2}
Given a covariance matrix $\Sigma$ with its eigenvalues $\{\lambda_k\}_{k=1}^p$, the sequence $\{\rho_n\}_{n \in \N}$ in Assumption~\ref{ass:emp} and the eigenvalues $\{{\lambda}_k\}$ satisfy $\lambda_{L_{\kappa}}^{-1}=o(n^{1/5})$ and $ \rho_n = o\left({\lambda}_{L_{\kappa}} ({\lambda}_{L_{\kappa}}-{\lambda}_{L_{\kappa}+1}) \right)$ as $n \to \infty$.
Also, the sequence $\{\ell_n\}_{n \in \N}$ in Assumption~\ref{ass:emp} satisfies that $\pr(\|\bar{\bm{\theta}}_2\|_2 \leq C) \geq 1 - \ell_n$ with some finite $C > 0$. 
\end{assumption}
These assumptions serve several purposes. The first condition prevents the $L_{\kappa}$-th eigenvalue from becoming excessively small, ensuring that the effective signal in the principal subspace remains sufficiently strong. The second condition requires that the covariance estimation error is dominated by the signal strength, which is mild given that $\rho_n$ typically decays at nearly an $n^{-1/2}$ rate.

\begin{theorem} \label{thm:disapp}
Consider the posterior distribution \eqref{eq:definition_posterior_fixsigma} corresponding to the regression model \eqref{eqn-model} with $R<\infty$ and an arbitrarily fixed $\sigma^2 > 0$. Suppose that Assumptions~\ref{ass:emp} and \ref{ass:emp2} hold. Then, if we take a decreasing function $f$ such that $f(L_{\kappa})-f(k)= \omega(\log U_{\kappa})$ for $k= L_{\kappa}+1,\ldots,U_{\kappa}$,
the distribution $\Pi^{\infty}( \cdot \mid \mD, \sigma^2)$ in \eqref{def:approximator_density} satisfies the following as $n \to \infty$:
\begin{align}
    \left\|\Pi(\cdot\mid \mD,\sigma^2) - \Pi^{\infty}(\cdot\mid \mD,\sigma^2)\right\|_{\mathrm{TV}} \xrightarrow[]{P^{*}} 0 . 
\end{align}
Furthermore, if $\sum_{n=1}^{\infty} \ell_n < \infty$ holds, then we have almost sure convergence.
\end{theorem}

From Theorem~\ref{thm:disapp}, we can easily approximate the posterior $\Pi(\cdot\mid \mD,\sigma^2)$ without laborious Bayesian computation, thereby simplifying the process of acquiring credible intervals and other Bayesian inferences. We can obtain a similar result on $\|\Pi(\cdot\mid \mD) - \Pi^{\infty}(\cdot\mid \mD)\|_{\mathrm{TV}}$ by marginalizing $\sigma^2$ with respect to $\pi_{\varsigma^2}$.

The condition for almost sure convergence, $\sum_{n=1}^{\infty} \ell_n < \infty$, is satisfied in the case where the covariates are sub-Gaussian by choosing $t_n$ in Example~\ref{example:subg} to be $\omega(\log n)$. In the log-concave case, it is satisfied for data where $r_0(\Sigma)$ in Example~\ref{example:logc} is $\omega(\log n)$.

\section{Proof Overview and Discussion}  \label{sec:proof-outline}

This section presents theoretical details of our results and connects them to related work.

\subsection{Generalized Statement}
We present a generalized version of the statement of Theorems~\ref{thm:poscon}~and~\ref{thm:posadv} for the posterior contraction, which is useful for outlining the proof of our theoretical analysis. Its difference from Theorems~\ref{thm:poscon}~and~\ref{thm:posadv} is that it generalizes the eigenvalue decay condition by introducing a new parameter $\lambda \geq 0$, which can depend on $n$.

In preparation for presenting the generalized statement, we define additional notation and an extended version of the variance/bias bounds $V_n(\varepsilon)$ and $B_n$ defined in Section \ref{sec:posterior_contraction}.
Let $\lambda \geq 0$ be a (virtual) regularization factor.
We first define $\bm{\beta}^* := (\hat{\bm{v}}_1,\ldots,\hat{\bm{v}}_p)^{\top} \bm{\theta}^*$,  $\bar{\mI} := \{ i=1,\ldots,U_{\kappa} : n  ({\lambda_i \sum_{i>L_{\kappa}} \lambda_{i}^2})^{1/2} > (\sum_{i>L_{\kappa}} \lambda_{i})^{1/2} (\lambda + \sum_{i>L_{\kappa}} \lambda_{i}) \}$, and $\tilde{\varepsilon}^2 := {\varepsilon}^2 - \sum_{i=L_{\kappa}+1}^p {\lambda}_i \beta_i^{*2} - 2\lambda_1\rho_n \sum_{i=L_{\kappa}+1}^p \beta_i^{*2}$ with $\varepsilon > 0$.
Also, we define the extended bounds as 
\begin{align}
    \bar{V}_n(\varepsilon) &= \sum_{i\in \bar{\mI}} \frac1n \log \left( \frac{n\sqrt{\lambda_i} \sqrt{\sum_{i>L_{\kappa}} \lambda_{i}^2}}{ (\lambda + \sum_{i>L_{\kappa}} \lambda_{i}) \sqrt{\sum_{i>L_{\kappa}} \lambda_{i}} } \right) + \frac{U_{\kappa}+\log( n/2\varepsilon)}{n} +\frac{n \sum_{i>L_{\kappa}} \lambda_{i}^2 }{(\lambda + \sum_{i>L_{\kappa}} \lambda_{i})^2}, 
    \label{eq:V_n:extension} \\
    \bar{B}_n(\varepsilon) &= 
        \frac{1}{n}\log\frac{1}{\varepsilon\pi_{\kappa}(L_{\kappa})} + \frac1n\sum_{i=1}^{L_{\kappa}} \frac{\beta_i^{*2}}{{\lambda}_i} + \frac{L_{\kappa}}{n} \log \left(\frac{{ 2 L_{\kappa} \lambda_1}}{\tilde{\varepsilon}}\right),\label{eq:B_n:extension}
\end{align}
for $\varepsilon > 0$. 
The largest difference from $(V_n(\varepsilon), B_n)$ in Section \ref{sec:posterior_contraction} is the presence of the regularization factor $\lambda$, which allows for error evaluation in generalized situations.

\begin{theorem}\label{thm:posext}
    Consider the regression model~\eqref{eqn-model} and the posterior distribution of $\bm{\theta}$ and $\sigma^2$, with $R\le\infty$. Suppose that Assumption~\ref{ass:emp} with $\rho_n = o(\lambda_{L_{\kappa}+1})$ holds and $R$ satisfies $\| \bm{\theta}^* \|_{\Sigma} <R/2$ and $\| \bm{\theta}^* \|_2 < \infty$. 
    For a sequence $\{\varepsilon_n\}_{n \in \N}$ satisfying that for any $\delta>0$ and $M>0$, there exists $\overline{n}\in\N$ such that for any $n\ge \overline{n}$, $\varepsilon_n<\delta$ and $n\varepsilon_n^2> M$,  $n \sum_{i > L_{\kappa}} \lambda_{i}^2 / \sum_{i > L_{\kappa}} \lambda_{i} \left( \lambda + \sum_{i > L_{\kappa}} \lambda_{i} \right)^2  \ge  \mone\left\{ n^2 \varepsilon_n^2 \lambda_{L_{\kappa}} \sum_{i>L_{\kappa}} \lambda_{i}^2  \le R^2 \sum_{i>L_{\kappa}} \lambda_{i}  (\lambda + \sum_{i>L_{\kappa}} \lambda_{i})^2\right\}$ for some $\lambda \geq 0$, $\tilde{\varepsilon}_n^2> 0$, and
        $\max\{ \bar{V}_n(\varepsilon_n), \bar{B}_n(\varepsilon_n)\} \le \varepsilon_n^2$,
    then, we have the following as $n \to \infty$, for some constant $C>0$:
\begin{align}
\Pi\Big(\{\bm{\theta}:\|\bm{\theta}-\bm{\theta}^*\|_{\Sigma}> C\varepsilon_n \}\mid \,\mD\Big) \overset{P^*}{\longrightarrow} 0.
\end{align}
Additionally, for the posterior distribution of $\sigma^2$, we have, for some constant $C>0$:
\begin{align}    
\Pi\Big(\{\sigma^2:|\sigma^2-(\sigma^{*})^2|> C \varepsilon_n\} \,\mid \,\mD\Big) \overset{P^*}{\longrightarrow} 0.
\end{align}
\end{theorem}

We elaborate on the supplementary conditions pertinent to the aforementioned theorem. The condition $\rho_n = o(\lambda_{L_{\kappa}+1})$ imposes a restriction on the value of $L_{\kappa}$, although this constraint is reasonable since the left-hand side decays at a fast rate close to $n^{-1/2}$. $\| \bm{\theta}^* \|_2 < \infty$ is standard in scenarios where the dimension $p$ diverges to infinity, as in~\cite{bartlett2020benign,tsigler2020benign}. 
The condition $n \sum_{i > L_{\kappa}} \lambda_{i}^2 / \sum_{i > L_{\kappa}} \lambda_{i} \left( \lambda + \sum_{i > L_{\kappa}} \lambda_{i} \right)^2  \ge  \mone\left\{ n^2 \varepsilon_n^2 \lambda_{L_{\kappa}} \sum_{i>L_{\kappa}} \lambda_{i}^2  \le R^2 \sum_{i>L_{\kappa}} \lambda_{i}  (\lambda + \sum_{i>L_{\kappa}} \lambda_{i})^2\right\}$ is technically intricate and is addressed in Proposition~\ref{thm:bound2}.
Concerning the condition ${\varepsilon}_n^2 - \sum_{i=L_{\kappa}+1}^p {\lambda}_i \beta_i^{*2} - 2\lambda_1\rho_n \sum_{i=L_{\kappa}+1}^p \beta_i^{*2}> 0$, its sufficient conditions are
\begin{equation} \label{eq:condition_ourBias}
    {\varepsilon}_n^2 > 4\lambda_1 \rho_n \sum_{i=L_{\kappa}+1}^p \beta_i^{*2} \quad \text{and} \quad {\varepsilon}_n^2 > 2\sum_{i=L_{\kappa}+1}^p {\lambda}_i \beta_i^{*2}.
\end{equation}
The former condition is mild since $\rho_n$ typically decays at an order approximately equal to $n^{-1/2}$, and other factors including $\sum_{i>L_{\kappa}} \beta_i^{*2}$ are bounded by a constant. The latter condition indicates that regions with smaller spectral values contribute negligibly. The relationship between this condition and the prior work is discussed in Section~\ref{subsec: connect}.

We provide two examples for the prior setting, $\bar{B}_n=  \bar{B}_n(\varepsilon_n)$, and $\bar{V}_n = \bar{V}_n(\varepsilon_n)$ with a suitably selected $\varepsilon_n$ under the conditions of Theorem~\ref{thm:posext}: one illustrating a trace-class example and another that does not satisfy the trace-class condition.
\begin{example}\label{example: polynomial_decay1}
    Consider the setup with the dimension $p=\omega(n)$ and eigenvalues $\{\lambda_k\}_{k=1}^{p}$ of the form
    $\lambda_k = k^{-\alpha}$ with $\alpha > 1$. Then, if $L_{\kappa} = n^{{1}/  ({2\alpha+2})}$, $U_{\kappa} = n^{{1}/  {(\alpha+1)}}$, $\lambda = n^{({2\alpha+3})/({4\alpha+4})}$, and $\pi_{\kappa}(L_{\kappa}) > c$ with some positive constant $c$, the following evaluation holds:
    \begin{align*}
        \bar{V}_n = O\left(n^{-{\alpha} / ({\alpha+1})}\right)   \quad \mathrm{and} \quad  \bar{B}_n  = O\left(n^{-1/2}\right),
    \end{align*}
    and if $\|\bm{\beta}_{1:L_{\kappa}}^*\|_2 = o(1)$, $\bar{B}_n = O\left(n^{-(\alpha+2)/({2\alpha+2})}\right)$ holds. 
\end{example}

\begin{example}\label{example: polynomial_decay2}
    Consider the setup with eigenvalues $\{\lambda_k\}_{k=1}^{p}$ of the form
    $\lambda_k = k^{-\alpha}$ with $0< \alpha \le 1$.  Then, if $L_{\kappa} = n^{{1} / {(2\alpha+2)}}$, $U_{\kappa} = n^{{1} / {(\alpha+1)}}$, $\lambda = n^{({2\alpha+3})/{(4\alpha+4)}}$, and $\pi_{\kappa}(L_{\kappa}) > c$ with any positive constant $c$, the following relation holds:
    \begin{align*}
        \bar{V}_n = O\left(n^{-{\alpha} / ({\alpha+1})}\right) \quad \mathrm{and} \quad  \bar{B}_n = O\left(n^{-{\alpha} / ({\alpha+1})}\right),
    \end{align*}
    where $p = \omega(n)$ if $\alpha=1$, $p> n^{({2\alpha+3})/{(4-4\alpha^2)}}$ if ${1} / {2} < \alpha < 1$, $p > n^{{4}/  {3}}$ if $\alpha = {1} / {2}$, and $p>n^{({2\alpha+1})/({\alpha+1})}$ if $0<\alpha < {1} / {2}$.
\end{example}

In Example~\ref{example: polynomial_decay1}, the eigenvalues follow a polynomial decay governed by the parameter $\alpha>1$, where $\Sigma$ is trace class. The evaluations show that as $ \alpha$ increases, indicating the simplicity of the data structure, the convergence rate of $\bar{V}_n$ improves.
Conversely, Example~\ref{example: polynomial_decay2} explores situations in which $\Sigma$ lacks a finite trace. These conditions on $p$ are imposed to establish the order $O\left( n^{- {\alpha}/  ({\alpha + 1})} \right)$, and the convergences remain valid even if these specific conditions on $p$ are not met. When simply contrasted with Example~\ref{example: polynomial_decay1}, the present results exhibit slower convergence. This suggests that the increased complexity in data generation demands a larger dataset for reliable predictions.

\subsection{Proof Overview}
We provide overviews of the proofs of Theorem~\ref{thm:posext}. The same discussion holds for Theorems~\ref{thm:poscon}~and~\ref{thm:posadv}.

To begin with, we describe how we apply the posterior contraction theory.
In preparation, we define a parameter space $\mP = \R^p \times \R_+$ and introduce its distance $d_{\mP}$ induced by the norms $\|\cdot\|_\Sigma$ for $\R^p$ and $|\cdot|$ for $\R_+$.
We also define a subset $\mP_{n,\varepsilon, \lambda}:=\mP_{n,\varepsilon,\lambda}^1 \times \mP_{n}^2 \subset \mP$, where
$\mP_{n,\varepsilon,\lambda}^1:=\left\{\bm{h}\in\mathbb{R}^p  : \|\bm{h}\|_2 \leq H_{n,\varepsilon,\lambda} \right\}$ and $\mP_{n}^2:=\{a\in\mathbb{R}_+ : n^{-1}\le a\le n\},$ where $H_{n,\varepsilon,\lambda}^2 = n^2 \varepsilon^2 \sum_{i>L_{\kappa}} \lambda_{i}^2 /\sum_{i>L_{\kappa}} \lambda_{i}  (\lambda + \sum_{i>L_{\kappa}} \lambda_{i})^2$. Let $P_{\bm{\theta}, \sigma^2}$ be a joint distribution of $(\bm{x},y)$ with the linear regression model~\eqref{eqn-model} with $(\bm{\theta}, \sigma^2)$, and recall that $P^* (= P_{\bm{\theta^*}, (\sigma^*)^2})$ is the joint distribution with the true parameter $(\bm{\theta}^*, (\sigma^*)^2)$.
We define KL-type neighbourhood as $\mB_{KL}(\varepsilon) :=\left\{(\bm{\theta},\sigma^2): \KL(P^*, P_{\bm{\theta}, \sigma^2}) \leq \varepsilon^2, \KV(P^*,P_{\bm{\theta}, \sigma^2}) \leq \varepsilon^2 \right\}$.
Then, from Theorem 2.1 of \cite{ghosal2000convergence}, to achieve the statements of Theorem~\ref{thm:posext}, it is sufficient to prove that there exists $\varepsilon_n$ such that the following three conditions hold with probability at least $1-\ell_n$ for any sufficiently large $n$, with some constant $C>0$,
\begin{align}
    & \frac{1}{n} \log \mN(\varepsilon_n, \mP_{n,\varepsilon_n, \lambda}, d_{\mP} )\leq \varepsilon_n^2 , \\
    & -\frac{1}{n}\log \Pi_{\vartheta,\varsigma^2} (\mB_{KL}(\varepsilon_n)\mid \mD_1) \leq C \varepsilon_n^2 , \\
    & \frac{1}{n}\log \Pi_{\vartheta,\varsigma^2} (\mP\setminus \mP_{n,\varepsilon_n, \lambda} \mid \mD_1) \leq -(C+4) \varepsilon_n^2 .
\end{align}
In the following, we discuss these three conditions separately.

\begin{proposition} \label{thm:bound1}
    For any $\varepsilon>0$ and some $\lambda \ge 0$, we have the following for any sufficiently large $n \in \N$:
    \begin{align}\label{ggv-thm2.1-1}
       \frac1n\log \mN(\varepsilon, \mP_{n,\varepsilon, \lambda}, d_{\mP} ) \lesssim \bar{V}_n(\varepsilon).
    \end{align}
\end{proposition}

\noindent\textbf{Proof Outline of Proposition~\ref{thm:bound1}.} 
To effectively bound the covering numbers associated with the high-dimensional parameter space, we begin by decomposing the entropy into more manageable components: a ``spike'' part $\mN(\varepsilon,\mP_{n,\varepsilon,\lambda}^1,\|\cdot\|_{\Sigma_{1:k}})$ and a ``noise'' part $\mN(\varepsilon,\mP_{n,\varepsilon,\lambda}^1,\|\cdot\|_{\Sigma_{k+1:p}})$.
The noise part requires analysis in high-dimensional spaces characterized by small eigenvalues, which makes bounding its entropy particularly challenging. To overcome this difficulty, we employ the Sudakov minoration~\citep{wainwright2019high,gine2021mathematical} to transform the problem into one involving dual norms. This approach allows us to exploit dimension-free bounds that remain unaffected by the high dimensionality.
The spike part, on the other hand, is of moderate dimensionality (not low-dimensional) and its eigenvalues are not small. Then, we apply a tight bound for the covering numbers of ellipsoids in terms of Euclidean balls, as established by \cite{DUMER20061667}.
\begin{proposition}\label{thm:bound3}
    Suppose that Assumption~\ref{ass:emp} with $\rho_n = o(\lambda_{L_{\kappa}+1})$, $\| \bm{\theta}^* \|_{\Sigma} < R/2$ and $\| \bm{\theta}^* \|_2 < \infty$ hold.
    For a sequence $\{\varepsilon_n\}_{n \in \N}$ such that $\varepsilon_n\to 0$, and ${\varepsilon}_n^2 > \sum_{i=L_{\kappa}+1}^p ({\lambda}_i + 2\lambda_1\rho_n)\beta_i^{*2}$ where $\bm{\beta}^* := (\hat{\bm{v}}_1,\ldots,\hat{\bm{v}}_p)^{\top} \bm{\theta}^*$,
    we have the following with sufficiently large $n \in \N$, with probability at least $1-\ell_n$:
    \begin{align}
    -\frac{1}{n}\log \Pi_{\vartheta,\varsigma^2} (\mB_{KL}(\varepsilon_n)\mid \mD_1)\lesssim
    \bar{B}_n(\varepsilon_n) \label{ggv-thm2.1-3}. ~ 
    \end{align}
\end{proposition}

\noindent\textbf{Proof Outline of Proposition~\ref{thm:bound3}.}
We evaluate the KL-type neighbourhood on the prior measure. It is bounded from below by $\Pi_{\vartheta,\varsigma^2} \left(\mB_2\mid \mB_1, \mD_1 \right)\Pi_{\varsigma^2} \left(\mB_1\right),$ where $\mB_1 \subset \R_+$ and $\mB_2 \subset \R^p \times \R_+ $ are defined as follows:
  \begin{align}
	\mB_1& :=
		\Big\{ \sigma^2: \frac{(\sigma^*)^2}{\sigma^2}-\log\Big(\frac{(\sigma^*)^2}{\sigma^2}\Big)-1\leq \varepsilon_n^2,\quad\frac{(\sigma^*)^4}{\sigma^4}-\frac{2(\sigma^*)^2}{\sigma^2}+1\leq \varepsilon_n^2\Big\},  \label{def:B1}\\
	\mB_2 & := 
		\Big\{ (\bm{\theta}, \sigma^2):
		\frac12\frac{\|\bm{\theta} - \bm{\theta}^*\|^2_{{\Sigma}} }{\sigma^2} \leq \varepsilon_n^2, \quad \frac{(\sigma^*)^2}{\sigma^2}\frac{\|\bm{\theta} - \bm{\theta}^*\|^2_{{\Sigma}}}{\sigma^2}  \leq \frac{\varepsilon_n^2}{2} \Big\}. \label{def:B2}
  \end{align}
Both $\mB_1$ and $\mB_2$ are defined as intersections of inverse images of closed sets under continuous functions and are thus measurable.
The term $\Pi_{\varsigma^2} \left(\mB_1\right)$ can be bounded by the tail property of the inverse-Gaussian distribution and integral calculations. To evaluate $\mB_2$ using the prior distribution, we leverage the lower dimensionality of the prior by performing rotation and truncation. Specifically, we consider the transformations \(\bm{\beta} = \hat{V}^{\top} \bm{\theta}\) and \(\bm{\beta}^* = \hat{V}^{\top} \bm{\theta}^*\), where \(\hat{V} = (\hat{\bm{v}}_1, \ldots, \hat{\bm{v}}_p) \in \R^{p \times p}\) is an orthogonal matrix composed of the estimated eigenvectors. The critical set for evaluating \(\mB_2\) in \eqref{def:B2} becomes $\left\{ \bm{\beta}\in \R^p : \left\| \bm{\beta} - \bm{\beta}^* \right\|_{\mathrm{diag}(\tilde{\lambda}_1, \ldots, \tilde{\lambda}_p)} \leq \varepsilon_n \right\}$, where \(\tilde{\lambda}_i = \hat{\lambda}_i + \lambda_1 \rho_n < \lambda_i + 2\lambda_1 \rho_n \) with high probability by Lemma~\ref{lem:eigen-error} in the appendix. Given the condition \(\varepsilon_n^2 > \sum_{i > L_{\kappa}} (\lambda_i + 2\lambda_1 \rho_n) \beta_i^{*2}\), the \(\varepsilon_n\)-ball centred at \(\bm{\beta}^*\) in the metric space induced by the norm \(\| \cdot \|_{\mathrm{diag}(\tilde{\lambda}_1, \ldots, \tilde{\lambda}_p)}\) has positive mass within the support \(S_{\mD_1,k}\) (Figure~\ref{fig:assumption_bound}). Finally, we evaluate the probability, eliminate the influence of eigenvalue estimation errors using Lemma~\ref{lem:eigen-error} and establish the main statement.

\begin{figure}[t]
  \begin{center}
  \includegraphics[width=0.75\textwidth]{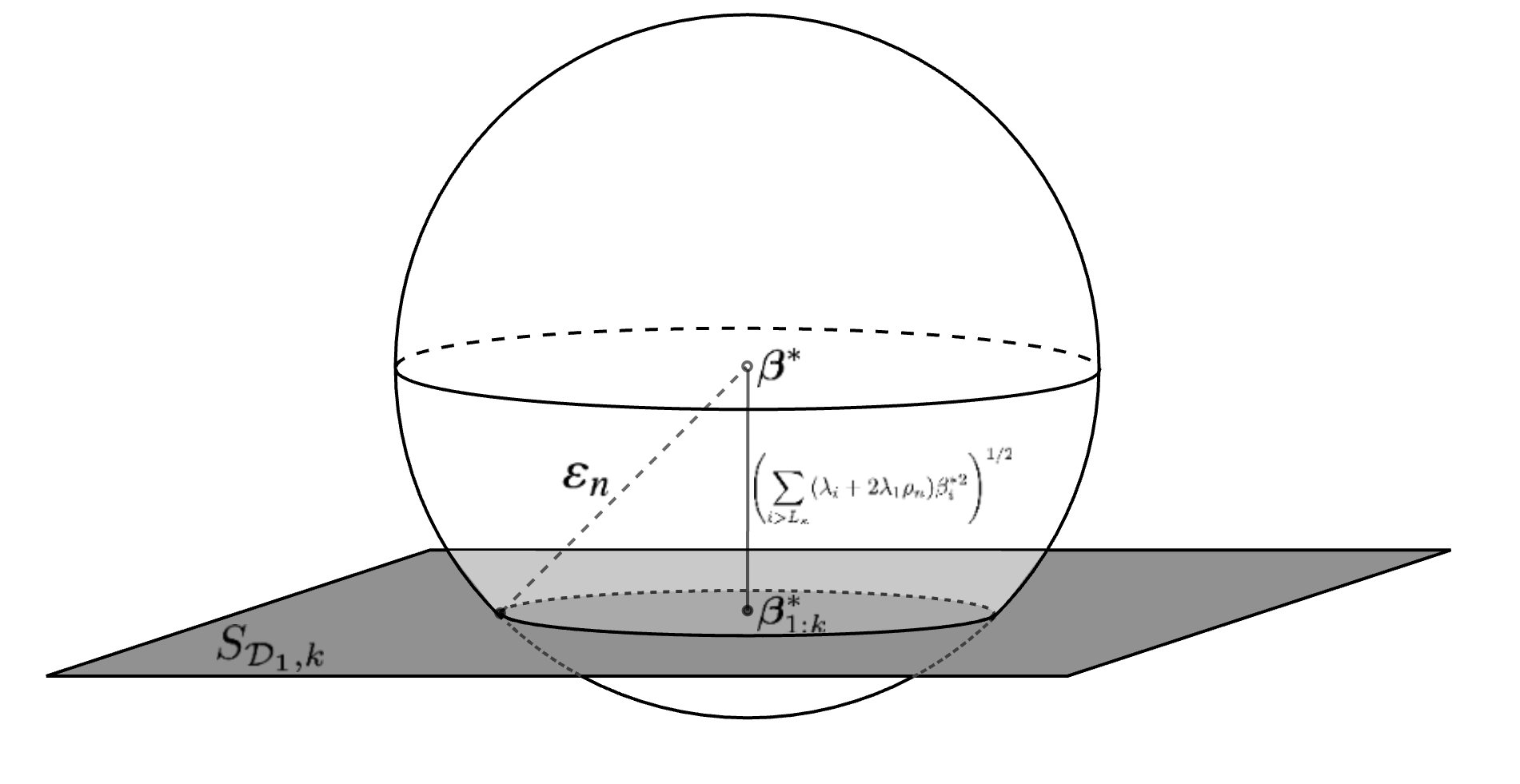}
  \end{center}
  \caption{An illustration of the support space $S_{\mD_1,k}$ of the prior distribution and an ellipse centred at $\bm{\beta}^*$, in the $p$-dimensional space $\mP_n^1$ equipped with the metric induced by the norm $\|\cdot\|_{\mathrm{diag}(\tilde{\lambda}_1,\ldots,\tilde{\lambda}_p)}$. \label{fig:assumption_bound}}
\end{figure}

\begin{proposition}\label{thm:bound2}    
    If Assumption~\ref{ass:emp} with $\rho_n = o(\lambda_{L_{\kappa}+1})$ holds, then, for any $\varepsilon>0$ and some $\lambda \ge 0$, we have the following for any sufficiently large $n \in \N$, with probability at least $1-\ell_n$:
    \begin{align}\label{ggv-thm2.1-2}
    \frac{1}{n} \log \Pi_{\vartheta,\varsigma^2}(\mP \setminus \mP_{n,\varepsilon, \lambda} \mid \mD_1) \lesssim
    \begin{cases} 
        \frac{\log \pi_{\kappa}(L_{\kappa})}{n} + \frac{ U_{\kappa}}{n} - \frac{H_{n,\varepsilon, \lambda}^2}{n},\quad  \text{if } \frac{n^2 \varepsilon^2 \lambda_{L_{\kappa}} \sum_{i>L_{\kappa}} \lambda_{i}^2 }{ \sum_{i>L_{\kappa}} \lambda_{i}  (\lambda + \sum_{i>L_{\kappa}} \lambda_{i})^2} \le R^2, \\
        - C_{\eta,\xi} , \quad \text{otherwise}, 
    \end{cases}
    \end{align}
    where $C_{\eta,\xi}>0$ is a constant depending upon the inverse-Gaussian parameters $\eta$ and $\xi$.
\end{proposition}

The condition $\eqref{ggv-thm2.1-2} < -\varepsilon^2$ is necessary for posterior contraction, and can be equivalently expressed as
    $({n \sum_{i > L_{\kappa}} \lambda_{i}^2 })/({ \sum_{i > L_{\kappa}} \lambda_{i} \left( \lambda + \sum_{i > L_{\kappa}} \lambda_{i} \right)^2 }) \geq  \mone\left\{ n^2 \varepsilon^2 \lambda_{L_{\kappa}} \sum_{i>L_{\kappa}} \lambda_{i}^2 \leq R^2 \sum_{i>L_{\kappa}} \lambda_{i}  \left(\lambda + \sum_{i>L_{\kappa}} \lambda_{i}\right)^2\right\}$.
This condition is violated when the value obtained by dividing the radius of $\mathcal{P}_{n,\varepsilon,\lambda}^1$ by $n$ becomes sufficiently small.

\noindent\textbf{Proof Outline of Proposition~\ref{thm:bound2}}.
In the proof, we study the discrepancy between the original parameter spaces $\mP$ and the restricted spaces $\mP_{n,\varepsilon, \lambda}$. Due to the one-dimensional nature of $\frac{1}{n} \log \Pi_{\varsigma^2} (\R_+ \setminus \mP_n^2)$, simple algebraic arguments show that it is bounded above by $-C_{\eta,\xi}$. In contrast, concerning $\frac{1}{n} \log \Pi_{\vartheta} (\R^p \setminus \mP_{n,\varepsilon,\lambda}^1)$, because working with a high-dimensional sphere is challenging, we reduce the problem to a medium-dimensional problem by rotating and truncating the prior distribution. The complexity of the calculations arises primarily from the elliptical support constraints of the prior distribution. In other words, we need to measure the tail probability of norms within an elliptical support. We define $E_1 = \{\bm{\beta}'\in\R^k : \|\tilde{\Sigma}_{1:k}^{1/2}\bm{\beta}'\|_2 \le R\}$ and $E_2 = \{\bm{\beta}'\in\R^k :  \|\bm{\beta}'\|_2 > H_{n,\varepsilon, \lambda}\}$, where $\tilde{\Sigma}_{1:k} = \mathrm{diag}(\hat{\lambda}_1, \ldots, \hat{\lambda}_{k})$.
The condition for $E_1 \cap E_2 = \emptyset$ is given by Weyl’s inequality and Assumption~\ref{ass:emp} as $H_{n,\varepsilon, \lambda} (\lambda_k -\lambda_1\rho_n)^{1/2} > R$. Under this condition, the probability is zero, so the statement holds. Otherwise, applying the Gaussian correlation inequality~\citep{Li1999GaussCorIneq}, we obtain $\pr( E_2 \mid E_1) \le \pr(E_2)$, allowing us to focus on measuring the tail probability of norms without the constraints. Subsequently, we transform the bounds on the tail probabilities of the Gaussian norm into a problem involving dual norms. Finally, properties of covering numbers and concentration inequalities yield the desired result.

\subsection{Connection to Ridge and Ridgeless Estimators} \label{subsec: connect}
In this section, we discuss a connection between the theoretical result for our posterior distribution and the risk of ridge and ridgeless estimators by \cite{tsigler2020benign}.

\cite{tsigler2020benign} assumes a diagonal covariance matrix of the covariates $\Sigma = \mathrm{diag}(\lambda_1,\ldots,\lambda_p)$, decomposes the predictive risk of the ridge/ridgeless estimators into a bias term and a variance term, and derives their upper bounds as follows:
\begin{equation}
    B_{TB} = \left( \frac{\lambda'+\sum_{i=k'+1}^p\lambda_i}{n}\right)^2 \sum_{i=1}^{k'} \frac{{\theta}_i^{*2}}{{\lambda}_i} + \sum_{i=k'+1}^p {\lambda}_i \theta_i^{*2} ,\quad V_{TB} = \frac{k'}{n} + \frac{n \sum_{i>k'} \lambda_{i}^2 }{(\lambda' + \sum_{i>k'} \lambda_{i})^2},
\end{equation}
where $\lambda' \geq 0$ represents the penalty parameter and $k' \in \N$  denotes a certain threshold value.
First, let us focus on the variance term. $\bar{V}_n$ in \eqref{eq:V_n:extension} appears consistent with $V_{TB}$ if we consider the dominance of its final two terms (a reasonable assumption given that the other terms merely consist of logarithmic summations).
Next, we explore the bias term. When analysing $\bar{B}_n$ as defined in \eqref{eq:B_n:extension}, we observe that $\frac{1}{n}\sum_{i=1}^{L_{\kappa}} \frac{\beta_i^{*2}}{{\lambda}_i}$ emerges as the rate-determining term, provided that $\pi_{\kappa}(L_{\kappa})$ and $\tilde{\varepsilon}_n$ do not exhibit exponential decay with respect to $n$. Furthermore, taking into account the assumption~\eqref{eq:condition_ourBias} concerning $\lambda_i$ and $\varepsilon_n$, the following relationships prove crucial (ignoring constant factors):
\begin{equation}\label{bias:ours}
      \frac{1}{n}\sum_{i=1}^{L_{\kappa}} \frac{\beta_i^{*2}}{{\lambda}_i} \le {\varepsilon}_n^2 \quad \mathrm{and} \quad  \sum_{i=L_{\kappa}+1}^p {\lambda}_i \beta_i^{*2} \le  {\varepsilon}_n^2,
\end{equation}
If we have the relation $\bm{\beta}^* = \bm{\theta}^*$, \eqref{bias:ours} is similar to $B_{TB}$. The difference in using $\bm{\beta}^*$ and $\bm{\theta}^*$ in the bounds arises from the fact that \cite{tsigler2020benign} restricts $\Sigma$ to be diagonal, while we do not introduce such a restriction.

\section{Selection Method of Hyperparameters} \label{sec:hyperparameter} 

We discuss a method for selecting important hyperparameters $R, L_\kappa$, and $U_\kappa$. 
While our theory allows for a range of hyperparameter choices satisfying certain conditions, we offer a data-driven method for practical purposes. This is just one proposal, and many other methods are possible.

Regarding the radius $R$ (if finite), the condition $R > 2\|\bm{\theta}^*\|_\Sigma$ is needed theoretically. Since we have $\Ep[y_i^2] = \|\bm{\theta}^*\|_\Sigma^2 + (\sigma^*)^2$, it is sufficient to select $R$ as satisfying $R^2 \geq 4 \Ep[y_i^2]$.
Hence, it is preferable to choose a large value of $R$ such that $\pr ( R > 2 \left( n^{-1} \sum_{i=1}^n y_i^2 \right)^{1/2} ) \geq 1 - \epsilon$ for some small $\epsilon > 0$.
Alternatively, one can simply select a sufficiently large value such as $10^5$ for $R$, which will be used in our experimental section.

Next, we consider the selection of $L_\kappa$ and $U_\kappa$.
Here, the interval $(L_\kappa, U_\kappa)$ should contain a value $k^* := \min \{ k\in\N : \sum_{i=k}^p \lambda_i > n \lambda_k \}$, which is referred to as the effective rank of $\Sigma$~\citep{bartlett2020benign,tsigler2020benign}. Based on this fact, we suggest selecting $L_{\kappa}$ and $U_{\kappa}$ as fractions and multiples of an estimated effective rank, respectively. 
Specifically, we select the parameters as $(L_{\kappa}, U_{\kappa}) = ( \hat{k^*}(\mD_1)/2,\, 2\hat{k^*}(\mD_1) )$, where $\hat{k^*}(\mD_1) := \min \{ k\in\N : \sum_{i=k}^p \hat{\lambda}_i > n \hat{\lambda}_k \}$ is the estimated effective rank.

\section{Simulation}\label{sec:num}
We empirically validated the properties of the posterior distribution in Section~\ref{sec:posterior}.
For sample sizes $n=50,100,200,400,$ and $800$, we generated the true parameter $\bm{\theta}^*\in \mathbb{R}^p$ from $N(\bm{0}, I_p)$, where $p=n^{4/3}$. Concerning the covariates, we set up the following two scenarios:
\begin{align}
    \mathrm{(i)}\ \bm{x}_i \sim N(\bm{0}, \Sigma), \quad i=1,\ldots, n;  \qquad\quad  \mathrm{(ii)}\ \bm{x}_i \sim L(\bm{0}, \Sigma), \quad i=1,\ldots,n.
\end{align}
Here, $L(\bm{0}, \Sigma)$ denotes a zero-mean multivariate Laplace distribution with covariance $\Sigma$. $\Sigma$ was generated as a matrix whose $j$th eigenvalue was $\exp(-j/2) + n\exp\{-n^{1/3}\}/p$ for $j=1,\ldots,p$. Then, we generated the response variable $y_i$ from $N(\bm{x}_i^{\top}\bm{\theta}^*, 1)$.

We generated $50$ different datasets using the above procedure and analysed each of them using the linear regression model in \eqref{eqn-model}. As the Bayesian method, we considered the priors $\Pi_{\vartheta}$ as \eqref{prior:theta} with $R=10^5$, $\Pi_{\varsigma^2}$ as \eqref{def:prior_sigma} with $\eta = \xi = 1$, and $\pi_{\kappa}$ as \eqref{def:prior_k} with $f(k) = 1$ and $(L_{\kappa}, U_{\kappa}) = (\hat{k^*}(\mD)/2, 2\hat{k^*}(\mD))$, where $\hat{k^*}(\mD) := \min \{k : \sum_{i=k}^p\hat{\lambda}_i > n\hat{\lambda}_i\}$. 
Note that the selection of the hyperparameters follows the approach presented in Section \ref{sec:hyperparameter}.
For the computation, we employed the Gibbs sampler~\citep[][]{gelfand1990sampling,gelman2013bayesian} to simulate the posterior distribution of $\bm{\theta}$ and extracted 10,000 posterior samples after a 10,000 burn-in period. The effective sample sizes of the parameters were more than $500$.

\begin{figure}[t]
  \begin{center}
  \includegraphics[width=\textwidth]{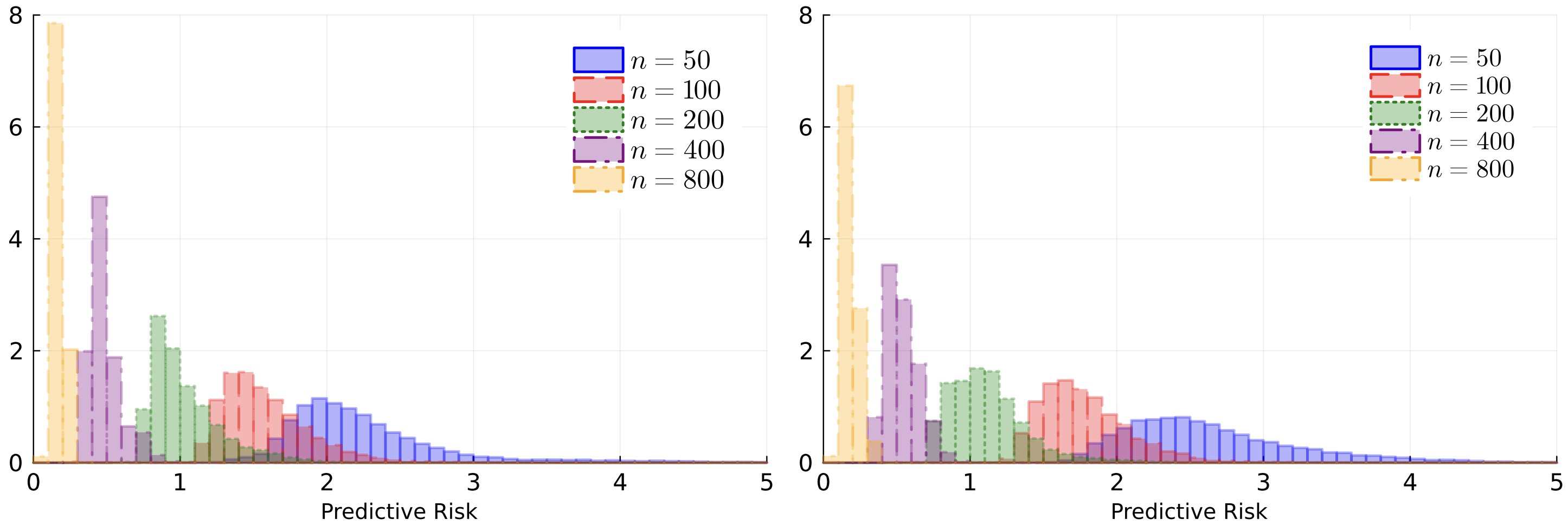}
  \end{center}
  \caption{Left and right figures show the histograms of the predictive risks for $n=50,\ 100, \ 200, \ 400,$ and $800$ in scenarios (i) and (ii), respectively.\label{fig:post1} }
  \vskip -0.2in
\end{figure}

First, we investigated the predictive risk of the proposed method in each scenario. Figure~\ref{fig:post1} illustrates histograms of 500,000 (10,000 samples $\times$ 50 datasets) predictive risks. This indicates that the risk tends to approach zero as the sample size $n$ increases. In addition, the deviation of the posterior risk decreases, and the distribution degenerates to a point-mass distribution at zero. This result is consistent with the theoretical findings presented in Theorems~\ref{thm:poscon}, \ref{thm:posadv}~and~\ref{thm:posext}, suggesting that the prediction works appropriately, even in the case of an over-parameterized model. Moreover, since the Laplace distribution has a fat tail, it may affect the posterior contraction in finite sample analyses. Consequently, the histograms in scenario (ii) are more dispersed compared to those in scenario (i).

\begin{figure}[tb]
  \begin{center}
  \includegraphics[width=\textwidth]{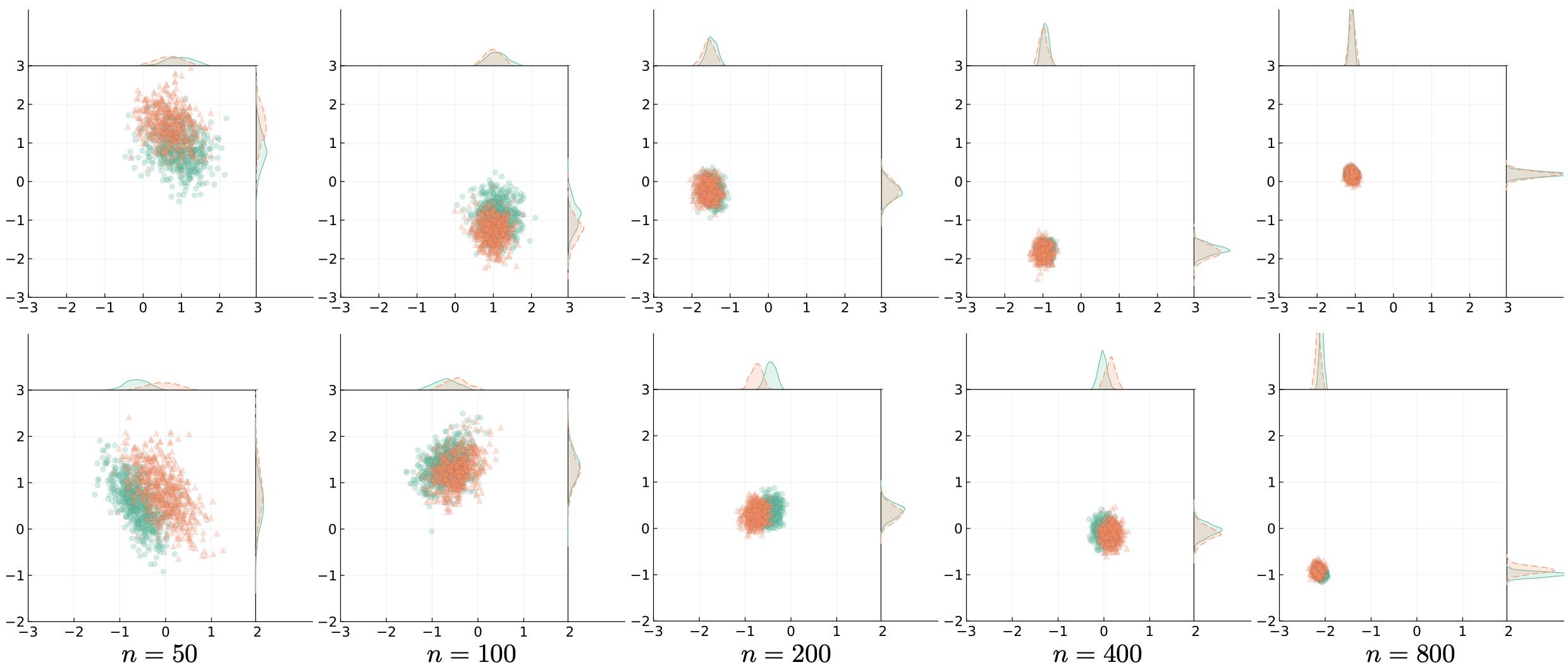}
  \end{center} 
  
  \caption{Circles and triangles respectively show samples by the Gibbs sampler and samples from approximate distributions of the two components of $\bm{\theta}$ with the largest variances in Scenario (i) (Scenario (ii)) for $n=50,\ 100,\ 200,\ 400,$ and $800$ if a panel is in the top row (bottom row).\label{fig:post2}
  The solid curve corresponds to the density estimate for the circle samples, while the dashed curve corresponds to the density estimate for the triangle samples.
  }
  \vskip -0.2in
\end{figure}

Next, we investigated the distributional approximation result of Theorem~\ref{thm:disapp} by visualizing the posterior samples from the Gibbs sampler and the samples from the approximate distribution defined in \eqref{def:approximator_density}. To satisfy the condition in Theorem~\ref{thm:disapp}, we set $f(k) = -k^2$. Figure~\ref{fig:post2} demonstrates the samples from the two distributions for one representative dataset. Specifically, the circles and triangles respectively represent samples by the Gibbs sampler and samples from the approximator of the two components of $\bm{\theta}$ with the largest variances in Scenario (i) (Scenario (ii)) for $n=50,\ 100,\ 200,\ 400,$ and $800$ if the panel is in the top row (bottom row). In both scenarios, the two sample sets increasingly overlap as the sample size grows. This corresponds to the fact that the distance between the posterior distribution and the approximate distribution converges to zero (Theorem~\ref{thm:disapp}). The marginal posterior distribution of each coordinate resembles a bell-shaped symmetrical distribution, supporting the validity of the Gaussian approximation.

\section{Real Data Analysis}\label{sec:real}
We apply the proposed method to a real dataset. 
Specifically, we predict depressive symptoms from the~\textit{HELPfull} dataset included in the \texttt{R} package~\textit{mosaicData}~\citep{mosaicdata}. The dataset contains the results of a clinical trial with adult inpatients selected from detoxification units. We use this dataset to consider the problem of predicting the subjects' CES-D scale (Center for Epidemiologic Studies Depression), a measure of depressive symptoms for which higher scores indicate more severe symptoms. The raw data include $788$ covariates (excluding the CES-D) for $1472$ subjects. Because all subjects had missing values, we removed $320$ covariates in addition to the subject ID and CESD-CUT (binary data of whether the CES-D is greater than $21$), which is highly related to CES-D. We then eliminated any subjects who still had missing values. The data cleaning process yielded 516 samples of 465-dimensional covariates.

We applied the proposed method to the training data, with the following prior setup: $\Pi_{\vartheta}$ as \eqref{prior:theta} with $R=10^5$, $\Pi_{\varsigma^2}$ as \eqref{def:prior_sigma} with $\eta = \xi = 1$, and $\pi_{\kappa}$ as \eqref{def:prior_k} with $f(k) = 1$ and $(L_{\kappa},U_{\kappa}) = (\hat{k^*}(\mD)/2, 2\hat{k^*}(\mD))$, where $\hat{k^*}(\mD) := \min \{k : \sum_{i=k}^p\hat{\lambda}_i > n\hat{\lambda}_i\}$. For the computation, we employed the Gibbs sampler to extract 10,000 posterior samples of $\bm{\theta}$ after a 50,000 burn-in period. As the predictor, we used the average of the posterior predictive samples, each of which is computed by the inner product of the covariates and a posterior sample of $\bm{\theta}$. We conducted 5-fold cross-validation (splitting the samples into five subsets, estimating parameters using four subsets, and making predictions on the remaining one) and quantified the difference between the observed outcomes $y$ and the predicted values using the root mean squared error (RMSE).

We randomly selected $100$ samples from the test data. Figure~\ref{fig:helpfull} then plots the true values, the posterior predictive mean, and the $95\%$ posterior predictive intervals. This demonstrates that the proposed predictor provides accurate predictions for the test data. The RMSE was 1.32, which is relatively small given that the mean and standard deviation of the CES-D scores in the dataset are 20.5 and 14.5, respectively. This result validates the usefulness of the proposed method in over-parameterized settings.

\begin{figure}[t]
  \begin{center}
  \includegraphics[width=\textwidth]{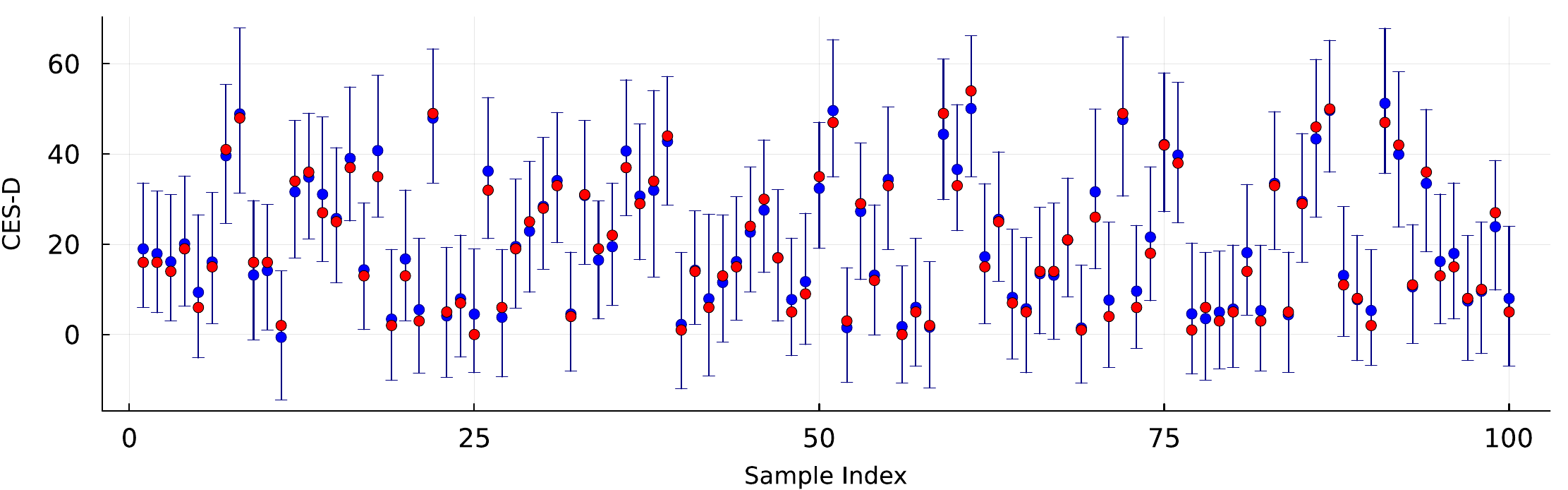}
  \end{center}
  \caption{Plot of the posterior predictive mean (blue points), 95\% posterior predictive interval (bars), and true values (red points) for 100 randomly selected samples from the test data.
  \label{fig:helpfull} }
\end{figure}

\section{Conclusion}\label{sec:dis}
This study focused on the problem of over-parameterized Bayesian linear regression. We departed from the traditional sparsity assumption and explored a non-sparse scenario. Under non-sparse conditions, we formulated a Gaussian prior distribution supported on a subspace reflecting the intrinsic dimensionality and proved the contraction of its posterior distribution. We also described the Gaussian approximation of the posterior distribution. These theoretical claims were further validated through numerical experiments. 

Finally, we highlight the significance of our study by detailing its contributions from two distinct perspectives.

\subsection{Modern Over-Parameterized Theory}

Our method enhances the utility of over-parameterized models from two perspectives: uncertainty evaluation and assumption relaxation.
In high-dimensional linear regression, a phenomenon known as the benign overfitting of the minimum norm interpolator is observed. This interpolator is defined as the estimator with the smallest norm among the solutions that make the quadratic loss zero. Analogously, in our approach, the likelihood function controls the quadratic loss, and the prior distribution controls the norm, resulting in posterior contractions to the true parameters. This enables the uncertainty quantification for the over-parameterized inference.
Furthermore, we contribute to overcoming the sub-Gaussianity restriction on covariates, which was imposed by earlier studies \citep{bartlett2020benign,tsigler2020benign,koehler2021uniform}. Although random variables with thin tails are more likely to achieve posterior contraction (benign overfitting), we prove that heavier-tailed distributions can also lead to posterior contraction. However, the contraction rate may be slow in certain settings, which is a concern. This issue is left for future studies and may be addressed using different proof strategies. 

\subsection{Bayesian High-Dimensional Theory}
The findings of this study advance the posterior contraction and Gaussian approximation theories in a high-dimensional Bayesian context.
While the asymptotic behaviour of high-dimensional Bayesian methods has been extensively studied, theoretical guarantees in over-parameterized regimes remain scarce and typically hinge on sparsity assumptions.
By contrast, this study proposes conditions and prior distributions that facilitate posterior contraction and Gaussian approximation for non-sparse high-dimensional problems by leveraging the intrinsically low dimensionality of the covariates. 
The model considered in this study is standard and does not accommodate flexible structures or non-iid situations. Nevertheless, owing to its simplicity, we successfully elucidated its theoretical properties. Moreover, because the investigation of more complex and adaptive methods often builds upon the analysis of simple methods, the insights gained from this study can serve as a foundation for further research.

\section*{Acknowledgements}
T.Wakayama was supported by JSPS KAKENHI (22J21090) and JST ACT-X (JPMJAX23CS).
M.Imaizumi was supported by JSPS KAKENHI (21K11780), JST CREST (JPMJCR21D2), and JST FOREST (JPMJFR216I).

\appendix

\section{Additional Simulation}\label{sec:exp_compare}
In this section, we focus on predictive risks and compare our approach with other non-sparse high-dimensional methods.
We consider four distinct data-generating scenarios. In each scenario, we fixed the sample size at $n = 600$ and varied the dimensionality $p$ among 300, 600, 1,200, 2,400, and 4,800. The response variable $y_i$ was generated using the linear model $ y_i \sim N(\bm{x}_i^{\top}\bm{\theta}^*, 1) $. The configurations for $\bm{\theta}^* $ and $ \bm{x}_i $ in each scenario are detailed as follows:
\begin{align}
    \mathrm{(i)}\ \bm{\theta}^*\sim N(\bm{0}, I_p),\quad \bm{x}_i\sim N(\bm{0}, \Sigma_{i});  \qquad\quad  
    \mathrm{(ii)}\ \bm{\theta}^*\sim N(\bm{0}, I_p),\quad \bm{x}_i\sim N(\bm{0}, \Sigma_{ii});
\end{align}
\begin{align}
    \mathrm{(iii)}\ \bm{\theta}^*\sim N(\bm{0}, \Sigma_{i}),\quad \bm{x}_i\sim N(\bm{0}, \Sigma_{i});  \qquad\quad  
    \mathrm{(iv)}\ \bm{\theta}^*\sim N(\bm{0}, \Sigma_{ii}),\quad \bm{x}_i\sim N(\bm{0}, \Sigma_{ii}),
\end{align}
where $\Sigma_{i}$ and $\Sigma_{ii}$ were generated as a matrix whose $j$th eigenvalue was $\exp(-j/2) + n\exp\{-n^{1/3}\}/p$ and $j^{-1}$, respectively, for $j=1,\ldots,p$. Notably, while the true parameter vector $\bm{\theta}^*$ and the covariate vectors $\bm{x}_i$ are orthogonal to each other in scenarios (i) and (ii), scenarios (iii) and (iv) consider cases where $\bm{\theta}^*$ and $\bm{x}_i$ are aligned, sharing the same covariance matrices.

We generated $25$ distinct datasets by the aforementioned procedure and analysed each dataset with four regression methods: the principal component regression (PCR) with the minimum number of components such that their cumulative explained variance exceeds 90\%; the minimum norm interpolator (MNI); the posterior mean of Bayesian linear regression with the standard normal prior; and the posterior mean of Bayesian linear regression with the proposed prior. We use the predictive risk to compare prediction performance.

\begin{figure}[t]
  \begin{center}
  \includegraphics[width=\textwidth]{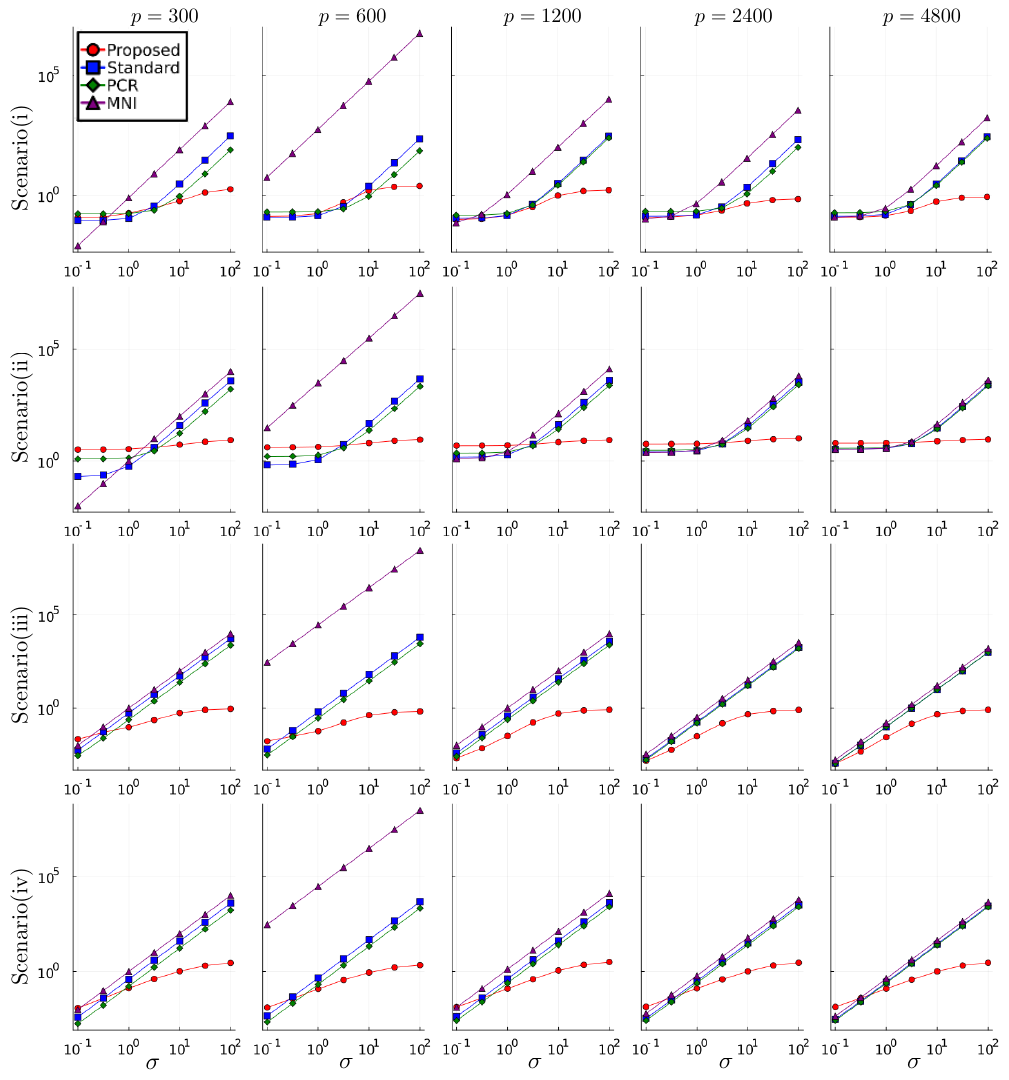}
  \end{center}
  \caption{Predictive risks of Bayesian linear regression with our proposed prior (circle), standard normal prior (square), the principal component regression (diamond) and minimum norm interpolator (triangle) for Scenarios (i)--(iv). \label{fig:sims} }
\end{figure}
Figure~\ref{fig:sims} shows the average predictive risks of each method, averaged over 25 datasets. In general, our proposed method exhibits lower predictive risks in higher-dimensional settings. While our method is not specifically designed for low-dimensional settings, unlike the other estimators, it still performs reasonably well when $p$ is relatively small. Notably, the proposed method also tends to excel in situations where the true parameter and the covariate are aligned. In Scenarios (i) and (iii), the proposed method is particularly effective in identifying and utilizing low-frequency components, which enhances its predictive performance. In contrast, in Scenarios (ii) and (iv), where eigenvalues decrease gradually, the other methods also work well.

\section{Examples of Setting of Theorem~\ref{thm:poscon}}

We present examples of $\Sigma$ and our prior that satisfy the theorem's assumptions and specify the effective variance and bias bounds.

\begin{example}
    Suppose the setup with the dimension $p=n^{1.6}$. Consider eigenvalues $\{\lambda_k\}_{k=1}^p$ taking the form
    \begin{equation}
        \lambda_k = \begin{cases}
            1 & 1\le i \le \frac{\sqrt{n}}{ (\log n)^2} \\
            \frac{1}{n^{2.5}} & \mathrm{otherwise}
        \end{cases}
    \end{equation}
    Then, if $L_{\kappa} = n^{1/2}/ (\log n)^2$, $U_{\kappa} = n^{1/2}/ \log n$ and $\pi_{\kappa}(L_{\kappa}) > c$ with some positive constant $c$, the following evaluation holds:
    \begin{align*}
        V_n = O\left( \frac{1}{ \sqrt{n}\log n }  \right) = B_n .
    \end{align*}
\end{example} 

\begin{example}
    Suppose the setup with the dimension $p=n^2$. Consider eigenvalues $\{\lambda_k\}_{k=1}^p$ of the following form
    \begin{equation}
        \lambda_k = \begin{cases}
            \frac{1}{\log\log(e+i)} & 1\le i \le \frac{n^{1/3}}{\log^2 n} \\
            \frac{1}{n^4} & \mathrm{otherwise}
        \end{cases}
    \end{equation}
    Then, if $L_{\kappa} = n^{1/3}/(\log n)^2$, $U_{\kappa} =  n^{1/3}/(\log\log n)^2$ and $\pi_{\kappa}(L_{\kappa}) > c$ with some positive constant $c$, the following evaluation holds:
    \begin{align*}
        V_n = O\left( \frac{1}{ (\log n)^5 }  \right)   \quad \mathrm{and} \quad  B_n  = O\left( \frac{(\log n)^2}{n} \right).
    \end{align*}
\end{example} 

\section{Proofs of Theorem~\ref{thm:posext}} \label{app:poscon}
In this section, we present three lemmas essential for establishing Theorem~\ref{thm:posext}, accompanied by their respective proofs.

\subsection{Lemmas for Inequalities} 

\begin{lemma} \label{lem:bound1}
    For any $\varepsilon>0$ and some $\lambda\ge 0$, we have the following for any sufficiently large $n \in \N$:
    \begin{align}
       \frac1n\log \mN(\varepsilon, \mP_{n,\varepsilon, \lambda}, d_{\mP}) \lesssim 
      &\frac1n \log\left( \frac{n}{2\varepsilon}\right) +
    \sum_{i\in\bar{\mI} } \frac1n \log \left( \frac{n\sqrt{\lambda_i} \sqrt{\sum_{i>L_{\kappa}} \lambda_{i}^2}}{  (\lambda + \sum_{i>L_{\kappa}} \lambda_{i}) \sqrt{\sum_{i>L_{\kappa}} \lambda_{i}} } \right) + \frac{U_{\kappa}}{n} \\
    &+ \frac{n \sum_{i>L_{\kappa}} \lambda_{i}^2 }{(\lambda + \sum_{i>L_{\kappa}} \lambda_{i})^2},
    \end{align}
    where $\bar{\mI} := \{ i=1,\ldots,U_{\kappa} : n  ({\lambda_i \sum_{i>L_{\kappa}} \lambda_{i}^2})^{1/2} > (\sum_{i>L_{\kappa}} \lambda_{i})^{1/2} (\lambda + \sum_{i>L_{\kappa}} \lambda_{i}) \}$.
\end{lemma}

\begin{lemma} \label{lem:bound2}
    If Assumption~\ref{ass:emp} with $\rho_n = o(\lambda_{L_{\kappa}+1})$ holds, for any $\varepsilon>0$ and some $\lambda\ge 0$, we have the following for any sufficiently large $n \in \N$, with probability at least $1-\ell_n$:
    \begin{align}
        &\Pi_{\vartheta,\varsigma^2}(\mP\setminus \mP_{n,\varepsilon, \lambda} \mid \mD_1)  \\
        &\leq  \exp\left(-n C_{\eta,\xi}\right) +\\
        & \sum_{k =L_{\kappa}}^{U_{\kappa}} \pi_{\kappa}(k) \exp\left( 4k - \frac{1}{16\lambda_1} \frac{n^2 \varepsilon^2 \sum_{i>L_{\kappa}} \lambda_{i}^2 }{ \sum_{i>L_{\kappa}} \lambda_{i}  (\lambda + \sum_{i>L_{\kappa}} \lambda_{i})^2}\right)
        \mone\left\{ \frac{n^2 \varepsilon^2 (\lambda_k -\lambda_1\rho_n) \sum_{i>L_{\kappa}} \lambda_{i}^2 }{ \sum_{i>L_{\kappa}} \lambda_{i}  (\lambda + \sum_{i>L_{\kappa}} \lambda_{i})^2} \le R^2 \right\} ,
    \end{align}  
    where $C_{\eta,\xi}$ is a constant factor depending upon inverse-Gaussian parameters $\eta$ and $\xi$.
\end{lemma}

\begin{lemma} \label{lem:bound3}
    Suppose that Assumption~\ref{ass:emp} with $\rho_n = o(\lambda_{L_{\kappa}+1})$, $\| \bm{\theta}^* \|_{\Sigma} <R/2$ and $\| \bm{\theta}^* \|_2 < \infty$ hold.
    For a sequence $\{\varepsilon_n\}_{n \in \N}$ such that $\varepsilon_n\to 0$ and $\tilde{\varepsilon}_n^2 := {\varepsilon}_n^2 - \sum_{i=L_{\kappa}+1}^p ({\lambda}_i + 2\lambda_1\rho_n)\beta_i^{*2} >0$, where $\bm{\beta}^* := (\hat{\bm{v}}_1,\ldots,\hat{\bm{v}}_p)^{\top} \bm{\theta}^*$,
    we have the following with sufficiently large $n \in \N$, with probability at least $1-\ell_n$:
    \begin{align}
        -\frac{1}{n}\log \Pi_{\vartheta,\varsigma^2} (\mB_{KL}(\varepsilon_n)\mid \mD_1) \le \frac{1}{n}\log\frac{1}{\varepsilon_n\pi_{\kappa}(L_{\kappa})} + \frac1n\sum_{i=1}^{L_{\kappa}} \frac{\beta_i^{*2}}{{\lambda}_i} + \frac{L_{\kappa}}{n} \log \left(\frac{{ 2 L_{\kappa} \lambda_1}}{\tilde{\varepsilon}_n}\right).
    \end{align}
\end{lemma}

\subsection{Proofs of Lemmas~\ref{lem:bound1}-\ref{lem:bound3} } 
\begin{proof}[of Lemma~\ref{lem:bound1}]
For notational brevity, let $H_{n,\varepsilon, \lambda}^2 = n^2 \varepsilon^2 \sum_{i>L_{\kappa}} \lambda_{i}^2 / \sum_{i>L_{\kappa}} \lambda_{i}  (\lambda + \sum_{i>L_{\kappa}} \lambda_{i})^2$.
At the beginning, we decompose the covering number in \eqref{ggv-thm2.1-1} and obtain
\begin{align}
    \log \mN(\varepsilon, \mP_{n,\varepsilon, \lambda}, d_{\mP}) \le \log \mN(\varepsilon,\mP_{n}^2,|\cdot|)+\log \mN(\varepsilon,\mP_{n,\varepsilon,\lambda}^1,\|\cdot\|_{\Sigma}). \label{ineq:bound1_1}
\end{align}
The first term of the right-hand side of \eqref{ineq:bound1_1} is bounded by $\log ( n/2\varepsilon)$, because the length of the interval $\mP_{n}^2$ is less than $n$ and $n/2\varepsilon$ open intervals of length $2\varepsilon$ are sufficient to cover it.

We focus mainly on the second term on the right-hand side of \eqref{ineq:bound1_1}.
From Lemma~\ref{lem:cov}, we obtain 
\begin{align} 
    \log \mN(\varepsilon,\mP_{n,\varepsilon,\lambda}^1,\|\cdot\|_{\Sigma})
    &= \log \mN(\varepsilon,\Sigma^{1/2}\mP_{n,\varepsilon,\lambda}^1,\|\cdot\|_{2})\\
    &\leq \log \mN(\varepsilon/2,\Sigma_{1:k}^{1/2}\mP_{n,\varepsilon,\lambda}^1,\|\cdot\|_{2})+\log \mN(\varepsilon/2,\Sigma_{k+1:p}^{1/2}\mP_{n,\varepsilon,\lambda}^1,\|\cdot\|_{2})\\
    &= \log \underbrace{\mN(\varepsilon/2,\mP_{n,\varepsilon,\lambda}^1,\|\cdot\|_{\Sigma_{1:k}})}_{=:\mN_1}+\log \underbrace{\mN(\varepsilon/2,\Sigma_{k+1:p}^{1/2}\mP_{n,\varepsilon,\lambda}^1,\|\cdot\|_{2})}_{=:\mN_2},
\end{align}
for any $k \in [1:p-1]$.
For algebraic operations of the covering number, refer to Chapter 7.1 of \citep{temlyakov2018multivariate}. To simplify the notation, we substitute $\varepsilon/2$ into $\varepsilon$. This substitution does not affect the following results.

In the following, we will bound $\mN_1$ and $\mN_2$ separately.

\textbf{Step (i): Bound $\mN_1$}. 
In this step, we derive an upper bound for $\log \mN_1$. We define the orthogonal matrix $V = (\bm{v}_1,\ldots\bm{v}_p)\in \R^{p \times p}$, where $\bm{v}_j$ denotes the eigenvectors defined in Section \ref{sec:pri}. For any $\bm{h}_1,\bm{h}_2\in \mathbb{R}^p$, it holds that
\begin{align}
    (\bm{h}_1-\bm{h}_2)^{\top}\Sigma_{1:k}(\bm{h}_1-\bm{h}_2) 
    &= (\bm{h}_1-\bm{h}_2)^{\top}V\mathrm{diag}(\lambda_1,\ldots\lambda_{k},0\ldots,0)V^{\top}(\bm{h}_1-\bm{h}_2)\\
    &= (\bm{h}_1'-\bm{h}_2')^{\top}\mathrm{diag}(\lambda_1,\ldots\lambda_{k},0\ldots,0)(\bm{h}_1'-\bm{h}_2'), \ \ \ \ \bm{h}_1',\bm{h}_2'\in V^{\top}\mathbb{R}^p.
\end{align}
This implies that if the first $k$ components of two elements $\bm{h}_1',\bm{h}_2'$ in $V^{\top}\mathbb{R}^p(=\mathbb{R}^p)$ are the same, the distance $(\bm{h}_1'-\bm{h}_2')^{\top}\mathrm{diag}(\lambda_1,\ldots\lambda_{k},0\ldots,0)(\bm{h}_1'-\bm{h}_2')$ is zero. As the $p-k$ components of $V^{\top}\mathbb{R}^p(=\mathbb{R}^p)$ can be ignored, we have
\begin{align}
    \mN_1 &= \mN(\varepsilon,\{\bm{\theta}\in\mathbb{R}^{k}:\|\bm{\theta}\|_2\leq H_{n,\varepsilon, \lambda}\},\|\cdot\|_{\check{\Sigma}_{1:k}}) \\
    & = \mN(\varepsilon,\{\check{\Sigma}_{1:k}^{1/2}\bm{\theta}\in\mathbb{R}^{k}:\|\bm{\theta}\|_2\leq H_{n,\varepsilon, \lambda}\},\|\cdot\|_{2}) \\
    & = \mN(\varepsilon,\{\bm{\theta}\in\mathbb{R}^{k}:\|\check{\Sigma}_{1:k}^{-1/2}\bm{\theta}\|_2\leq H_{n,\varepsilon, \lambda}\},\|\cdot\|_{2}) \\
    & = \mN(1,\{\bm{\theta}\in\mathbb{R}^{k}:\|\check{\Sigma}_{1:k}^{-1/2}\bm{\theta}\|_2\leq H_{n,\varepsilon, \lambda}/\varepsilon\},\|\cdot\|_{2}) 
    \label{ineq:bound1_2}
\end{align}
where $\check{\Sigma}_{1:k} = \mathrm{diag}(\lambda_1,\ldots\lambda_{k})$. Then, it is sufficient to calculate the covering number of the following ellipsoid with unit Euclidean balls:
\begin{equation} 
\left\{\bm{\theta} = (\theta_1,\theta_2,\ldots,\theta_k)^{\top}\in\mathbb{R}^{k}: \sum_{i=1}^k \frac{\theta_i^2}{(\sqrt{\lambda_i} H_{n,\varepsilon, \lambda} /\varepsilon )^2} \leq 1 \right\}.
\end{equation}
Let $\bar{\mI} := \{ i=1,\ldots,k : \sqrt{\lambda_i} H_{n,\varepsilon, \lambda} /\varepsilon > 1 \}$. Using Theorem~2 in \cite{DUMER20061667}, we obtain
\begin{equation}\label{ineq:bound1_3}
    \log \mN_1 \lesssim \sum_{i\in\bar{\mI} } \log \left( \sqrt{\lambda_i} H_{n,\varepsilon, \lambda}/\varepsilon \right) + k.
\end{equation}

\textbf{Step (ii): Bound $\mN_2$}.
Next, we consider the upper bound of $\mN_2$. From the Sudakov minoration, we have
\begin{align}
 \log \mN_2 = \log \mN(\varepsilon, \Sigma_{k+1:p}^{1/2}\mP_{n,\varepsilon,\lambda}^1,\|\cdot\|_2)
    \leq \varepsilon^{-2} \Ep\left[\sup_{\bm{h}\in \Sigma_{k+1:p}^{1/2}\mP_{n,\varepsilon,\lambda}^1} \bm{h}^{\top} \bm{Z} \right]^2, \label{ineq:bound1_4}
\end{align}
where $\bm{Z}\sim N(\bm{0}, I_{p})$ is an independent random variable. We define $\bm{Z}':=\Sigma_{k+1:p}^{1/2}\bm{Z}$, where $B_2^p$ is a unit ball in $\R^p$ with respect to $\|\cdot\|_2$. Then, we obtain the following relations:
\begin{align}
    \Ep\left[ \sup_{\bm{h}\in \Sigma_{k+1:p}^{1/2}\mP_{n,\varepsilon,\lambda}^1} \bm{h}^{\top} \bm{Z} \right] 
    &\leq  \Ep\left[ \sup_{\bm{h}\in \mP_{n,\varepsilon,\lambda}^1} \bm{h}^{\top} \bm{Z}' \right]=H_{n,\varepsilon, \lambda} \Ep \left[\sup_{\bm{h}\in B_2^p} \bm{h}^{\top} \bm{Z}'\right] \\
    &=H_{n,\varepsilon, \lambda}\Ep[\|\bm{Z}'\|_2] \leq H_{n,\varepsilon, \lambda}\sqrt{\Ep[ \|\bm{Z}'\|_2^2]}.
\end{align}
From $\Ep[ \|\bm{Z}'\|_2^2] = \sum_{j=k+1}^{p} \lambda_{j}$ and \eqref{ineq:bound1_4}, we have
\begin{align}
    \log \mN_2 = \log \mN(\varepsilon,\mP_{n,\varepsilon,\lambda}^1,\|\cdot\|_{\Sigma_{k+1:p}})\leq \frac{H_{n,\varepsilon, \lambda}^2}{\varepsilon^2} \sum_{j=k+1}^p \lambda_{j} . \label{ineq:bound1_5}
\end{align}

\textbf{Step (iii): Integrate the results}.
We simply combine the bounds \eqref{ineq:bound1_3} and \eqref{ineq:bound1_5} into the first inequality \eqref{ineq:bound1_1} to obtain the statement.
\end{proof}

\begin{proof}[of Lemma~\ref{lem:bound2}]
We start by decomposing the prior distribution as
\begin{align}
    &\Pi_{\vartheta,\varsigma^2}(\mP\setminus \mP_{n,\varepsilon, \lambda}\mid\mD_1) \\
    &\le \sum_{k =L_{\kappa}}^{U_{\kappa}} \pi_{\kappa}(k)\left\{\Pi_{\vartheta}(\{ \bm{\theta} : \|\bm{\theta}\|_2\ge H_{n,\varepsilon, \lambda} \}\mid \mD_1,k) + \Pi_{\varsigma^2}(\{ \sigma^2 : \sigma^2 \le n^{-1} \}) + \Pi_{\varsigma^2}(\{ \sigma^2 : \sigma^2 \ge n\})\right\}\\
    &= \sum_{k =L_{\kappa}}^{U_{\kappa}} \pi_{\kappa}(k)\Pi_{\vartheta}(\{ \bm{\theta} : \|\bm{\theta}\|_2\ge H_{n,\varepsilon, \lambda} \}\mid \mD_1,k) + \Pi_{\varsigma^2}(\{ \sigma^2 : \sigma^2 \le n^{-1} \}) + \Pi_{\varsigma^2}(\{ \sigma^2 : \sigma^2 \ge n\}). \label{ineq:bound2_1}
\end{align}
Regarding the last two terms on the right-hand side, as the prior for $\sigma^2$ is inverse-Gaussian, the tail decay is exponentially fast, and $\Pi_{\varsigma^2}(\{ \sigma^2 : \sigma^2 \le n^{-1} \})$ and $\Pi_{\varsigma^2}(\{ \sigma^2 : \sigma^2 \ge n\})$ are bounded by $e^{-nC_{\eta,\xi}}$ with a constant $C_{\eta,\xi}$ depending on $\eta$ and $\xi$. 

Next, we rewrite the first $U_{\kappa}-L_{\kappa}+1$ terms of \eqref{ineq:bound2_1} as
\begin{equation}
    \sum_{k =L_{\kappa}}^{U_{\kappa}} \pi_{\kappa}(k)\Pi_{\mathfrak{b}} (\{\bm{\beta}_{1:k} :  \|\bm{\beta}_{1:k}\|_2>H_{n,\varepsilon, \lambda} \} \mid \mD_1, k). \label{ineq:bound2_2} 
\end{equation}
We fix $k$ and examine $\Pi_{\mathfrak{b}} (\{\bm{\beta}_{1:k} : \|\bm{\beta}_{1:k}\|_2>H_{n,\varepsilon, \lambda}\}\mid \mD_1, k)$.
Remark that this probability is zero if $\{\bm{\beta}_{1:k}:\|\bm{\beta}_{1:k}\|_2>H_{n,\varepsilon, \lambda}\} \cap \{\bm{\beta}_{1:k}:\|\tilde{\Sigma}_{1:k}^{1/2}\bm{\beta}_{1:k}\|_2 \le R \}=\varnothing$, which holds with probability at least $1-\ell_n$ under the condition $H_{n,\varepsilon, \lambda} (\lambda_k -\lambda_1\rho_n)^{1/2} > R$. Indeed, if $\|\bm{\beta}_{1:k}\|_2>H_{n,\varepsilon, \lambda}$ and $H_{n,\varepsilon, \lambda} (\lambda_k -\lambda_1\rho_n)^{1/2} > R$, $\|\tilde{\Sigma}_{1:k}^{1/2}\bm{\beta}_{1:k}\|_2 > \hat{\lambda}_k^{1/2} H_{n,\varepsilon, \lambda} > (\lambda_k -\lambda_1\rho_n)^{1/2} H_{n,\varepsilon, \lambda} > R$, where the second inequality follows from Lemma~\ref{lem:eigen-error}. Hence, we develop an upper bound of $\Pi_{\mathfrak{b}} (\{\bm{\beta}_{1:k} : \|\bm{\beta}_{1:k}\|_2>H_{n,\varepsilon, \lambda}\}\mid \mD_1, k)$ under the condition $H_{n,\varepsilon, \lambda} (\lambda_k -\lambda_1\rho_n)^{1/2} \le R$ in this paragraph. We define a Gaussian measure $\check{\Pi}_{\mathfrak{b}} (\cdot) := N(\bm{0},\mathrm{diag}(\hat{\lambda}_1,\ldots\hat{\lambda}_{k}))(\cdot)$, which is an untruncated version of $\Pi_{\mathfrak{b}}(\cdot)$. The truncated Gaussian measure $\Pi_{\mathfrak{b}}$ has a larger mass on a centred set $\{\bm{\beta}_{1:k}\in\R^k : \|\bm{\beta}_{1:k}\|_2 \leq H_{n,\varepsilon, \lambda}\}$ than the untruncated Gaussian measure $\check{\Pi}_{\mathfrak{b}}$, because we have, from the Gaussian correlation inequality~\citep{Li1999GaussCorIneq}, 
\begin{align*}
    & \Pi_{\mathfrak{b}}(\{\bm{\beta}_{1:k}\in\R^k :  \|\bm{\beta}_{1:k}\|_2 \le H_{n,\varepsilon, \lambda}\}\mid \mD_1, k) \\
    & = \frac{\check{\Pi}_{\mathfrak{b}}(\{\bm{\beta}_{1:k}:\|\tilde{\Sigma}_{1:k}^{1/2}\bm{\beta}_{1:k}\|_2 \le R \} \cap \{\bm{\beta}_{1:k}\in\R^k :  \|\bm{\beta}_{1:k}\|_2 \le H_{n,\varepsilon, \lambda}\}\mid \mD_1, k) }{ \check{\Pi}_{\mathfrak{b}}(\{\bm{\beta}_{1:k}:\|\tilde{\Sigma}_{1:k}^{1/2}\bm{\beta}_{1:k}\|_2 \le R \} \mid \mD_1, k)} \\
    & \geq \check{\Pi}_{\mathfrak{b}}(\{\bm{\beta}_{1:k}\in\R^k :  \|\bm{\beta}_{1:k}\|_2 \le H_{n,\varepsilon, \lambda}\}\mid \mD_1, k) .
\end{align*} 
Hence, we obtain
\begin{align}
 \Pi_{\mathfrak{b}}(\{\bm{\beta}_{1:k}\in\R^k :  \|\bm{\beta}_{1:k}\|_2 > H_{n,\varepsilon, \lambda}\}\mid \mD_1, k) 
 &\leq \check{\Pi}_{\mathfrak{b}}(\{\bm{\beta}_{1:k}\in\R^k :  \|\bm{\beta}_{1:k}\|_2 > H_{n,\varepsilon, \lambda}\}\mid \mD_1, k) \\
 &\leq \check{\Pi}_{\mathfrak{b}}\left(\left\{ \bm{\beta}_{1:k}\in\R^k : \max_{\bm{z}\in N_{1/2}}\bm{z}^{\top}\bm{\beta}_{1:k}> H_{n,\varepsilon, \lambda}/2 \right\}\mid \mD_1, k\right)\\
 &\leq |N_{1/2}|\max_{\bm{z}\in N_{1/2}} \check{\Pi}_{\mathfrak{b}}(\{ \bm{\beta}_{1:k} : \bm{z}^{\top}\bm{\beta}_{1:k}> H_{n,\varepsilon, \lambda}/2\}\mid \mD_1, k ) \\
 &\leq  |N_{1/2}|\exp\left(-\frac{H_{n,\varepsilon, \lambda}^{2}}{8\hat{\lambda}_1}\right),\label{eqn:FuH_}
\end{align}
where $N_{1/2} (\subset B_2^k)$ denotes the $1/2$-minimal covering set of the unit ball $B_2^k \subset \R^k$. The second inequality holds because Lemma~\ref{lem:l2ball} shows $\|\bm{\beta}_{1:k}\|_2\leq 2\max_{\bm{z}\in N_{1/2}}\bm{z}^{\top}\bm{\beta}_{1:k}$.
The last inequality holds via the Chernoff bound with the sub-Gaussian $\bm{z}^{\top}\bm{\beta}_{1:k}$. 
Let $\bm{Z}\sim N(\bm{0},I_{k})$. From the Sudakov minoration as \eqref{ineq:bound1_4} and Jensen's inequality, we have
\begin{align} 
 |N_{1/2}|
 \leq \exp\left(2^2\Ep\left[\sup_{\bm{h}\in B_2^k}\bm{h}^{\top}\bm{Z} \right]^2\right)
 \leq \exp(4\Ep[\|\bm{Z}\|_2]^2)
 \leq \exp(4\Ep[\|\bm{Z}\|^2_2])
 =\exp(4k).\label{ineq:Fcover_}
\end{align}
From \eqref{eqn:FuH_} and \eqref{ineq:Fcover_}, we obtain
$\Pi_{\mathfrak{b}}(\{\bm{\beta}_{1:k}\in\R^k :  \|\bm{\beta}_{1:k}\|_2 > H_{n,\varepsilon, \lambda}\}\mid \mD_1, k)\le \exp(4k-{H_{n,\varepsilon, \lambda}^2}/(8\hat{\lambda}_1))$.

We substitute this result into \eqref{ineq:bound2_2} and obtain the following with probability at least $1-\ell_n$:
\begin{align}
    &\sum_{k =L_{\kappa}}^{U_{\kappa}} \pi_{\kappa}(k)\Pi_{\mathfrak{b}} (\{\bm{\beta}_{1:k} : \|\bm{\beta}_{1:k}\|_2>H_{n,\varepsilon, \lambda} \} \mid \mD_1, k) \nonumber \\
    &\le \sum_{k =L_{\kappa}}^{U_{\kappa}} \pi_{\kappa}(k)\exp\left(4k-\frac{H_{n,\varepsilon, \lambda}^{2}}{8\hat{\lambda}_1}\right) \mone\left\{H_{n,\varepsilon, \lambda} (\lambda_k -\lambda_1\rho_n)^{1/2} \le R\right\} \\
    &= \sum_{k =L_{\kappa}}^{U_{\kappa}} \pi_{\kappa}(k)\exp\left(4k-\frac{\lambda_1}{\hat{\lambda}_1}\frac{H_{n,\varepsilon, \lambda}^{2}}{8\lambda_1}\right)\mone\left\{H_{n,\varepsilon, \lambda} (\lambda_k -\lambda_1\rho_n)^{1/2} \le R\right\}. \label{ineq-tail_}
\end{align}
By Assumption~\ref{ass:emp}, we have  
\begin{align}
    \left|\frac{\hat{\lambda}_1}{\lambda_1}-1\right| = \frac{|\hat{\lambda}_1-\lambda_1|}{\lambda_1} \leq \frac{\|\hat{\Sigma}-\Sigma\|_{\mathrm{op}}}{\|\Sigma\|_{\mathrm{op}}} \leq \rho_n = o(1)
\end{align}
with probability at least $1-\ell_n$.
We apply this result to \eqref{ineq-tail_} and obtain the following with probability at least $1-\ell_n$:
\begin{align}
    &\sum_{k =L_{\kappa}}^{U_{\kappa}} \pi_{\kappa}(k)\Pi_{\mathfrak{b}} (\{ \bm{\beta}_{1:k} : \|\bm{\beta}_{1:k}\|_2>H_{n,\varepsilon, \lambda} \} \mid \mD_1, k) \\
    &\le \sum_{k =L_{\kappa}}^{U_{\kappa}} \pi_{\kappa}(k) \exp\left(4k- \frac{1}{(1+\rho_n)}\frac{H_{n,\varepsilon, \lambda}^{2}}{8\lambda_1}\right)\mone\left\{H_{n,\varepsilon, \lambda} (\lambda_k -\lambda_1\rho_n)^{1/2} \le R \right\}.
\end{align}  
Then, we obtain the statement. 
\end{proof}

\begin{proof}[of Lemma~\ref{lem:bound3}]
Following~\cite{ning2020bayesian,ghosal2017fundamentals}, the KL divergence $\KL(P^*, P_{\bm{\theta}, \sigma^2})$ is written as
\begin{align}
    \KL(P^*, P_{\bm{\theta}, \sigma^2}) 
    &= \int \log \left( \frac{\mathrm{d}P^*(y,\bm{x})}{\mathrm{d}P_{\bm{\theta}, \sigma^2}(y,\bm{x})}\right) \mathrm{d}P^*(y,\bm{x})\\
    &= \frac12\frac{(\sigma^*)^2}{\sigma^2} - \frac{1}{2}\log \left\{\frac{(\sigma^*)^2}{\sigma^2}\right\} -\frac12+ \frac12\frac{ {\|\bm{\theta}^* - \bm{\theta}\|_{\Sigma}^2} }{\sigma^2}.
\end{align}
Similarly, the KL variation $\KV(P^*, P_{\bm{\theta}, \sigma^2}) $ is 
\begin{align}
    \KV(P^*, P_{\bm{\theta}, \sigma^2}) &= \frac12\frac{(\sigma^*)^4}{\sigma^4} - \frac{(\sigma^*)^2}{\sigma^2} +\frac12+ \frac{(\sigma^*)^2}{\sigma^2}\frac{\|\bm{\theta}^* - \bm{\theta}\|^2_{\Sigma}}{\sigma^2}.
\end{align}
Hence, using the subsets $\mB_1 \subset \R_+$ and $ \mB_2 \subset \R^p \times \R_+$ defined in \eqref{def:B1}~and\eqref{def:B2}, it holds that
\begin{align}
    \Pi_{\vartheta,\varsigma^2} \left(\left\{(\bm{\theta},\sigma^2): \KL(P^*, P_{\bm{\theta}, \sigma^2}) \leq \varepsilon_n^2, 
	\KV(P^*, P_{\bm{\theta}, \sigma^2}) \leq \varepsilon_n^2 \right\} \mid \mD_1\right)
    \ge
    \Pi_{\vartheta,\varsigma^2} \left(\mB_2\mid \mB_1 , \mD_1\right)\Pi_{\varsigma^2} \left(\mB_1\right). \label{eq:Pi_B1_B2}
\end{align}
In the following, we present lower bounds on  $\Pi_{\varsigma^2} \left(\mB_1\right)$ and $\Pi_{\vartheta,\varsigma^2} (\mB_2\mid \mB_1, \mD_1)$.

\textbf{Step (i): Bound $\Pi_{\varsigma^2} \left(\mB_1\right)$}. First, we discuss a bound on $\Pi_{\varsigma^2} \left(\mB_1\right)$. As the logarithm has Maclaurin series $\log(1+x)= x-\frac{x^2}{2}+o(x^2)$ for $|x|<1$, we obtain
\begin{align}
    &\frac{(\sigma^*)^2}{\sigma^2}-\log\left( 1 + \frac{(\sigma^*)^2}{\sigma^2} -1\right) -1 \\
    &= \frac{(\sigma^*)^2}{\sigma^2}- \left\{\frac{(\sigma^*)^2}{\sigma^2} -1 - \frac12 \left(\frac{(\sigma^*)^2}{\sigma^2} -1\right)^2\right\}-1+o\left(\left(\frac{(\sigma^*)^2}{\sigma^2}-1\right)^2\right) \\
    &=  \frac12 \left(\frac{(\sigma^*)^4}{\sigma^4}-\frac{2(\sigma^*)^2}{\sigma^2} +1\right)+o\left(\left(\frac{(\sigma^*)^2}{\sigma^2}-1\right)^2\right).
\end{align}
We introduce a symbol $\Delta = {(\sigma^*)^2}/{\sigma^2}-1$. Then, we have $\Pi_{\varsigma^2}(\mB_1) = \Pi_{\varsigma^2}(\{ \sigma^2 :  \Delta^2 / 2 + o(\Delta^2) \leq \varepsilon_n^2 ,\Delta^2 \leq \varepsilon_n^2\} )$. The second inequality in the measure $\Delta^2 \leq \varepsilon_n^2 (\to 0)$ implies $\Delta^2 / 2 \leq \varepsilon_n^2/2$ and $o(\Delta^2) \leq \varepsilon_n^2/2$ with sufficiently large $n$. Thus, $\Pi_{\varsigma^2}(\mB_1) = \Pi_{\varsigma^2}(\{ \sigma^2 : \Delta \leq \varepsilon_n\})$ holds with such $n$.
Because the prior distribution of $\sigma^2$ is an inverse-Gaussian distribution, we consider the transform $\tilde{\sigma}=\sigma^{-2}$ and obtain
\begin{align}
     &\Pi_{\varsigma^2}(\Delta \leq \varepsilon_n)\\
     & = \Pi_{\varsigma^2} \left( \left\{\sigma^2: \left|{1}/{\sigma^2}-{1}/{(\sigma^* )^2}\right|\leq \frac{\varepsilon_n}{(\sigma^*)^2} \right\} \right) \notag  \\
    &\ge \int_{(\sigma^*)^{-2}-(\sigma^*)^{-2}\varepsilon_n}^{(\sigma^*)^{-2}+(\sigma^*)^{-2}\varepsilon_n} \tilde{\sigma}^{-1 / 2}e^{-\tilde{\sigma}}e^{-{1} / {\tilde{\sigma}}}\mathrm{d} \tilde{\sigma} \nonumber \\
    &\ge 2(\sigma^*)^{-2}\varepsilon_n \left ((\sigma^*)^{-2}+(\sigma^*)^{-2}\varepsilon_n\right)^{-1 / 2} \exp\left(-(\sigma^*)^{-2}-(\sigma^*)^{-2}\varepsilon_n\right) \exp\left(\frac{-1}{(\sigma^*)^{-2}-(\sigma^*)^{-2}\varepsilon_n}\right) \nonumber \\
    &\gtrsim \varepsilon_n. \label{ineq:bound3_2}
\end{align}
The first inequality follows from the transformation $\tilde{\sigma}=\sigma^{-2}$.
In the second inequality, $2({\sigma^*})^{-2}\varepsilon_n$ represents the integral interval length, and the other terms represent the minimum values of the integrands within the integral interval.

\textbf{Step (ii): Bound $\Pi_{\vartheta,\varsigma^2} (\mB_2\mid \mB_1, \mD_1)$}.
In the previous step, we showed that $\sigma$ is close to the constant $\sigma^*$ in $\Pi_{\varsigma^2}$ within the set $\mB_1$. To bound the probability $\Pi_{\vartheta,\varsigma^2}(\mB_2\mid\mB_1, \mD_1)$, it is sufficient to evaluate
\begin{equation}
    \sum_{k =L_{\kappa}}^{U_{\kappa}} \pi_{\kappa}(k)\Pi_{\vartheta} \left(\left\{\bm{\theta}: \| \bm{\theta}-\bm{\theta}^*\|_{\Sigma}^2 \geq \varepsilon_n^2 \right\} \mid \mD_1, k \right). \label{ineq:bound3_1}
\end{equation}
As we have $\|\bm{\theta}-\bm{\theta}^*\|_{\Sigma}^2 = (\bm{\theta}-\bm{\theta}^*)^{\top} (\Sigma-\hat{\Sigma}) (\bm{\theta}-\bm{\theta}^*) + (\bm{\theta}-\bm{\theta}^*)^{\top} \hat{\Sigma} (\bm{\theta}-\bm{\theta}^*) \leq (\bm{\theta}-\bm{\theta}^*)^{\top} (\lambda_1\rho_nI+\hat{\Sigma}) (\bm{\theta}-\bm{\theta}^*)$ from Assumption~\ref{ass:emp}, for any $k = L_{\kappa},\ldots,U_{\kappa}$ and $\bm{\theta}\in S_{\mD_1,k}$, it holds that
\begin{align}
     \|\bm{\theta}-\bm{\theta}^*\|_{\Sigma}^2 
     &\le (\bm{\beta}_{1:k}- \bm{\beta}_{1:k}^*)^{\top} \mathrm{diag}(\lambda_1\rho_n + \hat{\lambda}_1,\lambda_1\rho_n + \hat{\lambda}_2\ldots, \lambda_1\rho_n+ \hat{\lambda}_{k} )  (\bm{\beta}_{1:k}- \bm{\beta}_{1:k}^*) \\
     &\quad + \bm{\beta}_{k+1:p}^{*^\top} \mathrm{diag}(\lambda_1\rho_n + \hat{\lambda}_{k+1},\ldots, \lambda_1\rho_n+ \hat{\lambda}_{p} ) \bm{\beta}_{k+1:p}^*,
\end{align}
where $\bm{\beta}_{1:k}^*$ is the $k$-dimensional vector composed of the first $k$ elements of $\bm{\beta}^* = (\hat{\bm{v}}_1,\ldots,\hat{\bm{v}}_p)^{\top} \bm{\theta}^*$ and $\bm{\beta}_{k+1:p}^*$ is the remaining part.
Hence, we bound the probability in \eqref{ineq:bound3_1} as
\begin{equation}
    \Pi_{\vartheta} \left(\left\{\bm{\theta}: \| \bm{\theta}-\bm{\theta}^*\|_{\Sigma}^2 \geq \varepsilon_n^2 \right\} \mid \mD_1, k \right) 
    \le \Pi_{\mathfrak{b}} \left(\left\{\bm{\beta}_{1:k}: \left\| \bm{\beta}_{1:k} - \bm{\beta}_{1:k}^* \right\|_{\grave{\Sigma}_{1:k}}^2 +  \sum_{i=k+1}^p \tilde{\lambda}_i\beta_i^* \geq \varepsilon_n^2 \right\} \mid \mD_1, k \right), \label{ineq:bound3_4}
\end{equation}
where $\grave{\Sigma}_{1:k} := \mathrm{diag}(\tilde{\lambda}_1,\ldots,\tilde{\lambda}_k )$ and $\tilde{\lambda}_i = \lambda_1\rho_n + \hat{\lambda}_i$ for $i=1,2,\ldots,p$. Furthermore, it is bounded from above by
\begin{equation}
    \Pi_{\mathfrak{b}} \left(\left\{\bm{\beta}_{1:k}: \|\bm{\beta}_{1:k}-\bm{\beta}_{1:k}^* \|_{\grave{\Sigma}_{1:k}}^2 \geq \tilde{\varepsilon}_n^2 \right\}\mid \mD_1, k\right),\label{ineq:bound3_5}
\end{equation}
where $\tilde{\varepsilon}_n^2 = {\varepsilon}_n^2 - \sum_{i=k+1}^p (\lambda_i + 2\lambda_1\rho_n) \beta_i^{*2}$, which is less than ${\varepsilon}_n^2 - \sum_{i=k+1}^p \tilde{\lambda}_{i} \beta_i^{*2}$ from Lemma~\ref{lem:eigen-error}.

We develop an upper bound of \eqref{ineq:bound3_5}. With the normalizing constant of the truncated Gaussian distribution $C_{R,k,\hat{\lambda}_i}^{\mathrm{trnc}}= \left\{\int_{\mathbb{R}^{k}}  \exp\left(-\bm{\beta}_{1:k}^{\top} \tilde{\Sigma}_{1:k}^{-1} \bm{\beta}_{1:k} \right)\mone\{ \|\bm{\beta}_{1:k}\|_{\tilde{\Sigma}_{1:k}} \le R \} \mathrm{d} \bm{\beta}_{1:k} \right\}^{-1}$, we continue \eqref{ineq:bound3_5} as 
\begin{align}
    &\Pi_{\mathfrak{b}} \left(\left\{\bm{\beta}_{1:k}: \|\bm{\beta}_{1:k}-\bm{\beta}_{1:k}^* \|_{\grave{\Sigma}_{1:k}}^2 \geq \tilde{\varepsilon}_n^2 \right\}\mid \mD_1, k\right)\nonumber 
    \\
    &=C_{R,k,\hat{\lambda}_i}^{\mathrm{trnc}}\int_{\{\bm{\beta}_{1:k}:\|\grave{\Sigma}_{1:k}^{1/2}(\bm{\beta}_{1:k}^* - \bm{\beta}_{1:k})\|_{2}^2 \geq \tilde{\varepsilon}_n^2\}} \exp\left(-\bm{\beta}_{1:k}^{\top} \tilde{\Sigma}_{1:k}^{-1} \bm{\beta}_{1:k} \right)\mone\{ \|\bm{\beta}_{1:k}\|_{\tilde{\Sigma}_{1:k}} \le R \} \mathrm{d} \bm{\beta}_{1:k} ,\nonumber\\
    &=1-C_{R,k,\hat{\lambda}_i}^{\mathrm{trnc}}\int_{\{\bm{\beta}_{1:k}:\|\grave{\Sigma}_{1:k}^{{1/2}}(\bm{\beta}_{1:k}^* - \bm{u }_{1:k})\|_{2}^2 < \tilde{\varepsilon}_n^2\}} \exp\left(-\bm{\beta}_{1:k}^{\top} \tilde{\Sigma}_{1:k}^{-1} \bm{\beta}_{1:k} \right) \mathrm{d}\bm{\beta}_{1:k} \nonumber\\
    & = 1- \underbrace{\frac{C_{R,k,\hat{\lambda}_i}^{\mathrm{trnc}}}{C_{k,\hat{\lambda}_i}^{\mathrm{norm}}}}_{=:E_A} \underbrace{C_{k,\hat{\lambda}_i}^{\mathrm{norm}}\int_{\{\bm{\beta}_{1:k}:\|\grave{\Sigma}_{1:k}^{{1/2}}(\bm{\beta}_{1:k}^* - \bm{\beta}_{1:k})\|_{2}^2 < \tilde{\varepsilon}_n^2\}} \exp\left(-\bm{\beta}_{1:k}^{\top} \tilde{\Sigma}_{1:k}^{-1} \bm{\beta}_{1:k} \right) \mathrm{d}\bm{\beta}_{1:k}}_{=:E_B}, \label{eq:trnc_}
\end{align}
where $C_{k,\hat{\lambda}_i}^{\mathrm{norm}} := \left\{\int_{\mathbb{R}^k} \exp(-\bm{h}^{\top} \tilde{\Sigma}_{1:k}^{-1} \bm{h} ) \mathrm{d}\bm{h}\right\}^{-1}$. The second equality holds with probability at least $1-\ell_n$. Indeed, for any $\bm{\beta}_{1:k}\in \mathbb{R}^k$ such that
$\sum_{i=1}^k \tilde{\lambda}_i(\beta_i - \beta_i^*)^2 < \tilde{\varepsilon}_n^2$, we have $
\sum_{i=1}^k \hat{\lambda}_i \beta_i^2 \le 2\sum_{i=1}^k \hat{\lambda}_i \beta_i^{*2} + 2\sum_{i=1}^k \hat{\lambda}_i (\beta_i-\beta_i^*)^2
\le 2\|\bm{\theta}^*\|_{\hat{\Sigma}}^2 + 2\sum_{i=1}^k \tilde{\lambda}_i (\beta_i-\beta_i^*)^2
\le 2\|\bm{\theta}^*\|_{\Sigma}^2 + 2\lambda_1\rho_n\|\bm{\theta}^*\|_2^2 +  2\sum_{i=1}^k \tilde{\lambda}_i (\beta_i-\beta_i^*)^2
< \frac{R^2}{2} + o(1) < R^2$ with probability at least $1-\ell_n$ from Lemma~\ref{lem:eigen-error}. 

Regarding $E_A$, we obtain
\begin{align}
    \left(\frac{C_{R,k,\hat{\lambda}_i}^{\mathrm{trnc}}}{C_{k,\hat{\lambda}_i}^{\mathrm{norm}}} \right)^{-1}
    & = C_{k,\hat{\lambda}_i}^{\mathrm{norm}} \int_{ \|\bm{h}\|_{\tilde{\Sigma}_{1:k}} \le R}  \exp\left(-\bm{h}^{\top} \tilde{\Sigma}_{1:k}^{-1} \bm{h} \right) \mathrm{d} \bm{h} \\
    & = C_{k,\hat{\lambda}_i}^{\mathrm{norm}} \frac{1}{2^{k/2}} \left|\tilde{\Sigma}_{1:k} \right|^{1/2} \int_{ \|\bm{h}\|_{\tilde{\Sigma}_{1:k}} \le R}  \exp\left(-\frac{\|\bm{h}\|_2^2}{2} \right) \mathrm{d} \bm{h} \\   
    & = \frac{1}{(2\pi)^{k/2}} \int_{ \|\bm{h}\|_{\tilde{\Sigma}_{1:k}} \le R}  \exp\left(-\frac{\|\bm{h}\|_2^2}{2} \right) \mathrm{d} \bm{h} \\
    & \le 1
    \label{ineq:bound3_7}
\end{align}
Similarly, about $E_B$, it holds that
\begin{align*}
    E_B & = C_{k,\hat{\lambda}_i}^{\mathrm{norm}}\int_{\{\bm{\beta}_{1:k}:\|\grave{\Sigma}_{1:k}^{{1/2}}(\bm{\beta}_{1:k} - \bm{\beta}_{1:k}^*)\|_{2} < \tilde{\varepsilon}_n\}} \exp\left(-\bm{\beta}_{1:k}^{\top} \tilde{\Sigma}_{1:k}^{-1} \bm{\beta}_{1:k} \right) \mathrm{d}\bm{\beta}_{1:k} \\
    & = C_{k,\hat{\lambda}_i}^{\mathrm{norm}} \int_{\{\bm{h}:\|\grave{\Sigma}_{1:k}^{{1/2}}\bm{h}\|_{2} < \tilde{\varepsilon}_n\}} \exp\left(- (\bm{h}+\bm{\beta}_{1:k}^*)^{\top} \tilde{\Sigma}_{1:k}^{-1} (\bm{h} +\bm{\beta}_{1:k}^*) \right) \mathrm{d}\bm{h} \\
    & \geq C_{k,\hat{\lambda}_i}^{\mathrm{norm}} \frac{1}{2^{k/2}} |\tilde{\Sigma}_{1:k} |^{1/2}  \exp\left(-2\bm{\beta}_{1:k}^{*\top} \tilde{\Sigma}_{1:k}^{-1} \bm{\beta}_{1:k}^* \right)
    \int_{\{\bm{h}:\|\grave{\Sigma}_{1:k}^{{1/2}} \tilde{\Sigma}_{1:k}^{1/2} \bm{h}\|_{2} < \sqrt{2} \tilde{\varepsilon}_n\}} \exp\left(-\frac{\|\bm{h}\|_2^2 }{2} \right) \mathrm{d}\bm{h} \\
    & = \frac{1}{(2\pi)^{k/2}}  \exp\left(-2\bm{\beta}_{1:k}^{*\top} \tilde{\Sigma}_{1:k}^{-1} \bm{\beta}_{1:k}^* \right)
    \int_{\{\bm{h}:\|\grave{\Sigma}_{1:k}^{{1/2}} \tilde{\Sigma}_{1:k}^{1/2} \bm{h}\|_{2} < \sqrt{2} \tilde{\varepsilon}_n\}} \exp\left(-\frac{\|\bm{h}\|_2^2}{2} \right) \mathrm{d}\bm{h}
\end{align*}
The above inequality follows from the relation: $(\bm{a}+\bm{b})^{\top} A (\bm{a}+\bm{b}) \leq 2\bm{a}^{\top} A \bm{a} + 2\bm{b}^{\top} A \bm{b}$ for any semi-positive definite matrix $A$.
Using 
 \begin{equation*}
     \left\{ \bm{h} : \left\|\tilde{\Sigma}_{1:k}^{1/2}\grave{\Sigma}_{1:k}^{{1/2}} \bm{h} \right\|_2 < \sqrt{2}\tilde{\varepsilon}_n \right\} \supset 
     \left\{ \bm{h} : \left\| \sqrt{\hat{\lambda}_1 \tilde{\lambda}_1} \bm{h} \right\|_2 < \sqrt{2} \tilde{\varepsilon}_n \right\},
 \end{equation*}
we continue the above as follows:
\begin{align}
    E_B & \geq 
    \frac{1}{(2\pi)^{k/2}}  \exp\left(-2\bm{\beta}_{1:k}^{*\top} \tilde{\Sigma}_{1:k}^{-1} \bm{\beta}_{1:k}^* \right)
    \int_{\{\bm{h}:\|\grave{\Sigma}_{1:k}^{{1/2}} \tilde{\Sigma}_{1:k}^{1/2} \bm{h}\|_{2} < \sqrt{2} \tilde{\varepsilon}_n\}} \exp\left(-\frac{\|\bm{h}\|_2^2}{2} \right) \mathrm{d}\bm{h} \\
    & \geq \frac{1}{(2\pi)^{k/2}} \exp\left(-2\bm{\beta}_{1:k}^{*\top} \tilde{\Sigma}_{1:k}^{-1} \bm{\beta}_{1:k}^* \right)
    \int_{\{\bm{h}:\| \sqrt{\hat{\lambda}_1 \tilde{\lambda}_1} \bm{h}\|_{2} < \sqrt{2} \tilde{\varepsilon}_n\}} \exp\left(-\frac{\|\bm{h}\|_2^2}{2} \right) \mathrm{d}\bm{h} \\
    & \geq \frac{1}{(2\pi)^{k/2}}  \exp\left(-2\bm{\beta}_{1:k}^{*\top} \tilde{\Sigma}_{1:k}^{-1} \bm{\beta}_{1:k}^* \right)  \exp\left(- \frac{\tilde{\varepsilon}_n^2}{\hat{\lambda}_1 \tilde{\lambda}_1}\right) \frac{\pi^{k/2}}{\Gamma(\frac{k}{2}+1)} \left(\frac{\sqrt{2}\tilde{\varepsilon}_n}{\sqrt{\hat{\lambda}_1 \tilde{\lambda}_1}}\right)^k \\
    & = \exp\left(-2\bm{\beta}_{1:k}^{*\top} \tilde{\Sigma}_{1:k}^{-1} \bm{\beta}_{1:k}^* \right)  \exp\left(- \frac{\tilde{\varepsilon}_n^2}{\hat{\lambda}_1 \tilde{\lambda}_1}\right) \frac{1}{\Gamma(\frac{k}{2}+1)} \left(\frac{\tilde{\varepsilon}_n}{\sqrt{\hat{\lambda}_1 \tilde{\lambda}_1}}\right)^k.
\end{align}
The last inequality follows the fact that the probability of a ball is bounded from below by the product of the minimum probability density within the ball and its volume. From Stirling's formula and $\sqrt{2\pi} \le k^{\frac12}$ for large $k$, $\log\Gamma(\frac{k}{2}+1)<  k\log k$. Also, $\hat{\lambda}_1 \le \sqrt{ \hat{\lambda}_1 \tilde{\lambda}_1} \le \hat{\lambda}_1 + {\lambda}_1$.
Hence, we have
\begin{align}
    E_B & \geq \exp\left\{ -2\bm{\beta}_{1:k}^{*\top} \tilde{\Sigma}_{1:k}^{-1} \bm{\beta}_{1:k}^* - \frac{\tilde{\varepsilon}_n^2}{\hat{\lambda}_1^2} - k\log k - k \log \left(\frac{\hat{\lambda}_1 + {\lambda}_1}{\tilde{\varepsilon}_n}\right) \right \}.
    \label{ineq:bound3_6}
\end{align}

Using \eqref{ineq:bound3_7} and \eqref{ineq:bound3_6}, we continue \eqref{eq:trnc_} as
\begin{align}
    &\Pi_{\mathfrak{b}} \left( \{\bm{\beta}_{1:k} : \|\bm{\beta}_{1:k} - \bm{\beta}_{1:k}^*\|_{\grave{\Sigma}_{1:k}}^2 \geq \tilde{\varepsilon}_n^2 \} \mid \mD_1, k\right)\\
    & \leq 1 - \exp\left\{ -2\bm{\beta}_{1:k}^{*\top} \tilde{\Sigma}_{1:k}^{-1} \bm{\beta}_{1:k}^* - \frac{\tilde{\varepsilon}_n^2}{\hat{\lambda}_1^2} - k\log k - k \log \left(\frac{\hat{\lambda}_1 + {\lambda}_1}{\tilde{\varepsilon}_n}\right) \right \} .
\end{align}
Then, we obtain 
\begin{align}
    \Pi_{\vartheta,\varsigma^2} \left(\mB_2\mid \mB_1, \mD_1\right) 
    \geq \sum_{k =L_{\kappa}}^{U_{\kappa}} 
    \pi_{\kappa}(k) \exp\left\{ -2\bm{\beta}_{1:k}^{*\top} \tilde{\Sigma}_{1:k}^{-1} \bm{\beta}_{1:k}^* - \frac{\tilde{\varepsilon}_n^2}{\hat{\lambda}_1^2} - k\log k - k \log \left(\frac{\hat{\lambda}_1 + {\lambda}_1}{\tilde{\varepsilon}_n}\right) \right \}.  \label{ineq:bound3_3}
\end{align}

\textbf{Step (iii): Combining the result}.
Finally, combining the results in \eqref{ineq:bound3_2} and \eqref{ineq:bound3_3} yields
\begin{align}
    &\Pi_{\vartheta,\varsigma^2}  \left(\left\{(\bm{\theta},\sigma^2): \KL(P^*, P_{\bm{\theta}, \sigma^2}) \leq {\varepsilon}_n^2, 
	\KV(P^*, P_{\bm{\theta}, \sigma^2}) \leq {\varepsilon}_n^2 \right\} \mid \mD_1 \right)\\
    &\gtrsim \varepsilon_n \sum_{k =L_{\kappa}}^{U_{\kappa}} 
    \pi_{\kappa}(k) \exp\left\{ -2\bm{\beta}_{1:k}^{*\top} \tilde{\Sigma}_{1:k}^{-1} \bm{\beta}_{1:k}^* - \frac{\tilde{\varepsilon}_n^2}{\hat{\lambda}_1^2} - k\log k - k \log \left(\frac{\hat{\lambda}_1 + {\lambda}_1}{\tilde{\varepsilon}_n}\right) \right \} \\
    & = \sum_{k =L_{\kappa}}^{U_{\kappa}} 
    \exp \left[ - \log \frac{1}{\varepsilon_n \pi_{\kappa}(k)} - 2\bm{\beta}_{1:k}^{*\top} \tilde{\Sigma}_{1:k}^{-1} \bm{\beta}_{1:k}^* - \frac{\tilde{\varepsilon}_n^2}{\hat{\lambda}_1^2} - k\log k - k \log \left(\frac{\hat{\lambda}_1 + {\lambda}_1}{\tilde{\varepsilon}_n}\right) \right]
\end{align}
Then, we have
\begin{align*}
    &\frac{-1}{n} \log \Pi_{\vartheta,\varsigma^2} \left(\left\{\bm{\theta},\sigma^2: \KL(P^*, P_{\bm{\theta}, \sigma^2}) \leq \varepsilon_n^2, 
	\KV(P^*,P_{\bm{\theta}, \sigma^2}) \leq \varepsilon_n^2 \right\}\mid\mD_1\right) \\ 
    & \lesssim \frac{-1}{n} \log \sum_{k =L_{\kappa}}^{U_{\kappa}} 
    \exp \left\{ - \log \frac{1}{\varepsilon_n \pi_{\kappa}(k)} - 2\bm{\beta}_{1:k}^{*\top} \tilde{\Sigma}_{1:k}^{-1} \bm{\beta}_{1:k}^* - \frac{\tilde{\varepsilon}_n^2}{\hat{\lambda}_1^2} - k\log k - k \log \left(\frac{\hat{\lambda}_1 + {\lambda}_1}{\tilde{\varepsilon}_n}\right) \right\}  \\
    & \leq \frac{1}{n} \left\{ \log \frac{1}{\varepsilon_n \pi_{\kappa}(L_{\kappa})} 
    + 2\bm{\beta}_{1:L_{\kappa}}^{*\top} \tilde{\Sigma}_{1:L_{\kappa}}^{-1} \bm{\beta}_{1:L_{\kappa}}^* 
    + \frac{\tilde{\varepsilon}_n^2}{\hat{\lambda}_1^2} 
    + L_{\kappa}\log L_{\kappa} + L_{\kappa} \log \left(\frac{\hat{\lambda}_1 + {\lambda}_1}{\tilde{\varepsilon}_n}\right) \right\} .
\end{align*}
The rate-determining terms are 
\begin{equation} \label{eq:rate-determine}
    \frac{1}{n}\log\frac{1}{\varepsilon_n\pi_{\kappa}(L_{\kappa})} + \frac1n\sum_{i=1}^{L_{\kappa}} \frac{\beta_i^{*2}}{\hat{\lambda}_i} + \frac{L_{\kappa}}{n} \log \left(L_{\kappa}\frac{{\hat{\lambda}_1 + {\lambda}_1}}{\tilde{\varepsilon}_n}\right) ,
\end{equation}
which aligns with the main statement from Lemma~\ref{lem:eigen-error}.

\end{proof}

\section{Proofs of Examples~\ref{example: polynomial_decay1} and \ref{example: polynomial_decay2}} 

We remark that $a_n \approx b_n$ means $a_n \gtrsim b_n$ and $a_n \lesssim b_n$.
\begin{proof}[of Example~\ref{example: polynomial_decay1}]
    Remark that $\sum_{i> L_{\kappa}} \lambda_i \approx \frac{L_{\kappa}^{1-\alpha}}{\alpha-1}$ and $\sum_{i> L_{\kappa}} \lambda_i^2 \approx \frac{L_{\kappa}^{1-2\alpha}}{2\alpha-1}$. $\bar{V}_n$ is $O(n^{-\frac{\alpha}{\alpha+1}})$ because $\frac{n \sum_{i>L_{\kappa}} \lambda_{i}^2 }{(\lambda + \sum_{i>L_{\kappa}} \lambda_{i})^2}\approx n L_{\kappa}^{1-2\alpha} \lambda^{-2} = n^{\frac{4(\alpha+1) + 2(1-2\alpha) -2(2\alpha+3) }{4(\alpha+1)}} = n^{\frac{-\alpha}{\alpha+1}} = \frac{U_{\kappa}}{n}$ and the first part is negligible. Regarding $\bar{B}_n$, $\frac1n\sum_{i=1}^{L_{\kappa}} \frac{\beta_i^{*2}}{\hat{\lambda}_i} \le \frac{L_{\kappa}^{\alpha+1}}{n} \approx n^{-1/2}$ and the other terms are evaluated as $O( n^{-\frac{2\alpha+1}{2(\alpha+1)}}\log n)$. If $\|\bm{\beta}_{1:L_{\kappa}}\|_2$ is $O(1)$, $\frac1n\sum_{i=1}^{L_{\kappa}} \frac{\beta_i^{*2}}{\hat{\lambda}_i} = \frac{\hat{\lambda}_{L_{\kappa}}^{-1}}
    {n} \|\bm{\beta}_{1:L_{\kappa}}^{*}\|_2^2 \approx n^{\frac{-\alpha-2}{2(\alpha+1)}}$.
\end{proof}
\begin{proof}[of Example~\ref{example: polynomial_decay2}]
We analyse the behaviour of the quantities $\bar{V}_n$ and $\bar{B}_n$ by partitioning the parameter $\alpha$ into distinct cases.
\begin{itemize}
    \item Case $\alpha=1$. $\sum_{i> L_{\kappa}} \lambda_i \approx \log\frac{p}{L_{\kappa}}$ and $\sum_{i> L_{\kappa}} \lambda_i^2 \approx n^{-1/4}$. $\bar{V}_n$ is $O(n^{-\frac{\alpha}{\alpha+1}})$, because $\frac{n \sum_{i>L_{\kappa}} \lambda_{i}^2 }{(\lambda + \sum_{i>L_{\kappa}} \lambda_{i})^2}\approx n n^{-1/4} n^{-5/4} = n^{-\frac{1}{2}} = \frac{U_{\kappa}}{n}$ and the first part is negligible. Regarding $\bar{B}_n$, $\frac1n\sum_{i=1}^{L_{\kappa}} \frac{\beta_i^{*2}}{\hat{\lambda}_i} \le \frac{L_{\kappa}^{\alpha+1}}{n} \approx n^{-1/2}$ and the other terms are evaluated as $O(n^{-3/4})$.
    \item Case $\frac12<\alpha<1$. $\sum_{i> L_{\kappa}} \lambda_i \approx p^{1-\alpha}$ and $\sum_{i> L_{\kappa}} \lambda_i^2 \approx L_{\kappa}^{1-2\alpha}$. $\bar{V}_n$ is $O(n^{-\frac{\alpha}{\alpha+1}})$, because $\frac{n \sum_{i>L_{\kappa}} \lambda_{i}^2 }{(\lambda + \sum_{i>L_{\kappa}} \lambda_{i})^2}\approx n L_{\kappa}^{1-2\alpha} p^{-2(1-\alpha)} = n^{ \frac{3}{2(\alpha+1)}} p^{-2(1-\alpha)}$, which is bounded by $n^{-\frac{\alpha}{\alpha+1}}$ if $p > n^{\frac{2\alpha+3}{4(1-\alpha^2)}}$.
    Regarding $\bar{B}_n$, $\frac1n\sum_{i=1}^{L_{\kappa}} \frac{\beta_i^{*2}}{\hat{\lambda}_i} \le \frac{L_{\kappa}^{\alpha+1}}{n} \approx n^{-1/2}$ and the other two terms are bounded from above by $n^{-\frac{2\alpha+1}{2(\alpha+1)}} \log n$.
    \item Case $\alpha=\frac12$. $\sum_{i> L_{\kappa}} \lambda_i \approx 2p^{\frac12}$ and $\sum_{i> L_{\kappa}} \lambda_i^2 \approx \log\frac{p}{L_{\kappa}}$. About $\bar{V}_n$, because $\frac{n \sum_{i>L_{\kappa}} \lambda_{i}^2 }{(\lambda + \sum_{i>L_{\kappa}} \lambda_{i})^2}\lesssim \frac{n \log p}{n^{\frac{4}{3}} + p}$, $\bar{V}_n$ is bounded by $n^{-\frac{1}{3}}$ if $p > n^{\frac{4}{3}}$. Inequality $\frac1n\sum_{i=1}^{L_{\kappa}} \frac{\beta_i^{*2}}{\hat{\lambda}_i} \le n^{-\frac{1}{2}} \le n^{-\frac{1}{3}}$ implies the bound for $\bar{B}_n$.
    \item Case $0<\alpha<\frac12$. $\sum_{i> L_{\kappa}} \lambda_i \approx \frac{p^{1-\alpha}}{1-\alpha}$ and $\sum_{i> L_{\kappa}} \lambda_i^2 \approx \frac{p^{1-2\alpha}}{1-2\alpha}$. Regarding $\bar{V}_n$, because $\frac{n \sum_{i>L_{\kappa}} \lambda_{i}^2 }{(\lambda + \sum_{i>L_{\kappa}} \lambda_{i})^2}\approx \frac{n}{p}$, which is bounded from above by $n^{-\frac{\alpha}{\alpha+1}}$ if $p > n^{\frac{2\alpha+1}{\alpha+1}}$. Concerning $\bar{B}_n$, $\frac1n\sum_{i=1}^{L_{\kappa}} \frac{\beta_i^{*2}}{\hat{\lambda}_i} \lesssim n^{-1/2} < n^{-\frac{\alpha}{\alpha+1}}$.
\end{itemize}
\end{proof}

\section{Proof of Theorem \ref{thm:disapp}}  \label{app:disapp}

\begin{proof}[of Theorem \ref{thm:disapp}]
We commence by considering the event specified in Assumption~\ref{ass:emp2}. Given the event, the subsequent results hold with probability at least $1-2\ell_n$ for every $n\in\N$, which implies the convergence in probability. Furthermore, if we assume $\sum_{n\ge 1}\ell_n <\infty$, the Borel--Cantelli lemma allows us to claim almost sure convergence.

Regarding the approximation distribution, the density function of $\bm{\beta}_{1:L_{\kappa}} = T_{\mD_1,L_{\kappa}}^{-1}(\bm{\theta})$ (for $\bm{\theta}\in S_{\mD_1,L_\kappa}$) is
\begin{align}
    &\pi^{\infty}_{\mathfrak{b}}(\bm{\beta}_{1:L_{\kappa}}\mid \mD,\sigma^2)\\
    &= \frac{\exp\big\{- (\bm{\beta}_{1:L_{\kappa}} - \bm{m}_{\sigma^2})^\top \Lambda_{\sigma^2} (\bm{\beta}_{1:L_{\kappa}} - \bm{m}_{\sigma^2}) \big\}\mone\left\{    \|\bm{\beta}_{1:L_{\kappa}}\|_{\tilde{\Sigma}_{1:L_{\kappa}}}\le R 
  \right\} }{ \int_{\mathbb{R}^{L_{\kappa}} } \exp\big\{- (\bm{\beta}_{1:L_{\kappa}}'- \bm{m}_{\sigma^2})^\top \Lambda_{\sigma^2} (\bm{\beta}_{1:L_{\kappa}}'- \bm{m}_{\sigma^2})\big\} \mone\left\{    \|\bm{\beta}_{1:L_{\kappa}}'\|_{\tilde{\Sigma}_{1:L_{\kappa}}}\le R 
  \right\}\mathrm{d}\bm{\beta}_{1:L_{\kappa}}'  } \\
  & =: \frac{g_n^\infty(\bm{\beta}_{1:L_{\kappa}},\sigma^2) }{\int_{\mathbb{R}^{L_{\kappa}} } g_n^\infty(\bm{\beta}_{1:L_{\kappa}}',\sigma^2)\mathrm{d} \bm{\beta}_{1:L_{\kappa}}' }, 
\end{align}
where $g_n^\infty(\bm{\beta}_{1:L_{\kappa}},\sigma^2) = \exp\big\{- (\bm{\beta}_{1:L_{\kappa}} - \bm{m}_{\sigma^2})^\top \Lambda_{\sigma^2} (\bm{\beta}_{1:L_{\kappa}} - \bm{m}_{\sigma^2}) \big\}\mone\{\|\bm{\beta}_{1:k}\|_{\tilde{\Sigma}_{1:k}}\le R\}.$
Furthermore, the posterior density of $\bm{\beta}_{1:k}$ given $\mD$, $\sigma^2$, and $k$ is
\begin{align}
    &\pi(\bm{\beta}_{1:k} \mid \mD,\sigma^2, k) \\
    & = \frac{\exp\left\{-\frac{1}{2\sigma^2}\|X_2\hat{V}_{1:k} \bm{\beta}_{1:k} -Y_2\|_2^2 - \bm{\beta}_{1:k}^\top \tilde{\Sigma}_{1:k}^{-1} \bm{\beta}_{1:k} \right\}\mone\{\|\bm{\beta}_{1:k}\|_{\tilde{\Sigma}_{1:k}}\le R\} }{\int_{\mathbb{R}^k} \exp\left\{-\frac{1}{2\sigma^2}\|X_2\hat{V}_{1:k} \bm{\beta}_{1:k} '-Y_2\|_2^2- \bm{\beta}_{1:k}'^\top \tilde{\Sigma}_{1:k}^{-1} \bm{\beta}_{1:k}'  \right\}\mone\{  \|\bm{\beta}_{1:k}'\|_{\tilde{\Sigma}_{1:k}}\le R\} \mathrm{d} \bm{\beta}_{1:k} '  },\\
    & = \frac{\exp\left\{- ( \bm{\beta}_{1:k} -\bm{\mu}_{\sigma^2,n,k})^\top Q_{\sigma^2,n,k} (\bm{\beta}_{1:k} -\bm{\mu}_{\sigma^2,n,k})  \right\} \mone\{  \|\bm{\beta}_{1:k}\|_{\tilde{\Sigma}_{1:k}}\le R\} }{ \int_{\mathbb{R}^k}\exp\left\{- (\bm{\beta}_{1:k} '-\bm{\mu}_{\sigma^2,n,k})^\top Q_{\sigma^2,n,k} (\bm{\beta}_{1:k} '-\bm{\mu}_{\sigma^2,n,k})\right\} \mone\{  \|\bm{\beta}_{1:k}'\|_{\tilde{\Sigma}_{1:k}}\le R\} \mathrm{d}\bm{\beta}_{1:k}'   }    \\
    &=  \frac{g_n(\bm{\beta}_{1:k},\sigma^2,k) }{ \int_{\mathbb{R}^k}g_n( \bm{\beta}_{1:k}',k,\sigma^2)\mathrm{d} \bm{\beta}_{1:k} '}
\end{align}
where we define $Q_{\sigma^2,n,k} := \tilde{\Sigma}_{1:k}^{-1} + \hat{V}_{1:k}^{\top}X_2^{\top}X_2\hat{V}_{1:k}/2\sigma^2$, $\bm{\mu}_{\sigma^2,n,k} := Q_{\sigma^2,n,k}^{-1}\hat{V}_{1:L_{\kappa}}^{\top}X_2^{\top}X_2\bar{\bm{\theta}}_2/2\sigma^2$
and
$g_n(\bm{\beta}_{1:k},\sigma^2,k) := \exp\left\{- ( \bm{\beta}_{1:k} -\bm{\mu}_{\sigma^2,n,k})^\top Q_{\sigma^2,n,k} (\bm{\beta}_{1:k} -\bm{\mu}_{\sigma^2,n,k})  \right\} \mone\{  \|\bm{\beta}_{1:k}\|_{\tilde{\Sigma}_{1:k}}\le R\}$. 

Next, we evaluate the discrepancy between $\Pi(\cdot\mid \mD,\sigma^2)$ and $\Pi^{\infty}(\cdot\mid \mD,\sigma^2)$ as follows:
    \begin{align}
        &\left\|\Pi(\cdot\mid \mD,\sigma^2) - \Pi^{\infty}(\cdot\mid \mD,\sigma^2)\right\|_{\mathrm{TV}} \\
        & = \sup_{B\in \mB(\R^p)} \left|\Pi(B\mid \mD,\sigma^2) - \Pi^{\infty}(B\mid \mD,\sigma^2)\right| \\
        & = \sup_{B\in \mB(\R^p)} \left| \sum_{k=L_{\kappa}}^{U_{\kappa}} \Pi(B,k \mid \mD,\sigma^2) - \Pi^{\infty}(B \mid \mD,\sigma^2)\right| \\
        & = \sup_{B\in \mB(\R^p)} \left| \sum_{k=L_{\kappa}}^{U_{\kappa}} \pi_{\kappa}(k)\Pi(B \mid \mD,\sigma^2, k) - \Pi^{\infty}(B \mid \mD,\sigma^2)\right| \\
        & = \sup_{B\in \mB(\R^p)} \left(\left| \pi_{\kappa}(L_{\kappa})\Pi(B \mid \mD,\sigma^2, L_{\kappa}) - \Pi^{\infty}(B \mid \mD,\sigma^2)\right| + \sum_{k=L_{\kappa}+1}^{U_{\kappa}} \pi_{\kappa}(k)\Pi(B \mid \mD,\sigma^2, k)  \right)\\
        & \le \sup_{B\in \mB(\R^p)} \left| \pi_{\kappa}(L_{\kappa})\Pi(B \mid \mD,\sigma^2, L_{\kappa}) - \Pi^{\infty}(B \mid \mD,\sigma^2)\right| + \sup_{B\in \mB(\R^p)} \sum_{k=L_{\kappa}+1}^{U_{\kappa}} \pi_{\kappa}(k)\Pi(B \mid \mD,\sigma^2, k) \\
        & = \left\| \pi_{\kappa}(L_{\kappa}) \Pi(\cdot\mid \mD,\sigma^2, L_{\kappa}) - \Pi^{\infty}(\cdot\mid \mD,\sigma^2)\right\|_{\mathrm{TV}} 
        + \left\| \sum_{k=L_{\kappa}+1}^{U_{\kappa}} \pi_{\kappa}(k)\Pi(\cdot \mid \mD,\sigma^2, k)\right\|_{\mathrm{TV}} \label{eq:TVs}
    \end{align}
    
    Regarding the first term of \eqref{eq:TVs}, note that the distributions $\Pi(\cdot\mid \mD,\sigma^2, L_{\kappa})$ and $\Pi^{\infty}(\cdot\mid \mD,\sigma^2)$ share the same support; that is, they are mutually absolutely continuous. Then, with set $R_{\mD_1,L_{\kappa}}=\{\bm{\beta}_{1:L_{\kappa}}\in \R^{L_{\kappa}}:\|\bm{\beta}_{1:L_{\kappa}}\|_{\tilde{\Sigma}_{1:L_{\kappa}}} \le R \}$, we obtain 
    \begin{align}
        &\left\| \pi_{\kappa}(L_{\kappa}) \Pi(\cdot\mid \mD,\sigma^2, L_{\kappa}) - \Pi^{\infty}(\cdot\mid \mD,\sigma^2)\right\|_{\mathrm{TV}} \\
        & = \frac12  \int_{\R^{L_{\kappa}}} |\pi_{\kappa}(L_{\kappa}) \pi(\bm{\beta}_{1:L_{\kappa}}\mid \mD,\sigma^2, L_{\kappa}) - \pi^{\infty}(\bm{\beta}_{1:L_{\kappa}}\mid \mD,\sigma^2)| \mathrm{d}\bm{\beta}_{1:L_{\kappa}} \\
        & = \frac12 \int_{\R^{L_{\kappa}}} \left| \frac{\pi_{\kappa}(L_{\kappa})g_n(\bm{\beta}_{1:k},\sigma^2,L_\kappa)}{\int_{\R^{L_{\kappa}}} g_n(\bm{\beta}_{1:L_{\kappa}}',L_{\kappa},\sigma^2)\mathrm{d} \bm{\beta}_{1:L_{\kappa}}' } -\frac{g_n^\infty(\bm{\beta}_{1:L_{\kappa}},\sigma^2) }{\int_{\R^{L_{\kappa}}} g_n^\infty(\bm{\beta}_{1:L_{\kappa}}',\sigma^2)\mathrm{d} \bm{\beta}_{1:L_{\kappa}}' } \right| \mathrm{d}\bm{\beta}_{1:L_{\kappa}} \\
        & \le \frac{\int_{\R^{L_{\kappa}}}\left| \pi_{\kappa}(L_{\kappa})g_n(\bm{\beta}_{1:k},\sigma^2,L_\kappa)-g_n^\infty(\bm{\beta}_{1:L_{\kappa}},\sigma^2)\right|\mathrm{d}\bm{\beta}_{1:L_{\kappa}} }{\int_{\R^{L_{\kappa}}} g_n(\bm{\beta}_{1:L_{\kappa}}',L_{\kappa},\sigma^2)\mathrm{d} \bm{\beta}_{1:L_{\kappa}}' }  + \left(1-\pi_{\kappa}(L_{\kappa})\right) \\
        & \le  \sup_{\bm{\beta}_{1:L_{\kappa}} \in R_{\mD_1,L_{\kappa}} } \left| \pi_{\kappa}(L_{\kappa})  - \frac{g_n^\infty(\bm{\beta}_{1:L_{\kappa}},\sigma^2) }{g_n(\bm{\beta}_{1:k},\sigma^2,L_\kappa)}\right| + \left(1-\pi_{\kappa}(L_{\kappa})\right) ,
    \end{align}
    where the first inequality follows from Lemma~\ref{lem:integral0} and the second from Hölder’s inequality. To compare the functions $g_n^\infty$ and $g_n$ over $R_{\mD_1,L_{\kappa}}$, observe that
    \begin{align*}
        &-\log\left\{\frac{g_n^\infty(\bm{\beta}_{1:L_{\kappa}},\sigma^2)}{g_n(\bm{\beta}_{1:k},\sigma^2,L_\kappa)} \right\}\\
        &= (\bm{\beta}_{1:L_{\kappa}} - \bm{m}_{\sigma^2})^\top \Lambda_{\sigma^2} (\bm{\beta}_{1:L_{\kappa}} - \bm{m}_{\sigma^2})  - 
        ( \bm{\beta}_{1:L_{\kappa}} -\bm{\mu}_{\sigma^2,n,L_{\kappa}})^\top Q_{\sigma^2,n,L_{\kappa}} (\bm{\beta}_{1:L_{\kappa}} -\bm{\mu}_{\sigma^2,n,L_{\kappa}})\\
        &= \bm{m}_{\sigma^2}^\top \Lambda_{\sigma^2} \bm{m}_{\sigma^2}
        - 2\bm{m}_{\sigma^2}^\top \Lambda_{\sigma^2}\bm{\beta}_{1:L_{\kappa}} 
        -\bm{\mu}_{\sigma^2,n,L_{\kappa}}^\top Q_{\sigma^2,n,L_{\kappa}} \bm{\mu}_{\sigma^2,n,L_{\kappa}} 
        + 2 \bm{\mu}_{\sigma^2,n,L_{\kappa}}^\top Q_{\sigma^2,n,L_{\kappa}}\bm{\beta}_{1:L_{\kappa}}.
    \end{align*}
    Employing an argument based on the Neumann series and Lemma~\ref{lem:eigen-error}, one can show that these terms are of order $o(1)$ under Assumption~\ref{ass:emp2}. More precisely, because we obtain
    \begin{align*}
        &\left(\tilde{\Sigma}_{1:L_{\kappa}}^{-1} +  \frac{\hat{V}_{1:L_{\kappa}}^{\top}X_2^{\top}X_2\hat{V}_{1:L_{\kappa}}}{2\sigma^2}\right)^{-1}\frac{\hat{V}_{1:L_{\kappa}}^{\top}X_2^{\top}X_2\bar{\bm{\theta}}_2}{2\sigma^2} \\
        &= \left(\tilde{\Sigma}_{1:L_{\kappa}}^{-1} +  \frac{\hat{V}_{1:L_{\kappa}}^{\top}X_2^{\top}X_2\hat{V}_{1:L_{\kappa}}}{2\sigma^2}\right)^{-1}  \left(\frac{\hat{V}_{1:L_{\kappa}}^{\top}X_2^{\top}X_2\hat{V}_{1:L_{\kappa}}}{2\sigma^2}\right) 
        \left(\hat{V}_{1:L_{\kappa}}^{\top}X_2^{\top}X_2\hat{V}_{1:L_{\kappa}}\right)^{-1}
        \hat{V}_{1:L_{\kappa}}^{\top}X_2^{\top}X_2\bar{\bm{\theta}}_2\\
        &= \left\{ \left(\frac{\hat{V}_{1:L_{\kappa}}^{\top}X_2^{\top}X_2\hat{V}_{1:L_{\kappa}}}{2\sigma^2}\right)^{-1}  \tilde{\Sigma}_{1:L_{\kappa}}^{-1} + I \right\}^{-1} 
        \left(\hat{V}_{1:L_{\kappa}}^{\top}X_2^{\top}X_2\hat{V}_{1:L_{\kappa}}\right)^{-1}
        \hat{V}_{1:L_{\kappa}}^{\top}X_2^{\top}X_2\bar{\bm{\theta}}_2\\
        &= \left\{I - \left(\frac{\hat{V}_{1:L_{\kappa}}^{\top}X_2^{\top}X_2\hat{V}_{1:L_{\kappa}}}{2\sigma^2}\right)^{-1}  \tilde{\Sigma}_{1:L_{\kappa}}^{-1} + A_n \right\}
        \left(\hat{V}_{1:L_{\kappa}}^{\top}X_2^{\top}X_2\hat{V}_{1:L_{\kappa}}\right)^{-1}
        \hat{V}_{1:L_{\kappa}}^{\top}X_2^{\top}X_2\bar{\bm{\theta}}_2,
    \end{align*}
    with $\|A_n\|_{\mathrm{op}}=O\left(n^{-2} \lambda_{L_{\kappa}}^{-4} \right)$ as derived from the Neumann series, Assumption~\ref{ass:emp2}, Lemmas~\ref{lem:eigen-error}~and~\ref{lem:lambdaL}, it follows that
    \begin{align*}
        &\bm{m}_{\sigma^2}^\top \Lambda_{\sigma^2}\bm{\beta}_{1:L_{\kappa}} -  \bm{\mu}_{\sigma^2,n,L_{\kappa}}^\top Q_{\sigma^2,n,L_{\kappa}}\bm{\beta}_{1:L_{\kappa}}  \\
        &= \left[\left\{ I - \left(\frac{ \hat{V}_{1:L_{\kappa}}^{\top}X_2^{\top}X_2\hat{V}_{1:L_{\kappa}}}{2\sigma^2 }\right)^{-1} \tilde{\Sigma}_{1:L_{\kappa}}^{-1} 
        \right\} \left( \hat{V}_{1:L_{\kappa}}^{\top}X_2^{\top}X_2\hat{V}_{1:L_{\kappa}}\right)^{-1}\hat{V}_{1:L_{\kappa}}^{\top}X_2^{\top}X_2\bar{\bm{\theta}}_2\right]^{\top}\left(\tilde{\Sigma}_{1:L_{\kappa}}^{-1} + \frac{ \hat{V}_{1:L_{\kappa}}^{\top}X_2^{\top}X_2\hat{V}_{1:L_{\kappa}}}{2\sigma^2 }\right)\bm{\beta}_{1:L_{\kappa}} \\
        &\quad - \left\{\left(\tilde{\Sigma}_{1:L_{\kappa}}^{-1} +  \frac{\hat{V}_{1:L_{\kappa}}^{\top}X_2^{\top}X_2\hat{V}_{1:L_{\kappa}}}{2\sigma^2}\right)^{-1}\frac{\hat{V}_{1:L_{\kappa}}^{\top}X_2^{\top}X_2\bar{\bm{\theta}}_2}{2\sigma^2}\right\}^{\top} \left(\tilde{\Sigma}_{1:L_{\kappa}}^{-1} +  \frac{\hat{V}_{1:L_{\kappa}}^{\top}X_2^{\top}X_2\hat{V}_{1:L_{\kappa}}}{2\sigma^2}\right)
        \bm{\beta}_{1:L_{\kappa}} \\
        & = \bar{\bm{\theta}}_2^{\top}X_2^{\top}X_2\hat{V}_{1:L_{\kappa}} 
        \left( \hat{V}_{1:L_{\kappa}}^{\top}X_2^{\top}X_2\hat{V}_{1:L_{\kappa}}\right)^{-1} A_n
        \left(\tilde{\Sigma}_{1:L_{\kappa}}^{-1} +  \frac{\hat{V}_{1:L_{\kappa}}^{\top}X_2^{\top}X_2\hat{V}_{1:L_{\kappa}}}{2\sigma^2}\right)
        \bm{\beta}_{1:L_{\kappa}}
    \end{align*}
    and 
    \begin{align*}
        &\bm{m}_{\sigma^2}^\top \Lambda_{\sigma^2} \bm{m}_{\sigma^2}
        -\bm{\mu}_{\sigma^2,n,L_{\kappa}}^\top Q_{\sigma^2,n,L_{\kappa}} \bm{\mu}_{\sigma^2,n,L_{\kappa}} \\
        &= \bar{\bm{\theta}}_2^{\top}X_2^{\top}X_2\hat{V}_{1:L_{\kappa}} 
        \left( \hat{V}_{1:L_{\kappa}}^{\top}X_2^{\top}X_2\hat{V}_{1:L_{\kappa}}\right)^{-1} A_n 
        \left( \hat{V}_{1:L_{\kappa}}^{\top}X_2^{\top}X_2\hat{V}_{1:L_{\kappa}}\right)^{-1} \hat{V}_{1:L_{\kappa}}^{\top}  X_2^{\top}X_2 \bar{\bm{\theta}}_2,
    \end{align*}
    both of which are of order $o(1)$ from Assumption~\ref{ass:emp2} and Lemma~\ref{lem:lambdaL}. Since $ R_{\mD_1,L_{\kappa}}$ is compact, we obtain as $n\to \infty$,
    \begin{align*}
        \sup_{\bm{\beta}_{1:L_{\kappa}} \in R_{\mD_1,L_{\kappa}}} \left| \pi_{\kappa}(L_{\kappa})  - \frac{g_n^\infty(\bm{\beta}_{1:L_{\kappa}}) }{g_n(\bm{\beta}_{1:L_{\kappa}})}\right| \le \left| \pi_{\kappa}(L_{\kappa}) - 1\right| + \sup_{\bm{\beta}_{1:L_{\kappa}} \in R_{\mD_1,L_{\kappa}}}\left| \frac{g_n^\infty(\bm{\beta}_{1:L_{\kappa}}) }{g_n(\bm{\beta}_{1:L_{\kappa}})} - 1 \right|
        \to 0.
    \end{align*}
    
    Finally, since $f(L_{\kappa})-f(k)= \omega(\log U_{\kappa})$ for $k=L_{\kappa}+1,\ldots,U_{\kappa}$, we have
    \begin{align*}
        \left\| \sum_{k=L_{\kappa}+1}^{U_{\kappa}} \pi_{\kappa}(k)\Pi(\cdot \mid \mD,\sigma^2, k)\right\|_{\mathrm{TV}} \le \frac{\sum_{k=L_{\kappa}+1}^{U_{\kappa}}\exp\{f(k) - f(L_{\kappa})\} }{\sum_{k=L_{\kappa}}^{U_{\kappa}} \exp\{f(k) - f(L_{\kappa})\}},
    \end{align*}
    which converges to zero. This implies that \eqref{eq:TVs} converges to zero as $n\to \infty$ and completes the proof.
\end{proof}

\section{Supportive Result}

\begin{lemma} \label{lem:l2ball}
    Let $B_2^k$ denote the unit ball in $\mathbb{R}^k$ with respect to the Euclidean norm $\|\cdot\|_2$, and let $N_{1/2} \subset B_2^k$ be a minimal $1/2$-covering of $B_2^k$. Then, for any vector $\bm{\beta} \in \mathbb{R}^k$, it holds that
    \begin{equation}
        \|\bm{\beta}\|_2 \leq 2 \max_{\bm{z} \in N_{1/2}} \bm{z}^\top \bm{\beta}.
    \end{equation}
\end{lemma}

\begin{proof}[of Lemma \ref{lem:l2ball}]
    Consider the dual representation of the Euclidean norm, $\|\bm{\beta}\|_2 = \max_{\bm{h} \in B_2^k} \bm{h}^\top \bm{\beta}$.
    Since $N_{1/2}$ is a $1/2$-covering of $B_2^k$, for any $\bm{h} \in B_2^k$, there exists a vector $\bm{z} \in N_{1/2}$ such that $\|\bm{h} - \bm{z}\|_2 \leq \frac{1}{2}$.
    Hence, for any $\bm{h} \in B_2^k$ and the corresponding $\bm{z}$, we have
    \[
        \bm{h}^\top \bm{\beta} = \bm{z}^\top \bm{\beta} + (\bm{h} - \bm{z})^\top \bm{\beta} \leq \bm{z}^\top \bm{\beta} + \|\bm{h} - \bm{z}\|_2 \|\bm{\beta}\|_2 \leq \bm{z}^\top \bm{\beta} + \frac{1}{2} \|\bm{\beta}\|_2.
    \]
    Taking the maximum over all $\bm{h} \in B_2^k$, we obtain
    \[
        \|\bm{\beta}\|_2 = \max_{\bm{h} \in B_2^k} \bm{h}^\top \bm{\beta} \leq \max_{\bm{z} \in N_{1/2}} \bm{z}^\top \bm{\beta} + \frac{1}{2} \|\bm{\beta}\|_2.
    \]
    Rearranging terms yields $\|\bm{\beta}\|_2 \leq 2 \max_{\bm{z} \in N_{1/2}} \bm{z}^\top \bm{\beta}$.
\end{proof}

\begin{lemma}\label{lem:cov}
Let $S$ be a compact subset of $\mathbb{R}^p$, $A$ be a $p \times p$ positive semi-definite matrix, and $\delta \in \mathbb{R}_+$. The covering number of $S$ with respect to the pseudo norm $\|\cdot\|_{A}$ satisfies
\begin{equation} 
    \mathcal{N}(\delta, S, \|\cdot\|_{A})
    =  \mathcal{N}(\delta, A^{1/2}S, \|\cdot\|_{2}),
\end{equation}
where $A^{1/2}$ denotes the symmetric square root of $A$, and $\|\cdot\|_{2}$ is the standard Euclidean norm on $\mathbb{R}^p$.
\end{lemma}

\begin{proof}[of Lemma~\ref{lem:cov}]
First, recall that the pseudo norm $\|\cdot\|_{A}$ is defined by $\|x\|_{A} = \|A^{1/2}x\|_{2}$ for all $x \in \mathbb{R}^p$.  Let $\mathcal{N} = \mathcal{N}(\delta, S, \|\cdot\|_{A})$ and $\mathcal{N}' = \mathcal{N}(\delta, A^{1/2}S, \|\cdot\|_{2})$. 

Suppose that $\{\phi_1, \ldots, \phi_{\mathcal{N}'}\}$ is a $\delta$-covering of $A^{1/2}S$ with respect to $\|\cdot\|_{2}$. For each $i$, since $\phi_i \in A^{1/2}S$, there exists $\psi_i \in S$ such that $\phi_i = A^{1/2}\psi_i$. Consider any $\xi \in S$. Then, $A^{1/2}\xi \in A^{1/2}S$, and there exists $i \in \{1, \ldots, \mathcal{N}'\}$ such that 
\[
\|A^{1/2}\xi - \phi_i\|_{2} \leq \delta.
\]
Substituting $\phi_i = A^{1/2}\psi_i$, we have
\[
\|A^{1/2}(\xi - \psi_i)\|_{2} = \|\xi - \psi_i\|_{A} \leq \delta.
\]
Thus, $\{\psi_1, \ldots, \psi_{\mathcal{N}'}\}$ forms a $\delta$-covering of $S$ with respect to $\|\cdot\|_{A}$, implying $\mathcal{N} \leq \mathcal{N}'.$

Conversely, suppose that $\{\psi_1, \ldots, \psi_{\mathcal{N}}\}$ is a $\delta$-covering of $S$ with respect to $\|\cdot\|_{A}$. For each $i$, let $\phi_i = A^{1/2}\psi_i$. Consider any $\xi \in A^{1/2}S$. Then, there exists $s \in S$ such that $\xi = A^{1/2}s$. Since $\{\psi_1, \ldots, \psi_{\mathcal{N}}\}$ is a $\delta$-covering of $S$, there exists $i \in \{1, \ldots, \mathcal{N}\}$ such that 
\[
\|s - \psi_i\|_{A} \leq \delta.
\]
Applying $A^{1/2}$, we obtain
\[
\|A^{1/2}s - A^{1/2}\psi_i\|_{2} = \| \xi - \phi_i \|_{2} \leq \delta.
\]
Thus, $\{\phi_1, \ldots, \phi_{\mathcal{N}}\}$ forms a $\delta$-covering of $A^{1/2}S$ with respect to $\|\cdot\|_{2}$, suggesting $\mathcal{N}' \leq \mathcal{N}$.
\end{proof}

\begin{lemma}\label{lem:integral0}
Let $f_1$ and $f_2$ be any positive integrable functions on a measure space $(\mS, \mB, \mu)$, and let $B\in\mB$ be a set with $\int_B f_1 \mathrm{d}\mu>0$ and $\int_B f_2\mathrm{d}\mu>0$. Then for any constant $0 < c \le 1$, we have
\begin{align*}
    \int_B \left|\frac{cf_1 }{\int_{B} f_1 \mathrm{d}\mu } - \frac{f_2 }{\int_{B} f_2 \mathrm{d}\mu } \right| \mathrm{d}\mu
    \le 2\frac{\int_B \left|cf_1 - f_2 \right| \mathrm{d}\mu }{\int_{B} f_1 \mathrm{d}\mu} 
    + 1-c.
\end{align*}
\end{lemma}
\begin{proof}[of Lemma \ref{lem:integral0}]
For any $s\in \mS$, we begin by expressing the integrand in the left-hand side
\begin{align}
    \frac{cf_1(s) }{\int_{B} f_1 \mathrm{d}\mu } - \frac{f_2(s) }{\int_{B} f_2 \mathrm{d}\mu }  
    &= \frac{c(\int_{B} f_2 \mathrm{d}\mu) f_1(s) - (\int_{B} f_1 \mathrm{d}\mu) f_2(s)}{(\int_{B} f_1 \mathrm{d}\mu)(\int_{B} f_2 \mathrm{d}\mu)} \\
    &= \frac{c(\int_{B} f_2 \mathrm{d}\mu) f_1(s) - (\int_{B} f_2 \mathrm{d}\mu) f_2(s) + (\int_{B} f_2 \mathrm{d}\mu) f_2(s) - (\int_{B} f_1 \mathrm{d}\mu) f_2(s)}{(\int_{B} f_1 \mathrm{d}\mu)(\int_{B} f_2 \mathrm{d}\mu)} \\
    &= \frac{ (\int_{B} f_2 \mathrm{d}\mu)(cf_1(s) - f_2(s))}{(\int_{B} f_1 \mathrm{d}\mu)(\int_{B} f_2 \mathrm{d}\mu)} + \frac{ (\int_{B} f_2 \mathrm{d}\mu) f_2(s) - (\int_{B} f_1 \mathrm{d}\mu) f_2(s)}{(\int_{B} f_1 \mathrm{d}\mu)(\int_{B} f_2 \mathrm{d}\mu)}.
\end{align}
Taking absolute values and integrating over $B$, the triangle inequality yields
\begin{align*}
    \int_B \left|\frac{cf_1 }{\int_{B} f_1 \mathrm{d}\mu } - \frac{f_2 }{\int_{B} f_2 \mathrm{d}\mu } \right| \mathrm{d}\mu
    &\le \frac{\int_B \left|cf_1 - f_2 \right| \mathrm{d}\mu }{\int_{B} f_1 \mathrm{d}\mu} 
    + \frac{\int_{B} |f_1 - f_2| \mathrm{d}\mu}{\int_{B} f_1 \mathrm{d}\mu} \\
    &\le 2\frac{\int_B \left|cf_1 - f_2 \right| \mathrm{d}\mu }{\int_{B} f_1 \mathrm{d}\mu} 
    + 1-c,
\end{align*}
which is the desired inequality.
\end{proof}

\begin{lemma}  \label{lem:eigen-error}
    Let $\Sigma$ be a positive definite matrix with eigenvalues $\lambda_1 \geq \lambda_2 \geq \dots \geq \lambda_p > 0$, and let $\hat{\Sigma}$ be an estimator of $\Sigma$ such that $\|\hat{\Sigma} - \Sigma\|_{\text{op}} \leq \lambda_1 \rho_n$ with probability at least $1 - \ell_n$, where $\rho_n = o(\lambda_{L+1})$. Denote by $\hat{\lambda}_i$ the $i$-th eigenvalue of $\hat{\Sigma}$ for $i = 1, \dots, p$. Then, with probability at least $1 - \ell_n$, the following inequalities hold for $i = 1, \dots, p$:
    \begin{equation}
        \lambda_i - \lambda_1 \rho_n < \hat{\lambda}_i < \lambda_i + \lambda_1 \rho_n.
    \end{equation}
    Moreover, for $i = 1, \dots, L+1$, the estimated eigenvalues satisfy
    \begin{equation}
        \lambda_i \lesssim \hat{\lambda}_i \lesssim \lambda_i.
    \end{equation}
\end{lemma}

\begin{proof}[of Lemma~\ref{lem:eigen-error}]
    By Weyl's inequality, for each $i = 1, \dots, p$, the difference between the $i$-th eigenvalues of $\hat{\Sigma}$ and $\Sigma$ is bounded as
    \begin{equation}
        |\hat{\lambda}_i - \lambda_i| \leq \|\hat{\Sigma} - \Sigma\|_{\text{op}}.
    \end{equation}
    According to the given assumption, with probability at least $1 - \ell_n$, we have
    \begin{equation}
        \|\hat{\Sigma} - \Sigma\|_{\text{op}} \leq \lambda_1 \rho_n, 
    \end{equation}
    implying, with probability at least $1 - \ell_n$, we have $\lambda_i - \lambda_1 \rho_n \leq \hat{\lambda}_i \leq \lambda_i + \lambda_1 \rho_n$.
    Since $\rho_n = o(\lambda_{L})$, it follows that for each $i = 1, \dots, L+1$, the perturbation $\lambda_1 \rho_n$ is asymptotically negligible relative to $\lambda_i$. Consequently, we have $\lambda_i \lesssim \hat{\lambda}_i \lesssim \lambda_i$, for each $i = 1, \dots, L+1$, completing the proof.
\end{proof}

\begin{lemma}\label{lem:lambdaL}
    Let the data be split into $\mathcal{D}_1 = \{(\bm{x}_i, y_i)\}_{i=1}^{n/2}$ and $\mathcal{D}_2 = \{(\bm{x}_i, y_i)\}_{i=n/2+1}^{n}$, assuming $n$ is even.
    Suppose that the true covariance matrix $\Sigma\in\mathbb{R}^{p\times p}$ is positive definite with eigenvalues $\lambda_1\ge \lambda_2\ge \cdots\ge \lambda_p>0$ and that empirical covariance matrices $\hat{\Sigma} = \frac{2}{n}\sum_{i=1}^{n/2} \bm{x}_i\bm{x}_i^\top$ and $\hat{\Sigma}_2=\frac{2}{n}\sum_{i=n/2+1}^{n} \bm{x}_i\bm{x}_i^\top$ satisfy Assumption~\ref{ass:emp}, and let the $L$-th eigen-gap $\Delta_L=\lambda_{L}-\lambda_{L+1}$ be positive. If \(\rho_n=o\bigl(\lambda_{L}\,\Delta_L\bigr)\), and if we denote by \(\hat{V}_{1:L}\in\mathbb{R}^{p\times L}\) the matrix whose columns are the top \(L\) eigenvectors of \(\hat{\Sigma}\) (so that \(\hat{V}_{1:L}^\top \hat{V}_{1:L}=I_{L}\)), then with probability at least \(1-2\ell_n\) we have 
    \begin{equation}
    \lambda_{\min}\left(\hat{V}_{1:L}^{\top}\hat{\Sigma}_2\hat{V}_{1:L}\right) \ge c\lambda_{L},
    \end{equation}
    for some constant $c>0$. Consequently, with probability at least \(1-2\ell_n\),
    \begin{equation}
    \left\|\left(\hat{V}_{1:L}^{\top}\hat{\Sigma}_2\hat{V}_{1:L}\right)^{-1}\right\|_{\mathrm{op}} = O\left(\lambda_{L}^{-1}\right).
    \end{equation}
\end{lemma}

Before beginning the proof, we introduce some notions.
Let \(U\) and \(V\) be matrices with orthonormal columns that span two subspaces of the same dimension \(k\). The principal angles \(\theta_1, \dots, \theta_k\) between these subspaces are defined so that the singular values of \(U^\top V\) are $\cos\theta_1, \dots, \cos\theta_k$. Then, we define $\sin\Theta$ as the diagonal matrix with entries $\sin\theta_1, \dots, \sin\theta_k$; in particular, its operator norm is $\max_{1\le i\le k}\sin\theta_i$.

\begin{proof}[of Lemma~\ref{lem:lambdaL}]
Write
\begin{equation}
    \hat{V}_{1:L}^\top\hat{\Sigma}_2\,\hat{V}_{1:L} = \hat{V}_{1:L}^\top\Sigma\,\hat{V}_{1:L} + \hat{V}_{1:L}^\top(\hat{\Sigma}_2-\Sigma)\hat{V}_{1:L}.
\end{equation}
Applying Weyl's inequality yields
\begin{equation}
    \lambda_{\min}\left(\hat{V}_{1:L}^\top\hat{\Sigma}_2\hat{V}_{1:L}\right) \ge \lambda_{\min}\left(\hat{V}_{1:L}^\top\Sigma\hat{V}_{1:L}\right) - \left\|\hat{V}_{1:L}^\top(\hat{\Sigma}_2-\Sigma)\hat{V}_{1:L}\right\|_{\mathrm{op}}.
\end{equation}

First, we evaluate $\lambda_{\min}\left(\hat{V}_{1:L}^\top\Sigma\hat{V}_{1:L}\right)$. 
Let $V_{1:L}\in\mathbb{R}^{p\times L}$ be the matrix whose columns are the eigenvectors corresponding to the top $L$ eigenvalues of $\Sigma$, with $V_{1:L}^\top V_{1:L}=I_{L}$. By the Davis--Kahan $\sin\Theta$ theorem~\citep[e.g.][]{pensky2024davis} and Assumption~\ref{ass:emp}, there exists an orthogonal matrix $O\in\mathbb{R}^{L\times L}$ such that
\begin{equation}
    \|\hat{V}_{1:L} - V_{1:L}O\|_{\mathrm{op}} 
    \le \sqrt{2} \| \sin\Theta(\hat{V}_{1:L}, V_{1:L})  \|_{\mathrm{op}}
    \le \epsilon_n,\quad \mathrm{with }\ \epsilon_n = O\left(\frac{\rho_n\lambda_1}{\Delta_L}\right).
\end{equation}
We express $\hat{V}_{1:L}$ as
\begin{equation}
    \hat{V}_{1:L} = V_{1:L}O + E, \quad \mathrm{with }\ \|E\|_{\mathrm{op}} \le \epsilon_n.
\end{equation}
Then, we have
\begin{equation}
    \hat{V}_{1:L}^\top\Sigma\hat{V}_{1:L} = O^\top{V}_{1:L}^\top\Sigma V_{1:L}O + E', \quad \mathrm{with }\ \|E'\|_{\mathrm{op}} \le 2\lambda_1\epsilon_n.
\end{equation}
and from Weyl's inequality,
\begin{equation}
\lambda_{\min}\left(\hat{V}_{1:L}^\top\Sigma\hat{V}_{1:L}\right) \ge \lambda_{L} - 2\lambda_1\epsilon_n .
\end{equation}

Next, we bound $\|\hat{V}_{1:L}^\top(\hat{\Sigma}_2-\Sigma)\hat{V}_{1:L}\|_{\mathrm{op}}$.
Since $\hat{V}_{1:L}$ is orthonormal, from Assumption~\ref{ass:emp}, with probability at least $1-\ell_n$,
\begin{equation}
    \left\|\hat{V}_{1:L}^\top(\hat{\Sigma}_2-\Sigma)\hat{V}_{1:L}\right\|_{\mathrm{op}} \le \left\|\hat{\Sigma}_2-\Sigma\right\|_{\mathrm{op}} \le \lambda_1\rho_n.
\end{equation}

Thus, we obtain the union bound, with probability at least $1-2\ell_n$,  
\begin{equation}
    \lambda_{\min}\left(\hat{V}_{1:L}^\top\hat{\Sigma}_2\hat{V}_{1:L}\right) \ge \lambda_{L} - \lambda_1(2\epsilon_n +\rho_n).
\end{equation}
Since both $\epsilon_n$ and $\rho_n$ are $o(\lambda_{L})$, we deduce that
\begin{equation}
    \lambda_{\min}\left(\hat{V}_{1:L}^\top\hat{\Sigma}_2\hat{V}_{1:L}\right) \ge \tilde{c}\lambda_{L},
\end{equation}
for some constant $\tilde{c}>0$. The operator norm of the inverse is
\begin{equation}
\left\|\left(\hat{V}_{1:L}^\top\hat{\Sigma}_2\hat{V}_{1:L}\right)^{-1}\right\|_{\mathrm{op}} = \frac{1}{\lambda_{\min}\left(\hat{M}\right)} \le \frac{1}{\tilde{c}\lambda_{L}} = O\left(\lambda_{L}^{-1}\right).
\end{equation}
This completes the proof.
\end{proof}

\bibliographystyle{alpha}
\bibliography{main}

\newcommand{\etalchar}[1]{$^{#1}$}
\begin{thebibliography}{vdPKvdV14}

\bibitem[ADL13a]{armagan2013generalized}
Artin Armagan, David~B Dunson, and Jaeyong Lee.
\newblock Generalized double pareto shrinkage.
\newblock {\em Statistica Sinica}, 23(1):119, 2013.

\bibitem[ADL{\etalchar{+}}13b]{armagan2013posterior}
Artin Armagan, David~B Dunson, Jaeyong Lee, Waheed~U Bajwa, and Nate Strawn.
\newblock Posterior consistency in linear models under shrinkage priors.
\newblock {\em Biometrika}, 100(4):1011--1018, 2013.

\bibitem[Alq13]{alquier2013bayesian}
Pierre Alquier.
\newblock {Bayesian Methods for Low-Rank Matrix Estimation: Short Survey and Theoretical Study}.
\newblock In {\em Algorithmic Learning Theory}, pages 309--323. Springer, 2013.

\bibitem[BB06]{bagnoli2006log}
Mark Bagnoli and Ted Bergstrom.
\newblock Log-concave probability and its applications.
\newblock In {\em Rationality and Equilibrium: A Symposium in Honor of Marcel K. Richter}, pages 217--241. Springer, 2006.

\bibitem[BCG21]{banerjee2021bayesian}
Sayantan Banerjee, Isma{\"e}l Castillo, and Subhashis Ghosal.
\newblock Bayesian inference in high-dimensional models.
\newblock {\em arXiv preprint arXiv:2101.04491}, 2021.

\bibitem[BDPW17]{bhadra2017horseshoe+}
Anindya Bhadra, Jyotishka Datta, Nicholas~G Polson, and Brandon Willard.
\newblock The horseshoe+ estimator of ultra-sparse signals.
\newblock {\em Bayesian Analysis}, 12(4):1105--1131, 2017.

\bibitem[BG10]{brown2010inference}
Philip~J Brown and Jim~E Griffin.
\newblock Inference with normal-gamma prior distributions in regression problems.
\newblock {\em Bayesian analysis}, 5(1):171--188, 2010.

\bibitem[BG18]{bai2018high}
Ray Bai and Malay Ghosh.
\newblock High-dimensional multivariate posterior consistency under global--local shrinkage priors.
\newblock {\em Journal of Multivariate Analysis}, 167:157--170, 2018.

\bibitem[BG20]{belitser2020empirical}
Eduard Belitser and Subhashis Ghosal.
\newblock Empirical bayes oracle uncertainty quantification for regression.
\newblock {\em The Annals of Statistics}, 48(6):3113--3137, 2020.

\bibitem[BLLT20]{bartlett2020benign}
Peter~L Bartlett, Philip~M Long, G{\'a}bor Lugosi, and Alexander Tsigler.
\newblock Benign overfitting in linear regression.
\newblock {\em Proceedings of the National Academy of Sciences}, 117(48):30063--30070, 2020.

\bibitem[BN20]{belitser2020needles}
Eduard Belitser and Nurzhan Nurushev.
\newblock Needles and straw in a haystack: robust confidence for possibly sparse sequences.
\newblock {\em Bernoulli}, 26(1):191--225, 2020.

\bibitem[Bon11]{bontemps2011bernstein}
Dominique Bontemps.
\newblock {Bernstein von Mises Theorems for Gaussian Regression with increasing number of regressors}.
\newblock {\em The Annals of Statistics}, 39(5):2557--2584, 2011.

\bibitem[BPPD15]{bhattacharya2015dirichlet}
Anirban Bhattacharya, Debdeep Pati, Natesh~S Pillai, and David~B Dunson.
\newblock Dirichlet--laplace priors for optimal shrinkage.
\newblock {\em Journal of the American Statistical Association}, 110(512):1479--1490, 2015.

\bibitem[Cas24]{castillo2024bayesian}
Isma{\"e}l Castillo.
\newblock {\em {Bayesian Nonparametric Statistics}}.
\newblock Lecture Notes in Mathematics. Springer, 1 edition, 2024.

\bibitem[CCD{\etalchar{+}}18]{chernozhukov2018double}
Victor Chernozhukov, Denis Chetverikov, Mert Demirer, Esther Duflo, Christian Hansen, Whitney Newey, and James Robins.
\newblock Double/debiased machine learning for treatment and structural parameters.
\newblock {\em The Econometrics Journal}, 21(1):C1--C68, 2018.

\bibitem[CFSK20]{cadonna2020triple}
Annalisa Cadonna, Sylvia Fr{\"u}hwirth-Schnatter, and Peter Knaus.
\newblock Triple the gamma—a unifying shrinkage prior for variance and variable selection in sparse state space and tvp models.
\newblock {\em Econometrics}, 8(2):20, 2020.

\bibitem[CL21]{chatterji2021finite}
Niladri~S Chatterji and Philip~M Long.
\newblock Finite-sample analysis of interpolating linear classifiers in the overparameterized regime.
\newblock {\em The Journal of Machine Learning Research}, 22(1):5721--5750, 2021.

\bibitem[CL22]{chatterji2022foolish}
Niladri~S Chatterji and Philip~M Long.
\newblock {Foolish Crowds Support Benign Overfitting}.
\newblock {\em Journal of Machine Learning Research}, 23(125):1--12, 2022.

\bibitem[CM24]{cheng2022dimension}
Chen Cheng and Andrea Montanari.
\newblock Dimension free ridge regression.
\newblock {\em Annals of Statistics}, 52(6):2879--2912, 2024.

\bibitem[CPS10]{carvalho2010horseshoe}
Carlos~M Carvalho, Nicholas~G Polson, and James~G Scott.
\newblock The horseshoe estimator for sparse signals.
\newblock {\em Biometrika}, 97(2):465--480, 2010.

\bibitem[CSHvdV15]{castillo2015bayesian}
Isma{\"e}l Castillo, Johannes Schmidt-Hieber, and Aad van~der Vaart.
\newblock Bayesian linear regression with sparse priors.
\newblock {\em The Annals of Statistics}, 43(5):1986--2018, 2015.

\bibitem[CvdV12]{castillo2012needles}
Isma{\"e}l Castillo and Aad van~der Vaart.
\newblock Needles and straw in a haystack: Posterior concentration for possibly sparse sequences.
\newblock {\em The Annals of Statistics}, 40(4):2069--2101, 2012.

\bibitem[Dum06]{DUMER20061667}
Ilya Dumer.
\newblock Covering an ellipsoid with equal balls.
\newblock {\em Journal of Combinatorial Theory, Series A}, 113(8):1667--1676, 2006.

\bibitem[DW18]{dobriban2018high}
Edgar Dobriban and Stefan Wager.
\newblock High-dimensional asymptotics of prediction: Ridge regression and classification.
\newblock {\em The Annals of Statistics}, 46(1):247--279, 2018.

\bibitem[GCS{\etalchar{+}}13]{gelman2013bayesian}
A.~Gelman, J.B. Carlin, H.S. Stern, D.B. Dunson, A.~Vehtari, and D.B. Rubin.
\newblock {\em Bayesian Data Analysis}.
\newblock Chapman and Hall/CRC, 3 edition, 2013.

\bibitem[GGvdV00]{ghosal2000convergence}
Subhashis Ghosal, Jayanta~K Ghosh, and Aad~W van~der Vaart.
\newblock Convergence rates of posterior distributions.
\newblock {\em Annals of Statistics}, 28(2):500--531, 2000.

\bibitem[Gho97]{ghosal1997normal}
Subhashis Ghosal.
\newblock Normal approximation to the posterior distribution for generalized linear models with many covariates.
\newblock {\em Mathematical Methods of Statistics}, 6(3):332--348, 1997.

\bibitem[Gho99]{ghosal1999asymptotic}
Subhashis Ghosal.
\newblock Asymptotic normality of posterior distributions in high-dimensional linear models.
\newblock {\em Bernoulli}, 5(2):315--331, 1999.

\bibitem[Gho00]{ghosal2000asymptotic}
Subhashis Ghosal.
\newblock Asymptotic normality of posterior distributions for exponential families when the number of parameters tends to infinity.
\newblock {\em Journal of Multivariate Analysis}, 74(1):49--68, 2000.

\bibitem[GN21]{gine2021mathematical}
Evarist Gin{\'e} and Richard Nickl.
\newblock {\em {Mathematical Foundations of Infinite-Dimensional Statistical Models}}.
\newblock Cambridge University Press, 2021.

\bibitem[GS90]{gelfand1990sampling}
Alan~E. Gelfand and Adrian F.~M. Smith.
\newblock {Sampling-Based Approaches to Calculating Marginal Densities}.
\newblock {\em Journal of the American Statistical Association}, 85(410):398--409, 1990.

\bibitem[GvdV17]{ghosal2017fundamentals}
Subhashis Ghosal and Aad van~der Vaart.
\newblock {\em {Fundamentals of Nonparametric Bayesian Inference}}, volume~44.
\newblock Cambridge University Press, 2017.

\bibitem[HMRT22]{hastie2022surprises}
Trevor Hastie, Andrea Montanari, Saharon Rosset, and Ryan~J Tibshirani.
\newblock Surprises in high-dimensional ridgeless least squares interpolation.
\newblock {\em The Annals of Statistics}, 50(2):949--986, 2022.

\bibitem[KL17]{koltchinskii2017concentration}
Vladimir Koltchinskii and Karim Lounici.
\newblock Concentration inequalities and moment bounds for sample covariance operators.
\newblock {\em Bernoulli}, 23(1):110--133, 2017.

\bibitem[KZSS21]{koehler2021uniform}
Frederic Koehler, Lijia Zhou, Danica~J Sutherland, and Nathan Srebro.
\newblock Uniform convergence of interpolators: Gaussian width, norm bounds, and benign overfitting.
\newblock {\em arXiv preprint arXiv:2106.09276}, 2021.

\bibitem[LBH15]{lecun2015deep}
Yann LeCun, Yoshua Bengio, and Geoffrey Hinton.
\newblock Deep learning.
\newblock {\em Nature}, 521(7553):436--444, 2015.

\bibitem[LC12]{le2012asymptotic}
Lucien Le~Cam.
\newblock {\em Asymptotic methods in statistical decision theory}.
\newblock Springer Science \& Business Media, 2012.

\bibitem[Li99]{Li1999GaussCorIneq}
Wenbo Li.
\newblock {A Gaussian Correlation Inequality and its Applications to Small Ball Probabilities}.
\newblock {\em Electronic Communications in Probability}, 4:111--118, 1999.

\bibitem[Lou17]{loukas2017close}
Andreas Loukas.
\newblock How close are the eigenvectors of the sample and actual covariance matrices?
\newblock In {\em International Conference on Machine Learning}, pages 2228--2237. PMLR, 2017.

\bibitem[LR20]{liang2020just}
Tengyuan Liang and Alexander Rakhlin.
\newblock Just interpolate: Kernel “ridgeless” regression can generalize.
\newblock {\em The Annals of Statistics}, 48(3):1329--1347, 2020.

\bibitem[LRZ20]{liang2020multiple}
Tengyuan Liang, Alexander Rakhlin, and Xiyu Zhai.
\newblock {On the Multiple Descent of Minimum-Norm Interpolants and Restricted Lower Isometry of Kernels}.
\newblock In {\em Conference on Learning Theory}, pages 2683--2711. PMLR, 2020.

\bibitem[LW17]{louizos2017multiplicative}
Christos Louizos and Max Welling.
\newblock {Multiplicative Normalizing Flows for Variational Bayesian Neural Networks}.
\newblock In {\em International Conference on Machine Learning}, pages 2218--2227. PMLR, 2017.

\bibitem[LW21]{li2021minimum}
Yue Li and Yuting Wei.
\newblock Minimum $\ell_1$ -norm interpolators: Precise asymptotics and multiple descent.
\newblock {\em arXiv preprint arXiv:2110.09502}, 2021.

\bibitem[LZG21]{li2021towards}
Zhu Li, Zhi-Hua Zhou, and Arthur Gretton.
\newblock {Towards an Understanding of Benign Overfitting in Neural Networks}.
\newblock {\em arXiv preprint arXiv:2106.03212}, 2021.

\bibitem[MMW17]{martin2017empirical}
Ryan Martin, Raymond Mess, and Stephen~G Walker.
\newblock Empirical bayes posterior concentration in sparse high-dimensional linear models.
\newblock {\em Bernoulli}, 23(3):1822--1847, 2017.

\bibitem[MNS{\etalchar{+}}21]{muthukumar2021classification}
Vidya Muthukumar, Adhyyan Narang, Vignesh Subramanian, Mikhail Belkin, Daniel Hsu, and Anant Sahai.
\newblock Classification vs regression in overparameterized regimes: Does the loss function matter?
\newblock {\em The Journal of Machine Learning Research}, 22(1):10104--10172, 2021.

\bibitem[NI25]{nakakita2022benign}
Shogo Nakakita and Masaaki Imaizumi.
\newblock {Benign Overfitting in Time-Series Linear Model with Over-Parameterization}.
\newblock {\em Bernoulli}, 2025.

\bibitem[NJG20]{ning2020bayesian}
Bo~Ning, Seonghyun Jeong, and Subhashis Ghosal.
\newblock Bayesian linear regression for multivariate responses under group sparsity.
\newblock {\em Bernoulli}, 26(3):2353--2382, 2020.

\bibitem[NR23]{nie2023bayesian}
Lizhen Nie and Veronika Ro{\v{c}}kov{\'a}.
\newblock {Bayesian Bootstrap Spike-and-Slab LASSO}.
\newblock {\em Journal of the American Statistical Association}, 118(543):2013--2028, 2023.

\bibitem[PC08]{park2008bayesian}
Trevor Park and George Casella.
\newblock {The Bayesian Lasso}.
\newblock {\em Journal of the American Statistical Association}, 103(482):681--686, 2008.

\bibitem[Pen24]{pensky2024davis}
Marianna Pensky.
\newblock {Davis-Kahan Theorem in the two-to-infinity norm and its application to perfect clustering}.
\newblock {\em arXiv preprint arXiv:2411.11728}, 2024.

\bibitem[PKH23]{mosaicdata}
Randall Pruim, Daniel Kaplan, and Nicholas Horton.
\newblock {\em mosaicData: Project MOSAIC Data Sets}, 2023.
\newblock R package version 0.20.4.

\bibitem[RG14]{rovckova2014emvs}
Veronika Ro{\v{c}}kov{\'a} and Edward~I George.
\newblock {EMVS: The EM approach to Bayesian variable selection}.
\newblock {\em Journal of the American Statistical Association}, 109(506):828--846, 2014.

\bibitem[SGWH20]{shen2020optimal}
Y~Shen, C~Gao, D~Witten, and F~Han.
\newblock Optimal estimation of variance in nonparametric regression with random design.
\newblock {\em The Annals of Statistics}, 48(6):3589--3618, 2020.

\bibitem[SL22]{song2022nearly}
Qifan Song and Faming Liang.
\newblock Nearly optimal bayesian shrinkage for high-dimensional regression.
\newblock {\em Science China Mathematics}, pages 1--34, 2022.

\bibitem[SLD18]{Srivastava2018scalable}
Sanvesh Srivastava, Cheng Li, and David~B. Dunson.
\newblock {Scalable Bayes via Barycenter in Wasserstein Space}.
\newblock {\em Journal of Machine Learning Research}, 19(8):1--35, 2018.

\bibitem[SS24]{suzuki2024optimal}
Keita Suzuki and Taiji Suzuki.
\newblock Optimal criterion for feature learning of two-layer linear neural network in high dimensional interpolation regime.
\newblock In {\em The Twelfth International Conference on Learning Representations}, 2024.

\bibitem[SvdVvZ13]{szabo2013empirical}
BT~Szab{\'o}, AW~van~der Vaart, and JH~van Zanten.
\newblock {Empirical Bayes scaling of Gaussian priors in the white noise model}.
\newblock {\em Electronic Journal of Statistics}, 7:991--1018, 2013.

\bibitem[TB23]{tsigler2020benign}
Alexander Tsigler and Peter~L Bartlett.
\newblock Benign overfitting in ridge regression.
\newblock {\em Journal of Machine Learning Research}, 24(123):1--76, 2023.

\bibitem[Tem18]{temlyakov2018multivariate}
Vladimir Temlyakov.
\newblock {\em Multivariate approximation}, volume~32.
\newblock Cambridge University Press, 2018.

\bibitem[THAB19]{trippe2019lr}
Brian Trippe, Jonathan Huggins, Raj Agrawal, and Tamara Broderick.
\newblock {LR-GLM: High-dimensional Bayesian inference using low-rank data approximations}.
\newblock In {\em International Conference on Machine Learning}, pages 6315--6324. PMLR, 2019.

\bibitem[TI24]{tsuda2023benign}
Toshiki Tsuda and Masaaki Imaizumi.
\newblock {Benign overfitting of non-sparse high-dimensional linear regression with correlated noise}.
\newblock {\em Electronic Journal of Statistics}, 18(2):4119 -- 4197, 2024.

\bibitem[TW05]{tong2005estimating}
Tiejun Tong and Yuedong Wang.
\newblock Estimating residual variance in nonparametric regression using least squares.
\newblock {\em Biometrika}, 92(4):821--830, 2005.

\bibitem[vdPKvdV14]{van2014horseshoe}
St{\'e}phanie~L van~der Pas, Bas~JK Kleijn, and Aad~W van~der Vaart.
\newblock The horseshoe estimator: Posterior concentration around nearly black vectors.
\newblock {\em Electronic Journal of Statistics}, 8(2):2585--2618, 2014.

\bibitem[Ver18]{vershynin2018high}
Roman Vershynin.
\newblock {\em High-dimensional probability: An introduction with applications in data science}, volume~47.
\newblock Cambridge University Press, 2018.

\bibitem[Wai19]{wainwright2019high}
Martin~J Wainwright.
\newblock {\em High-dimensional statistics: A non-asymptotic viewpoint}, volume~48.
\newblock Cambridge University Press, 2019.

\bibitem[WBCL08]{wang2008effect}
Lie Wang, Lawrence~D Brown, T~Tony Cai, and Michael Levine.
\newblock Effect of mean on variance function estimation in nonparametric regression.
\newblock {\em The Annals of Statistics}, 36(2):646--664, 2008.

\bibitem[WDY21]{wang2021tight}
Guillaume Wang, Konstantin Donhauser, and Fanny Yang.
\newblock Tight bounds for minimum l1-norm interpolation of noisy data.
\newblock {\em arXiv preprint arXiv:2111.05987}, 2021.

\bibitem[WNNY23]{wu2023statistical}
Teng Wu, Naveen N.~Narisetty, and Yun Yang.
\newblock Statistical inference via conditional bayesian posteriors in high-dimensional linear regression.
\newblock {\em Electronic Journal of Statistics}, 17(1):769--797, 2023.

\bibitem[XH19]{Xu2019pca}
Ji~Xu and Daniel~J Hsu.
\newblock On the number of variables to use in principal component regression.
\newblock In {\em Advances in Neural Information Processing Systems}, volume~32, 2019.

\bibitem[Yan19]{yang2019posterior}
Dana Yang.
\newblock Posterior asymptotic normality for an individual coordinate in high-dimensional linear regression.
\newblock {\em Electronic Journal of Statistics}, 13:3082--3094, 2019.

\bibitem[YVSG18]{10.1214/17-BA1091}
Yuling Yao, Aki Vehtari, Daniel Simpson, and Andrew Gelman.
\newblock {Using Stacking to Average Bayesian Predictive Distributions (with Discussion)}.
\newblock {\em Bayesian Analysis}, 13(3):917--1007, 2018.

\bibitem[Zel86]{zellner1986assessing}
Arnold Zellner.
\newblock {On Assessing Prior Distributions and Bayesian Regression Analysis with g-Prior Distributions}.
\newblock In {\em Bayesian Inference and Decision Techniques: Essays in Honor of Bruno de Finetti}, pages 233--243. Elsevier Science Publishers, Inc., 1986.

\bibitem[Zhi24]{zhivotovskiy2021dimension}
Nikita Zhivotovskiy.
\newblock Dimension-free bounds for sums of independent matrices and simple tensors via the variational principle.
\newblock {\em Electronic Journal of Probability}, 29:1--28, 2024.

\bibitem[ZTSG19]{zhao2019adaptive}
Han Zhao, Yao-Hung~Hubert Tsai, Ruslan Salakhutdinov, and Geoffrey~J Gordon.
\newblock {Learning Neural Networks with Adaptive Regularization}.
\newblock In {\em Advances in Neural Information Processing Systems}, 2019.

\bibitem[ZZ14]{zhang2014confidence}
Cun-Hui Zhang and Stephanie~S Zhang.
\newblock {Confidence Intervals for Low Dimensional Parameters in High Dimensional Linear Models}.
\newblock {\em Journal of the Royal Statistical Society: Series B: Statistical Methodology}, pages 217--242, 2014.

\end{thebibliography}

\end{document}